\documentclass[10pt]{article}%

\usepackage[hide]{ed}

\usepackage{rotating} 

\usepackage{mathrsfs}
\DeclareMathAlphabet{\mathpzc}{OT1}{pzc}{m}{it}

\newcommand{\starA}{\mathscr{A}_k}

\usepackage{color}
\usepackage{amsmath}
\usepackage{graphicx}
\usepackage{amsfonts}
\usepackage{amssymb}%
\usepackage{latexsym}
\usepackage{psfrag}
\usepackage{subfigure}
\usepackage{accents}

\newtheorem{theorem}{Theorem}[section]
\newtheorem{corollary}[theorem]{Corollary}
\newtheorem{definition}[theorem]{Definition}
\newenvironment{proof}[1][Proof]{\noindent \emph{#1.} }
{\hfill \ \rule{0.5em}{0.5em}}
\newtheorem{lemma}[theorem]{Lemma}
\newtheorem{proposition}[theorem]{Proposition}

\newtheorem{assumption}[theorem]{Assumption}
\newtheorem{example}[theorem]{Example}
\newtheorem{remark}[theorem]{Remark}

\numberwithin{equation}{section}
\numberwithin{figure}{section}
\numberwithin{table}{section}

\usepackage[a4paper]{geometry}
\newcommand{\noi}{\noindent}

\newcommand{\bzero}{\mathbf{0}}
\newcommand{\R}{\mathbb{R}}
\newcommand{\bT}{\mathbf{T}}

\newcommand{\NN}{\mathbb{N}}

\newcommand{\cH}{{\cal H}}

\newcommand{\cP}{{\cal P}}

\newcommand{\cS}{{\cal S}}

\newcommand{\cK}{{\cal K}}
\newcommand{\cD}{{\cal D}}

\newcommand{\cN}{{\cal N}}

\newcommand{\cZ}{{\cal Z}}
\newcommand{\be}{\mathbf{e}}
\newcommand{\bx}{\mathbf{x}}
\newcommand{\bd}{\mathbf{d}}
\newcommand{\br}{\mathbf{r}}

\newcommand{\bn}{\mathbf{n}}

\newcommand{\bnu}{\boldsymbol{\nu}}

\newcommand{\bt}{\mathbf{t}}
\newcommand{\ba}{\mathbf{a}}
\newcommand{\bbb}{\mathbf{b}}

\newcommand{\bv}{\mathbf{v}}
\newcommand{\bw}{\mathbf{w}}

\newcommand{\by}{\mathbf{y}}

\newcommand{\bF}{\mathbf{F}}
\newcommand{\bQ}{\mathbf{Q}}

\newcommand{\bze}{\mathbf{0}}
\newcommand{\supp}{\mathrm{supp}}
\newcommand{\dist}{\mathrm{dist}}

\newcommand{\bZ}{\mathbf{Z}}

\newcommand{\ri}{{\rm i}}
\newcommand{\rd}{{\rm d}}

\newcommand{\beq}{\begin{equation}}
\newcommand{\eeq}{\end{equation}}
\newcommand{\beqs}{\begin{equation*}}
\newcommand{\eeqs}{\end{equation*}}
\newcommand{\bit}{\begin{itemize}}
\newcommand{\eit}{\end{itemize}}
\newcommand{\ben}{\begin{enumerate}}
\newcommand{\een}{\end{enumerate}}
\newcommand{\bal}{\begin{align}}
\newcommand{\eal}{\end{align}}
\newcommand{\bals}{\begin{align*}}
\newcommand{\eals}{\end{align*}}
\newcommand{\bse}{\begin{subequations}}
\newcommand{\ese}{\end{subequations}}
\newcommand{\bpr}{\begin{proposition}}
\newcommand{\epr}{\end{proposition}}
\newcommand{\bre}{\begin{remark}}
\newcommand{\ere}{\end{remark}}
\newcommand{\bpf}{\begin{proof}}
\newcommand{\epf}{\end{proof}}
\newcommand{\ble}{\begin{lemma}}
\newcommand{\ele}{\end{lemma}}
\newcommand{\bco}{\begin{corollary}}
\newcommand{\eco}{\end{corollary}}
\newcommand{\bex}{\begin{example}}
\newcommand{\eex}{\end{example}}
\newcommand{\bth}{\begin{theorem}}
\newcommand{\enth}{\end{theorem}}

\newcommand{\Rea}{\mathbb{R}}
\newcommand{\Com}{\mathbb{C}}

\newcommand{\esssup}{\mathop{{\rm ess} \sup}}
\newcommand{\essinf}{\mathop{{\rm ess} \inf}}

\newcommand{\diam}{\mathop{{\rm diam}}}

\newcommand{\Oi}{{\Omega^-}}

\newcommand{\Oe}{{\Omega^+}}

\newcommand{\domaingen}{D}
\newcommand{\domain}{D}

\newcommand{\pD}{{\partial D}}

\newcommand{\GR}{{\Gamma_R}}

\newcommand{\eps}{\varepsilon}

\newcommand{\pdiff}[2]{\frac{\partial #1}{\partial #2}}

\newcommand{\dnu}{\partial_n u}

\newcommand{\dnuv}{\partial_{\nu} v}
\newcommand{\dnpu}{\partial_n^+ u}
\newcommand{\dnmu}{\partial_n^- u}
\newcommand{\dnpv}{\partial_n^+ v}
\newcommand{\dnmv}{\partial_n^- v}

\newcommand{\gmu}{\gamma^- u}
\newcommand{\gpu}{\gamma^+ u}
\newcommand{\gmv}{\gamma^- v}
\newcommand{\gpv}{\gamma^+ v}

\newcommand{\ngus}{|\nabla u|^2}

\newcommand{\gu}{\nabla u}

\newcommand{\ngvs}{|\nabla v|^2}

\newcommand{\gv}{\nabla v}

\newcommand{\nT}{\nabla_{\Gamma}}

\newcommand{\ntS}{\nabla_{\Gamma} S}

\newcommand{\half}{\frac{1}{2}}

\newcommand{\LtG}{{L^2(\Gamma)}}

\newcommand{\LtGt}{{\LtG\rightarrow \LtG}}

\newcommand{\HhG}{{H^{1/2}(\Gamma)}}
\newcommand{\HmhG}{{H^{-1/2}(\Gamma)}}

\newcommand{\HoG}{H^1(\Gamma)}

\newcommand{\Holoc}{H^1_{{\rm loc}}}

\newcommand{\tendi}{\rightarrow \infty}
\newcommand{\tendo}{\rightarrow 0}

\newcommand{\opA}{A'_{k,\eta}}

\newcommand{\nxy}{|\bx-\by|}

\def\XXint#1#2#3{{\setbox0=\hbox{$#1{#2#3}{\int}$}
     \vcenter{\hbox{$#2#3$}}\kern-.5\wd0}}

\usepackage{hyperref}
\definecolor{myblue}{rgb}{0,0,0.6}
\hypersetup{colorlinks=true,
linkcolor=myblue,citecolor=myblue,filecolor=myblue,urlcolor=myblue}
\newcommand*{\N}[1]{\left\|#1\right\|}

\allowdisplaybreaks[4]

\newcommand{\tfa}{\text{ for all }}
\newcommand{\tfor}{\text{ for }}

\newcommand{\tas}{\text{ as }}
\newcommand{\tand}{\text{ and }}
\newcommand{\tst}{\text{ such that }}
\newcommand{\tfind}{\text{ find }}

\newcommand{\Hilb}{\cH}

\newcommand{\PhiH}{\Phi_k}

\newcommand{\vertiii}[1]{{\left\vert\kern-0.25ex\left\vert\kern-0.25ex\left\vert #1
    \right\vert\kern-0.25ex\right\vert\kern-0.25ex\right\vert}}

\newcommand{\DOmegabar}{\cD(\baro)}
\newcommand{\baro}{\overline{\Omega}}

\newcommand{\opG}{A^\dag_{\vfb,\alpha}}
\newcommand{\opLIZ}{A'_{I,\vfb,\alpha}}
\newcommand{\opLIZi}{A_{I,\vfb,\alpha}}
\newcommand{\opLIZRHS}{B_{I,\vfb,\alpha}}
\newcommand{\opLEZ}{A'_{E,\vfb,\alpha}}
\newcommand{\opLEZi}{A_{E,\vfb,\alpha}}
\newcommand{\opLEZRHS}{B_{E,\vfb,\alpha}}

\newcommand{\opLIx}{A'_{I,\bx,\alpha}}
\newcommand{\opLIxi}{A_{I,\bx,\alpha}}
\newcommand{\opLEx}{A'_{E,\bx,\alpha}}
\newcommand{\opLExi}{A_{E,\bx,\alpha}}

\newcommand{\newangle}{\theta}

\newcommand{\vfb}{\bZ}
\newcommand{\vfext}{\bZ_{\rm ext}}
\newcommand{\vfextt}{\widetilde{\bZ}_{\rm ext}}
\newcommand{\vfd}{\widetilde{\bZ}}

\newcommand{\radius}{\kappa}

\newcommand{\Caldp}{\mathcal{K}_{\vfb}}
\newcommand{\Caldb}{K_{\vfb}}

\newcommand{\mymatrix}[1]{\mathsf{#1}}
\newcommand{\MA}{{\mymatrix{A}}}
\newcommand{\MM}{{\mymatrix{M}}}
\newcommand{\MC}{{\mymatrix{C}}}
\newcommand{\MD}{{\mymatrix{D}}}
\newcommand{\MB}{{\mymatrix{B}}}
\newcommand{\MI}{{\mymatrix{I}}}

\DeclareMathOperator{\cond}{cond}
\newcommand{\projP}{P_\Gamma}
\newcommand{\projQ}{Q_\Gamma}

\newcommand{\Ccoer}{C_{\rm coer}}
\newcommand{\mythmname}[1]{\textbf{\emph{(#1.)}}}

\usepackage{soul}

\definecolor{amcol}{rgb}{0.8,0,0}

\definecolor{escol}{rgb}{0,0,0.8}
\definecolor{estcol}{rgb}{0,0.6,0}
\newcommand{\es}[1]{{\color{escol}{#1}}}

\newcommand{\esnote}[1]{\ednote{\es{ES: #1}}}

\definecolor{cwcol}{rgb}{0.5,0,0.5}
\definecolor{cwstcol}{rgb}{0,0.6,0.6}

\usepackage{bigfoot}

\begin{document}

\title{Coercive second-kind boundary integral equations for the Laplace Dirichlet problem on Lipschitz domains
}

\author{S. N.~Chandler-Wilde\thanks{Department of Mathematics and Statistics, University of Reading,
Whiteknights, PO Box 220, Reading, RG6 6AX, UK, \tt S.N.Chandler-Wilde@reading.ac.uk}
\,\,, E. A. Spence\thanks{Department of Mathematical Sciences, University of Bath, Bath, BA2 7AY, UK, \tt E.A.Spence@bath.ac.uk }
}

\date{\today}



\maketitle

\begin{abstract}
We present new second-kind integral-equation formulations of the interior and exterior Dirichlet problems for Laplace's equation.
 The operators in these formulations are   both continuous and coercive on general Lipschitz domains in $\Rea^d$, $d\geq 2$, in the space $L^2(\Gamma)$, where $\Gamma$ denotes the boundary of the domain.
These properties of continuity and coercivity immediately imply that (i) the Galerkin method converges when applied to these formulations; and (ii) the Galerkin matrices are well-conditioned as the discretisation is refined, without the need for operator preconditioning (and we prove a corresponding result about the convergence of GMRES). 
 The main significance of these results is that it was recently proved (see Chandler-Wilde and Spence, Numer.~Math., 150(2):299-271, 2022) that there exist 2- and 3-d Lipschitz domains and 3-d star-shaped Lipschitz polyhedra for which the operators in the standard second-kind integral-equation formulations for Laplace's equation (involving the double-layer potential and its adjoint) \emph{cannot} be written as the sum of a coercive operator and a compact operator in the space $L^2(\Gamma)$. Therefore  there exist 2- and 3-d Lipschitz domains and 3-d star-shaped Lipschitz polyhedra for which Galerkin methods  in $\LtG$ do  \emph{not} converge when applied to the standard second-kind formulations, but \emph{do} converge for the new formulations.
\end{abstract}

\section{Introduction}\label{sec:intro}

\subsection{Boundary integral equations for Laplace's equation}
If an explicit expression for the fundamental solution of a linear PDE is known, then boundary value problems (BVPs) for that PDE can be converted to integral equations on the boundary of the domain.
The main advantage of this procedure is that the dimension of the problem is reduced; indeed, the problem is converted from one on a $d$-dimensional domain to one on a $(d-1)$-dimensional domain. Futhermore, if the original domain is the exterior of a bounded obstacle, then the problem is reduced from one on a $d$-dimensional \emph{infinite} domain, to one on a $(d-1)$-dimensional \emph{finite} domain.

This reduction to the boundary has both theoretical and practical benefits:~on the theoretical side, C.~Neumann famously used boundary integral equations (BIEs) to prove existence of the solution of the Dirichlet problem for Laplace's equation in convex domains in \cite{Ne:77} (see, e.g., the account in \cite[Chapter 1]{Mc:00}), and BIEs have a long history of use in the harmonic analysis literature to prove wellposedness of BVPs on rough domains (see, e.g., \cite{CoMcMe:82}, \cite{Ve:84}, \cite{Ca:85}, \cite[\S2.1]{Ke:94}, \cite{MitreaTaylor:99}, \cite[Chapter 15]{MeCo:00}, \cite[Chapter 4]{Ta:00}, \cite{MiMiTa:01}).
On the more practical side, numerical methods based on Galerkin, collocation, or numerical quadrature discretisation of BIEs, coupled with fast matrix-vector multiply and compression algorithms, and iterative solvers such as GMRES, provide spectacularly effective computational tools for solving a range of linear boundary value problems, for example in potential theory, elasticity, and acoustic and electromagnetic wave scattering (see, e.g., \cite{Rokhlin83,At:97, LaSc:99,BrKu:01,ContETAL02, ChSoCuVeHa:04,BoSa:04,XiTaWe:08,GrGuMaRo09,SaSc:11, ChDaLo:17}).

Let $\Phi(\bx,\by)$ be the fundamental solution for Laplace's equation:
\beq\label{eq:fund}
\Phi(\bx,\by):=
\displaystyle{\frac{1}{2\pi} \log \left(\frac{a}{\nxy}\right),}  \quad d= 2,
\qquad :=
\dfrac{1}{(d-2)C_d|\bx-\by|^{d-2}},  \quad d\geq 3,
\eeq
where $C_d$ is the surface area of the unit sphere $S^{d-1}\subset \Rea^d$ and  $a >0$.
With $\Gamma$ the boundary of a bounded Lipschitz domain,
the boundary integral operators (BIOs) $S$, $D$, $D'$, and $H$,
the \emph{single-layer}, \emph{double-layer}, \emph{adjoint double-layer}, and \emph{hypersingular} operators, respectively,
 are defined for $\phi\in\LtG$, $\psi\in\HoG$, and $\bx\in\Gamma$ by
\beq\label{eq:bio1}
S_k \phi(\bx) = \int_\Gamma \PhiH(\bx,\by) \phi(\by)\, \rd s(\by), \quad
D \phi(\bx) = \int_\Gamma \pdiff{\Phi(\bx,\by)}{n(\by)} \phi(\by)\, \rd s(\by), \quad
\eeq
and
\beq\label{eq:bio2}
D' \phi(\bx) = \int_\Gamma \pdiff{\Phi(\bx,\by)}{n(\bx)} \phi(\by)\, \rd s(\by), \quad
H \psi(\bx) = \pdiff{}{n(\bx)}\int_\Gamma \pdiff{\PhiH(\bx,\by)}{n(\by)} \psi(\by)\, \rd s(\by).
\eeq
When $\Gamma$ is Lipschitz, the integrals in $D$ and $D'$ are defined as Cauchy principal values, in general only for almost all $\bx\in \Gamma$ with respect to the surface measure $\rd s$.
The definition of $H$ on spaces larger than $H^1(\Gamma)$ is complicated (it must be understood either as a finite-part integral, or as the non-tangential limit of a potential; see \cite[Chapter 7]{Mc:00}, \cite[Page 113]{ChGrLaSp:12} respectively), but these details are not essential to the present paper. The standard mapping properties of $S, D, D'$, and $H$ on Sobolev spaces on $\Gamma$ are recalled in Appendix \ref{app:B} (see \eqref{eq:map}).

The BIE operators involved in the standard first- and second-kind BIEs for the Dirichlet and Neumann problems for  Laplace's equation are shown in Table \ref{tab:bies}; although we do not explicitly consider the Neumann problem in this paper, we use the information in this table in what follows.
For the details of the right-hand sides and unknowns for the integral equations corresponding to the operators in Table \ref{tab:bies}, see, e.g., \cite[\S3.4]{SaSc:11}, \cite[Chapter 7]{Mc:00}, \cite[Chapter 7]{St:08}, \cite[\S2.5]{ChGrLaSp:12}.
Recall that the adjective ``direct" in the table refers to equations where the unknown is either the Dirichlet or Neumann trace of the solution to the corresponding BVP, and the adjective ``indirect" refers to equations where the unknown does not have immediate physical relevance.

\begin{table}
\begin{center}
\begin{tabular}{|c|c|c|c|c|}
\hline
&Interior Dirichlet   & Interior Neumann & Exterior Dirichlet  & Exterior Neumann  \\
& problem & problem & problem & problem \\
\hline
 Direct  & $S$ & $H$ &  $S$ & $H$\\
 & $\displaystyle{ \frac{1}{2}I- D^\prime}$   & $\displaystyle{ \frac{1}{2}I+ D}$  &  $\displaystyle{ \frac{1}{2}I+ D'}$ & $\displaystyle{ \frac{1}{2}I- D}$\\
 \hline
 Indirect  & $S$ & $H$ &  $S$ & $H$\\
& $\displaystyle{ \frac{1}{2}I- D}$ & $\displaystyle{ \frac{1}{2}I+ D'}$& $\displaystyle{ \frac{1}{2}I+ D}$ & $\displaystyle{ \frac{1}{2}I- D'}$\\
\hline
     \end{tabular}
     \end{center}
     \caption{
The integral operators involved in the standard boundary-integral-equation formulations of the interior and exterior Dirichlet and Neumann problems for Laplace's equation.\label{tab:bies}}
\end{table}

Following \cite[Pages 9 and 10]{SaSc:11},
we call BIEs \emph{first kind} where the unknown function only appears under the integral, and \emph{second kind} where the unknown function appears outside the integrand as well as inside; by this definition, the BIEs in the first and third row of Table \ref{tab:bies} are first kind, and in the second and fourth row second kind.
An alternative definition of second kind BIEs is that, in addition to the unknown function appearing outside the integrand as well as inside, the BIO is Fredholm of index zero (i.e., the Fredholm alternative applies to the BIE); see, e.g., \cite[\S1.1.4]{At:97}.
Every BIE that we describe in the paper as second-kind is second-kind in both senses above.

\subsection{The Galerkin method} \label{sec:Gal}

We focus on solving Laplace BIEs with the \emph{Galerkin method}. The Galerkin method  applied to the equation $A\phi = f$, where $\phi, f \in \cH$, $A:\cH\rightarrow \cH$ is a continuous (i.e.~bounded) linear operator, and $\cH$ is a complex\footnote{It is convenient, since we deal with non-self-adjoint operators and talk at some points about spectra and numerical ranges, to assume throughout that all Hilbert spaces and function spaces are complex. Of course results for the corresponding real case are easily deduced, if needed, from the complex function space case.}
Hilbert space, is:~given a sequence $(\cH_N)_{N=1}^\infty$ of finite-dimensional subspaces of $\cH$ with $\dim(\cH_N)\tendi$ as $N\tendi$,
\beq\label{eq:G}
\tfind \phi_N \in\cH_N \tst \big(A\phi_N,\psi_N)_\cH=\big(f,\psi_N\big)_\cH \quad \tfa \psi_N\in \cH_N.
\eeq
We say that the \emph{Galerkin method converges for the sequence $(\cH_N)_{N=1}^\infty$} if, for every $f\in \cH$, the Galerkin equations \eqref{eq:G} have a unique solution for all sufficiently large $N$ and $\phi_N\to A^{-1}f$ as $N\to\infty$.
We say that \emph{$(\cH_N)_{N=1}^\infty$ is asymptotically dense in $\cH$} if, for every $\phi\in \cH$,
\beq\label{eq:space_converge}
\min_{\psi_N\in \cH_N}\N{\phi-\psi_N}_{\cH} \to 0 \quad \mbox{as} \quad N\to \infty.
\eeq

A necessary condition for the convergence of the Galerkin method is that $(\cH_N)_{N=1}^\infty$ is asymptotically dense in $\cH$.
Indeed, a standard necessary and sufficient condition  (e.g., \cite[Chapter II, Theorem 2.1]{GoFe:74}) for convergence of the Galerkin method is that  $(\cH_N)_{N=1}^\infty$ is asymptotically dense and that, for some $N_0\in \NN$ and $C_{\mathrm{dis}}>0$,
\begin{equation} \label{eq:d_infsup}
 \frac{\|\cP_N A\psi_N\|_{\cH}}{\|\psi_N\|_\cH} \geq C_{\mathrm{dis}} \quad \mbox{for all non-zero }\psi_N\in \cH_N \mbox{ and } N\geq N_0,
\end{equation}
where $\cP_N$ is orthogonal projection of $\cH$ onto $\cH_N$. Importantly, if \eqref{eq:d_infsup} holds, then (\cite[Chapter II, Equation (2.5)]{GoFe:74} or see \cite[Theorem 4.2.1 and Remark 4.2.5]{SaSc:11})
\begin{equation} \label{eq:quasioptimal2}
\|\phi- \phi_N\|_{\cH} \leq \left(1+ \frac{\|A\|_{\cH\to\cH}}{C_{\mathrm{dis}}}\right)\min_{\psi_N\in \cH_N}\|\phi-\psi_N\|_\cH, \quad \mbox{for } N\geq N_0,
\end{equation}
where $\phi=A^{-1}f$ and $\phi_N$ is the unique solution of the Galerkin equations \eqref{eq:G}.
We note that \eqref{eq:quasioptimal2} is known as a \emph{quasioptimal} error estimate.

We now recap the main abstract theorem on convergence of the Galerkin method; this theorem uses the definition that
an operator $A:\cH \rightarrow \cH$ is \emph{coercive}\footnote{In the literature, the property \eqref{eq:coer} (and its analogue for operators $A: \cH\rightarrow \cH'$, where $\cH'$ is the dual of $\cH$) is sometimes called ``$\Hilb$-ellipticity" (as in, e.g., \cite[Page 39]{SaSc:11}, \cite[\S3.2]{St:08}, and \cite[Definition 5.2.2]{HsWe:08}) or ``strict coercivity" (e.g., \cite[Definition 13.22]{Kr:89}),  with ``coercivity" then used to mean \emph{either} that $A$ is the sum of a coercive operator and a compact operator (as in, e.g., \cite[\S3.6]{St:08} and \cite[\S5.2]{HsWe:08}) \emph{or} that $A$ satisfies a G\aa rding inequality (as in \cite[Definition 2.1.54]{SaSc:11}).} if there exists $\Ccoer>0$ such that
\beq\label{eq:coer}
\big|(A\psi,\psi)_{\cH}\big|\geq \Ccoer \N{\psi}^2_{\cH} \quad \tfa \psi \in \cH.
\eeq

\begin{theorem}[The main abstract theorem on convergence of the Galerkin method.]\label{thm:Galerkin}

\

\begin{enumerate}
\item[(a)]
If $A$ is invertible then there exists a sequence $(\cH_N)_{N=1}^\infty$ for which the Galerkin method converges.
\item[(b)]
If $A$ is invertible then the following are equivalent:
\begin{enumerate}
\item[(i)]
The Galerkin method converges for every asymptotically-dense sequence $(\cH_N)_{N=1}^\infty$ in $\cH$.
\item[(ii)] $A=A_0+K$ where $A_0$ is coercive and $K$ is compact.
\end{enumerate}
\item[(c)]
 If $A$ is coercive (i.e.~\eqref{eq:coer} holds) then, for every sequence $(\cH_N)_{N=1}^\infty$ and every $N\in \NN$, the Galerkin equations \eqref{eq:G} have a unique solution $\phi_N$ and, where $\phi = A^{-1}f$,
\beq\label{eq:quasioptimal}
\big\|\phi-\phi_N\big\|_\cH \leq \frac{\|A\|_{\cH\rightarrow\cH}}{\Ccoer} \,  \min_{\psi \in \cH_N}\big\|\phi-\psi\big\|_\cH,
\eeq
(so that $\phi_N\to \phi$ as $N\to\infty$ if $(\cH_N)_{N=1}^\infty$ is asymptotically dense in $\cH$).
\end{enumerate}
\end{theorem}

\bpf[References for the proof]
Part (a) was first proved in \cite[Theorem 1]{Ma:74}; see also \cite[Chapter II, Theorem 4.1]{GoFe:74}.
Part (b) was first proved in \cite[Theorem 2]{Ma:74}, with this result building on results in \cite{Va:65}; see also \cite[Chapter II, Lemma 5.1 and Theorem 5.1]{GoFe:74}.
Part (c) is C\'ea's Lemma, first proved in \cite{Ce:64}.
\epf

\subsection{The rationale for using second-kind BIEs posed in $\LtG$}\label{sec:rationale}

The BIOs in Table \ref{tab:bies} are coercive in the trace spaces $H^{\pm 1/2}(\Gamma)$ (or certain subspaces of these) for Lipschitz $\Gamma$, thus insuring convergence of the associated Galerkin methods by Part (c) of Theorem \ref{thm:Galerkin}; this coercivity theory was established for first-kind equations by N\'ed\'elec and Planchard \cite{NePl:73}, Le Roux \cite{LeRo:74}, \cite{LeRo:77}, and Hsiao and Wendland \cite{HsWe:77},
and for second-kind equations by Steinbach and Wendland \cite{StWe:01}. These arguments involve transferring boundedness/coercivity properties of the PDE solution operator to the associated boundary integral operators
via the trace map and layer potentials; the generality of these arguments is why coercivity holds with $\Gamma$ only assumed to be Lipschitz, and Costabel \cite{Co:07} highlighted how these ideas can be traced back to the work of Gauss and Poincar\'e.

Despite convergence of the associated Galerkin methods, using the first-kind formulations in the trace spaces has the disadvantage that
the condition numbers of the Galerkin matrices grow as the discretisation is refined; e.g., for the $h$-version of the Galerkin method (where convergence is obtained by decreasing the mesh-width $h$ and keeping the polynomial degree fixed), the condition numbers grow like $h^{-1}$; see, e.g. \cite[\S4.5]{SaSc:11}.
The design of appropriate preconditioning strategies for these Galerkin matrices has therefore been a classic topic of study in the BIE community for over 20 years, with proposed solutions including
(i) preconditioning with an opposite-order operator \cite{StWe:98} (see also the survey \cite{Hi:06}), (ii)
using wavelets, either as an approximation space (e.g., \cite{voSc:96,HaSc:04,HaUt:18}) or in preconditioning (e.g., \cite{TrStZa:98,ScLaSc:03});
using domain decomposition methods; see, e.g., \cite{HeSt:03} and the recent book \cite{StTr:21}.
Furthermore, using the second-kind formulations in the trace spaces has the disadvantage that the inner products on $H^{\pm 1/2}(\Gamma)$ are non-local and non-trivial to evaluate; even if the basis functions $\phi_N$ and $\psi_N$ in \eqref{eq:G} have supports only on a subset of $\Gamma$, $(A\phi_N,\psi_N)_{\cH}$ is an integral over all of $\Gamma$, and the calculation of the Galerkin matrix in this case is impractical.

For the second-kind BIEs, an attractive alternative to working in the trace spaces is to work in $\LtG$.
When $\Gamma$ is $C^1$, $D$ and $D'$ are compact in $\LtG$ by the results of Fabes, Jodeit, and Rivi\`ere \cite[Theorems 1.2 and 1.9]{FaJoRi:78} and thus each of the second-kind BIOs $\half I \pm D$ and $\half I \pm D'$ is the sum of
a coercive operator and a compact operator, and convergence of the associated Galerkin methods in $\LtG$ is ensured by Part (b) of Theorem \ref{thm:Galerkin}. Since the $\LtG$ norm is local, $(A\phi_N,\psi_N)_{\cH}$ is an integral over the support of $\psi_N$, and the Galerkin matrix is much more easily computable. Furthermore, when $D$ and $D'$ are compact, the condition numbers of the Galerkin matrices of  $\half I \pm D$ and $\half I \pm  D'$ are independent of the discretisation (without preconditioning); see \cite[\S3.6.3]{At:97}, \cite[\S4.5.5]{Ha:95}.

\subsection{Convergence of the Galerkin method in $L^2(\Gamma)$ for the standard second-kind integral equations
 on polyhedral and Lipschitz domains.
}\label{sec:openbook}

The disadvantage of using second-kind BIEs in $\LtG$ is that convergence of the Galerkin method is harder to establish when
$\Gamma$ is only Lipschitz, or Lipschitz polyhedral. Indeed, in these cases $D$ and $D'$ are not compact; e.g., when $\Gamma$ has a corner or edge their spectra are not discrete; see, e.g. \cite[\S8.1.3]{At:97}.
When $\Gamma$ is only Lipschitz, $D$ and $D'$  are bounded on $L^2(\Gamma)$ by  the results on boundedness of the Cauchy integral on Lipschitz $\Gamma$ of Coifman, McIntosh, and Meyer \cite{CoMcMe:82} (following earlier work by Calder\'on \cite{Ca:77} on boundedness for $\Gamma$ with small Lipschitz character).
Verchota \cite{Ve:84}
showed that the operators
$\half I \pm D$ and $\half I \pm D'$
are Fredholm of index zero on $L^2(\Gamma)$; when $\Gamma$ is connected, $\half I - D$ and $\half I- D'$ are invertible on $\LtG$ and $\half I + D$ and $\half I+ D'$ invertible on $L^2_0(\Gamma)$, the set of $\phi\in\LtG$ with mean value zero; see \cite[Theorems 3.1 and 3.3(i)]{Ve:84}.
\footnote{The invertibility of $\half I-D'$ on $\LtG$
implies that the bilinear form of the associated least-squares formulation
\beqs
a(\phi,\psi)= \left(\left(\half I - D'\right)\phi, \left(\half I - D'\right)\psi\right)_\LtG
\eeqs
is coercive. This formulation, however, suffers from the same disadvantages as the variational formulation of $\half I- D'$ in $\HmhG$, including that computing the entries of the Galerkin matrix requires computing integrals over all of $\Gamma$, even when the basis functions have support on (small) subsets of $\Gamma$.
}

A long-standing open question has been
\begin{quotation}
\noindent Can $\half I \pm D$ and $\half I \pm  D'$ be written as the sum of a coercive operator and a compact operator in the space $L^2(\Gamma)$ when $\Gamma$ is only assumed to be Lipschitz?
\end{quotation}
By Part (b) of Theorem \ref{thm:Galerkin}, this question is equivalent to the question: does the Galerkin method
applied to  $\half I \pm D$ and $\half I \pm  D'$ in $\LtG$ converge
for every asymptotically-dense sequence of subspaces
 when $\Gamma$ is only assumed to be Lipschitz?

Until recently, this question was answered only in the following two cases, both in the affirmative: (i) $\Gamma$ is a 2d curvilinear polygon with each side $C^{1,\alpha}$ for some $0<\alpha<1$ and with each corner angle in the range $(0,2\pi)$. (ii) $\Gamma$ is Lipschitz, with sufficiently small Lipschitz character. Regarding (i): this result was announced by Shelepov in \cite{Sh:69}, with details of the proof given in \cite{Sh:91}, and with the analogous result for polygons following from the result of Chandler \cite[\S3]{Ch:84}; see, e.g. \cite[Lemma 1.5]{BoMaNiPa:90}.
Regarding (ii): Wendland \cite[\S4.2]{We:09} recognised that the results of I.~Mitrea \cite[Lemma 1, Page 392]{Mi:99} about the essential spectral radius could be adapted to prove this result, with this result proved in full in \cite[Corollary 3.5]{ChSp:22}; for more discussion on both (i) and (ii), see \cite[\S1]{ChSp:22}.

The recent paper \cite{ChSp:22} finally settled the question above negatively by giving examples of 2-d Lipschitz domains and 3-d star-shaped Lipschitz polyhedra for which $\half I \pm D$ and $\half I \pm  D'$ \emph{cannot} be written as the sum of a coercive operator and a compact operator in the space $L^2(\Gamma)$.
The 3-d star-shaped Lipschitz polyhedra are defined in \cite[Definition 5.7]{ChSp:22}, and called the \emph{open-book polyhedra}; see Figure \ref{fig:book} for an example, where we use the notation that $\Omega_{\theta,n}$ is the open-book polyhedron with $n$ pages and opening angle $\theta$.
Given $\epsilon>0$ there exists $\theta_0\in (0,\pi]$ such that the essential numerical range of $D$ in $\LtG$ contains the interval $[-\sqrt{n}/2 + \epsilon, \sqrt{n}/2-\epsilon]$ \cite[Theorem 1.3]{ChSp:22}. By the definition of the essential numerical range (see, e.g., \cite[Equation 2.3]{ChSp:22}), this result implies that if $\theta$ is sufficiently small and $n\geq 2$, then $\half I \pm D$ and $\half I \pm  D'$ \emph{cannot} be written as the sum of a coercive operator and a compact operator in the space $L^2(\Gamma)$ when $\Gamma = \partial \Omega_{\theta,n}$.

\begin{figure}
\includegraphics[width=.5\textwidth]{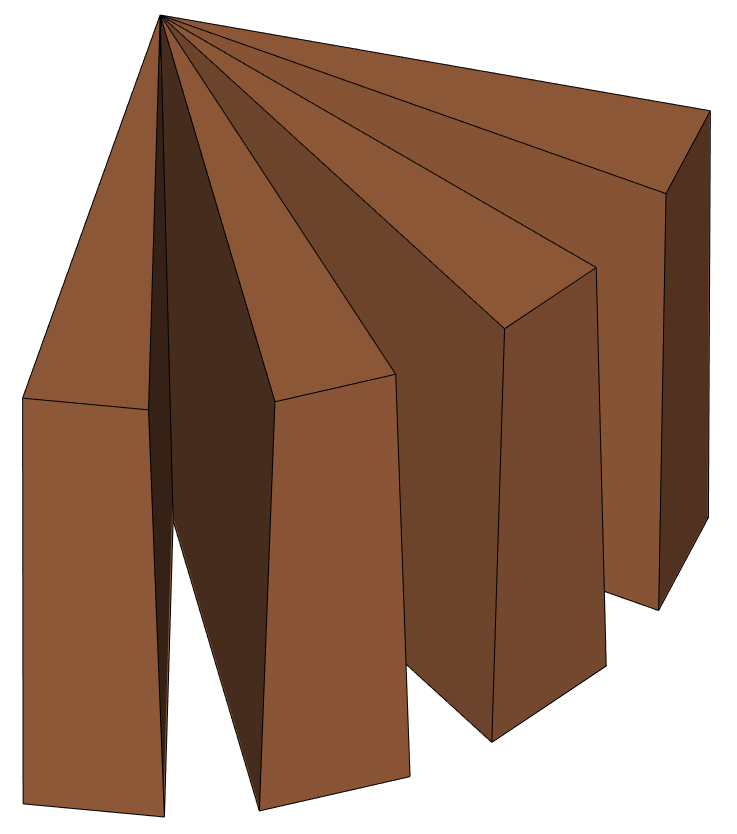}\includegraphics[width=.5\textwidth]{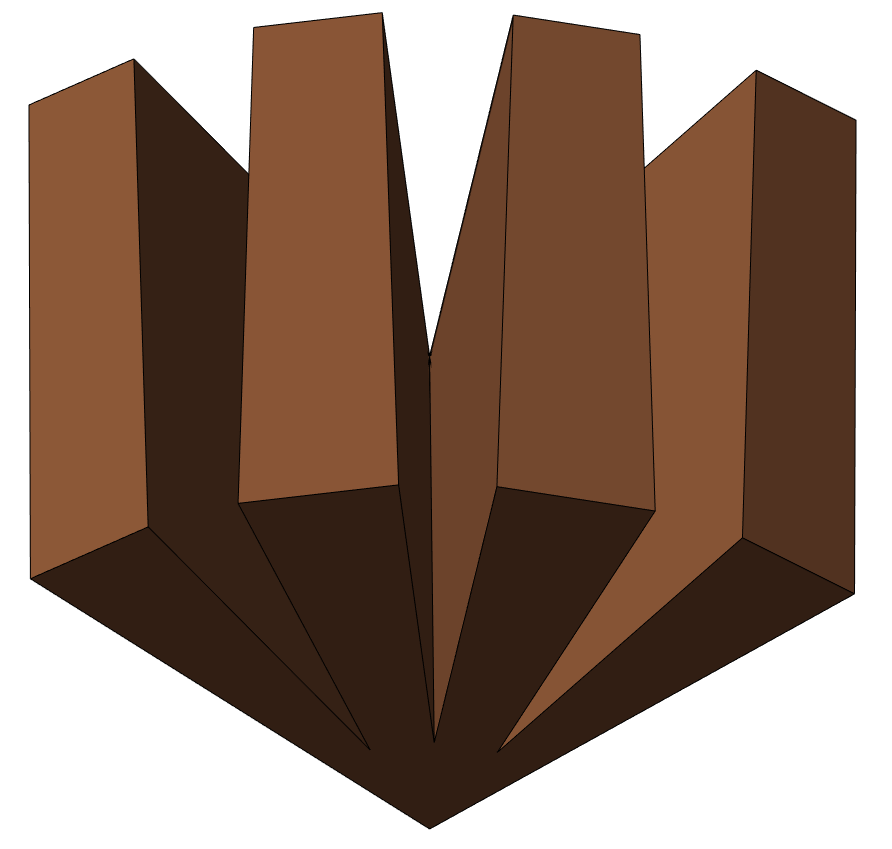}
\caption{\label{fig:book} Views from above and below of the open-book polyhedron $\Omega_{\theta,n}$ of \cite[Definition 5.7]{ChSp:22}, with $n=4$ pages and opening angle $\newangle = \pi/2$}
\end{figure}

Nevertheless, Part (b) of Theorem \ref{thm:Galerkin} only shows that the Galerkin method applied to these domains does not converge for every asymptotically dense sequence $(\cH_N)_{N=1}^\infty \subset \LtG$, leaving opening the possibility that all Galerkin methods used in practice (based on boundary element method discretisation \cite{St:08, SaSc:11}) are in fact convergent. However, the following result from \cite{ChSp:22} clarifies that this is not the case.

\begin{theorem}\textbf{\emph{(\cite[Theorem 1.4]{ChSp:22})}}\label{thm:GalGenNew} Suppose that $A$ is invertible but $A$ cannot be written in the form $A=A_0+K$, where $A_0$ is coercive and $K$ is compact, and that $(\cH^*_N)_{N=1}^\infty$ is a sequence of finite-dimensional subspaces of $\cH$, with $\cH^*_1\subset \cH^*_2 \subset ...$,
for which the Galerkin method converges.
Then there exists a sequence $(\cH_N)_{N=1}^\infty$ of finite-dimensional subspaces of $\cH$, with $\cH_1\subset \cH_2 \subset ...$, such that:
\begin{enumerate}
\item[(a)] the Galerkin method does not converge
for the sequence $(\cH_N)_{N=1}^\infty$; and
\item[(b)] for each $N\in \NN$,
\beq\label{eq:embed}
\cH^*_N \subset \cH_N \subset \cH^*_{M_N},\quad\text{ for some } M_N\in \NN.
\eeq
\end{enumerate}
\end{theorem}

We can apply this result when $(\cH^*_N)_{N=1}^\infty$ is a sequence of boundary element subspaces that is asymptotically dense in $\LtG$, in which case $(\cH_N)_{N=1}^\infty$, satisfying \eqref{eq:embed}, is also a sequence of boundary element subspaces (since $\cH_N\subset \cH_{M_N}^*$) and is also asymptotically dense in $\LtG$ (since $\cH^*_N\subset \cH_N$).

In summary, the results of \cite{ChSp:22} show that there exist Lipschitz and polyhedral boundaries $\Gamma$ for which there are Galerkin methods for solving BIEs involving $\half I \pm D$ and $\half I \pm  D'$
that do \emph{not} converge, with these methods based on asymptotically-dense sequences $(\cH_N)_{N=1}^\infty \subset\LtG$ of boundary element subspaces.

\subsection{Motivation for the present paper and summary of the main results}

Given the negative results of  \cite{ChSp:22} about convergence of the Galerkin method for the standard second-kind formulations, a natural question is therefore
\begin{quotation}
\noindent Do there exist second-kind BIE formulations in $L^2(\Gamma)$ of Laplace's equation where, with $\Gamma$ only assumed to be Lipschitz, the operators are continuous, invertible, and can be written as the sum of a coercive operator and a compact operator?
\end{quotation}
In this paper we answer this question in the affirmative for the Laplace interior and exterior Dirichlet problems. We present new BIE formulations that are continuous and in fact \emph{coercive} (i.e., not just the sum of a coercive and a compact operator) in the space $L^2(\Gamma)$, with $\Gamma$ only assumed to be Lipschitz; thus
convergence of the Galerkin method in $\LtG$
for every asymptotically-dense sequence $(\cH_N)_{N=1}^\infty$, plus the explicit error estimate \eqref{eq:quasioptimal},
is ensured by Part (c) of Theorem \ref{thm:Galerkin}. Furthermore, the strong property of coercivity allows us to prove that, if the Galerkin matrices are preconditioned by a specified diagonal matrix, then
the number of GMRES iterations required to solve the associated linear systems to a prescribed accuracy does \emph{not} increase as the discretisation is refined and $N$ increases.

In summary, the new BIEs introduced in this paper are such that, when solving the Laplace interior and exterior Dirichlet problems on a general Lipschitz domain:
\ben
\item Given any asymptotically-dense sequence of subspaces, the associated Galerkin method is provably convergent; and
\item 
For a wide variety of subspaces, including piecewise polynomials (of arbitrary degree) on anisotropic meshes,
the Galerkin matrices are provably well conditioned -- with the number of GMRES iterations independent of the subspace dimension -- after preconditioning by only a diagonal matrix.
\een 
Indeed, 
\S\ref{sec:openbook} 
recalled that the standard second-kind BIEs in $\LtG$ do not satisfy Point 1. Furthermore, the proposed remedies to the growth of the condition number of the first-kind BIEs in the trace spaces recapped in \S\ref{sec:rationale}, although tremendously successful in many contexts, do not satisfy Point 2. Indeed, to our knowledge, there is no theory on either operator preconditioning or wavelet preconditioning of piecewise-polynomial discretisations using anisotropic 
meshes on general Lipschitz polyhedra. Furthermore, whilst there exists theory for domain-decomposition methods on anisotropic meshes (e.g., \cite{HeSt:03}) the preconditioners are more complicated, and expensive, than multiplication by a diagonal matrix.

\paragraph{Outline of the paper.} \S\ref{sec:notation} defines more precisely the Laplace BVPs we consider. \S\ref{sec:recap} recaps results about a non-standard layer potential introduced in \cite{Ca:85} and its non-tangential limits. \S\ref{sec:main_results} states the main results. \S\ref{sec:idea} discusses the ideas behind the main results, and the links to other work in the literature. \S\ref{sec:proofs} proves the main results, except the parts of the proofs that are related to the wellposedness and regularity of the Laplace oblique Robin problem, with these given in \S\ref{sec:oblique}. \S\ref{sec:Helmholtz} presents results for the Helmholtz exterior Dirichlet problem (with these results corollaries of the Laplace results in \S\ref{sec:main_results}).

\subsection{Notation and statement of the BVPs}\label{sec:notation}

Let $\Oi\subset \Rea^d$, $d\geq 2$, be a bounded (not necessarily connected) Lipschitz open set, and let $\Oe:= \Rea^d\setminus \overline{\Oi}$ and $\Gamma :=\partial \Oi$.
Let $\bn$ be the outward-pointing unit normal vector to $\Oi$ (so $\bn$ points out of $\Oi$ and into $\Oe$).
For $v \in H^1(\Oi)$ let $\gmv$ denote its Dirichlet trace. For $v \in H^1(\Oi,\Delta):= \{ w : w \in H^1(\Oi), \Delta w \in L^2(\Oi)\}$ let $\dnmv$ denote its Neumann trace; recall that, if $v \in H^2(\Oi)$ then $\dnmv = \bn \cdot \gamma^- \nabla v$.
Similarly, for $v\in \Holoc(\Oe):=\{ w: \Oe\rightarrow \Rea : \chi w \in H^1(\Oe) \tfa \chi \in C^\infty_{\rm comp}(\Rea^d)\}$,
let $\gpv$ denote its Dirichlet trace. For
$v\in \Holoc(\Oe, \Delta):=\{ w: \Oe\to \Rea : \chi w \in H^1(\Oe), \chi \Delta w \in L^2(\Oe) \tfa \chi \in C^\infty_{\rm comp}(\Rea^d)\}$, let $\dnpv$ denote its Neumann trace.

\begin{definition}[Laplace interior Dirichlet problem (IDP)]\label{def:idp}
Given $g_D\in H^{1/2}(\Gamma)$, we say that $u\in H^1(\Oi)$ satisfies the {\em interior Dirichlet problem (IDP)} if $\Delta u =0$ in $\Oi$ and $\gmu = g_D$ on $\Gamma$.
\end{definition}

\begin{definition}[Laplace exterior Dirichlet problem (EDP)]\label{def:edp}
With $\Oi$ and $\Oe$ as above, assume further that $\Oe$ is connected. Given $g_D\in H^{1/2}(\Gamma)$, we say that $u\in H^1_{\rm{loc}}(\Oe)$ satisfies the {\em exterior Dirichlet problem (EDP)} if $\Delta u =0$ in $\Oe$, $\gpu = g_D$ on $\Gamma$, and $u(\bx)= O(1)$ when $d=2$ and $u(\bx) = o(|\bx|^{3-d})$ when $d\geq 3$ as $|\bx| \rightarrow \infty$ (uniformly in all directions $\bx/|\bx|$).
\end{definition}

\noi We make three remarks.

(i) Recall that, by elliptic regularity (see, e.g., \cite[Theorem 4.16]{Mc:00}),
the solution of the IDP and EDP are $C^\infty$ in $\Oi$ and $\Oe$ respectively. Therefore, the pointwise conditions at infinity imposed in the EDP
make sense.

(ii) For the IDP and EDP, uniqueness of the solution is shown in, e.g., \cite[Corollary 8.3]{Mc:00} and \cite[Theorems 8.9 and 8.10]{Mc:00} respectively\footnote{
\cite[Theorem 8.10]{Mc:00} proves uniqueness under the condition $u(\bx) = O(|\bx|^{2-d})$ as $|\bx| \rightarrow \infty$, but when $d\geq 3$ the proof still goes through when $u(\bx) = o(|\bx|^{3-d})$ as $|\bx| \rightarrow \infty$.
}.
Existence then follows from Fredholm theory and, e.g., \cite[Theorems 7.5, 7.6, and 7.15]{Mc:00}.

(iii) The Neumann traces of the solutions of both the IDP and EDP are in $\HmhG$; see, e.g., \cite[Lemma 4.3]{Mc:00}. Later, we consider both these BVPs when the Dirichlet data is in $H^1(\Gamma)$. The regularity result of Ne\v{c}as \cite[\S5.1.2]{Ne:67} (restated as Theorem \ref{thm:Necas} below) then implies that $\dnmu$ and $\dnpu$ (in Definitions \ref{def:idp} and \ref{def:edp} respectively) are both in $\LtG$, as opposed to just in $\HmhG$.

\

The IDP and EDP can equivalently be formulated in terms of non-tangential limits, with these alternative formulations standard in the harmonic-analysis literature (see, e.g., \cite[Corollary 3.2]{Ve:84}, \cite[\S3]{Ca:85}, \cite[Theorem 2]{MeCo:00}, \cite[Proposition 5.1]{Ta:00}). We state these alternative formulations, and recall their equivalence, so that we can easily use results from the harmonic-analysis literature (summarised in Appendix \ref{app:HA} below).

Given $\bx\in \Gamma$, let $\Theta^{\pm}(\bx)$ be the \emph{non-tangential approach set to $\bx$ from $\Omega^\pm$} defined by
\beq\label{eq:ntapproach2}
\Theta^{\pm}(\bx):= \Big\{ \by \in \Omega^{\pm} \, : \, |\bx-\by|\leq \min\big\{c, C \dist(\by,\Gamma)\big\}\Big\},
\eeq
for some $c>0$ and some $C>1$ sufficiently large depending on the Lipschitz character of $\Omega^\pm$.\footnote{E.g., $C> M+1$ is large enough when $M$ is the Lipschitz character.}
If $u\in C(\Omega^{\pm})$, its non-tangential maximal function $u^*: \Gamma \rightarrow [0,\infty]$ is defined by
\beq\label{eq:ntmax}
u^*(\bx) := \sup_{\by \in \Theta^{\pm}(\bx)}|u(\by)|, \quad \bx\in\Gamma.
\eeq
Define the non-tangential limit
\beq\label{eq:ntlim}
\widetilde{\gamma}^\pm u(\bx) := \lim_{\by\rightarrow \bx,\, \by \in \Theta^\pm(\bx)}u(\by).
\eeq
If $u\in C^2(\Omega^\pm)$, $\Delta u=0$, and $u^* \in \LtG$, then $\widetilde{\gamma}^\pm u(x)$ is well-defined for almost all $x\in \Gamma$ and $\widetilde{\gamma}^\pm u \in \LtG$ by \cite[Corollary 5.5]{JeKe:95} (restated as Part (ii) of Theorem \ref{thm:JeKe:95} below). Furthermore, if $u\in H^s_{\rm loc}(\Omega^\pm)$ with $s>1/2$, then $\widetilde{\gamma}^\pm u =\gamma^\pm u$ by \cite[Lemma A.9]{ChGrLaSp:12} (restated as Lemma \ref{lem:A9} below).

\begin{definition}[Laplace IDP formulated via non-tangential limits]\label{def:idp2}
Given $g_D\in L^2(\Gamma)$, we say that $u\in C^2(\Oi)$ with $u^* \in L^2(\Gamma)$ satisfies the IDP if $\Delta u =0$ in $\Oi$ and $\widetilde{\gamma}^- u = g_D$ on $\Gamma$.
\end{definition}

\begin{definition}[Laplace EDP formulated via non-tangential limits]\label{def:edp2}
With $\Oi$ and $\Oe$ as above, assume further that $\Oe$ is connected.
Given $g_D\in L^2(\Gamma)$, we say that $u\in C^2(\Oe)$ with $u^* \in L^2(\Gamma)$ satisfies the EDP if $\Delta u =0$ in $\Oe$, $\widetilde{\gamma}^+ u = g_D$ on $\Gamma$, and $u(\bx)= O(1)$ when $d=2$ and $u(\bx) = o(|\bx|^{3-d})$ when $d\geq 3$ as $|\bx| \rightarrow \infty$ (uniformly in all directions $\bx/|\bx|$).
\end{definition}

Existence and uniqueness of the solutions of these formulations of the IDP and EDP go back to the work of Dahlberg \cite{Da:79}, and are given explicitly in, e.g., \cite[Corollary 3.2 and Lemma 3.7]{Ve:84}, \cite[\S3]{Ca:85}.
The following equivalence result is proved in Appendix \ref{app:BVPequiv}.

\begin{theorem}[Equivalence of the different formulations of the IDP and EDP]\label{thm:BVPequiv}
If $g_D\in H^{1/2}(\Gamma)$, then the solution of the IDP in the sense of Definition \ref{def:idp} is the solution of the IDP in the sense of Definition \ref{def:idp2}, and vice versa.

Similarly, if $g_D\in H^{1/2}(\Gamma)$, then the solution of the EDP in the sense of Definition \ref{def:edp} is the solution of the EDP in the sense of Definition \ref{def:edp2}, and vice versa.
\end{theorem}

\subsection{Recap of results about layer potentials and their non-tangential limits}\label{sec:recap}

Recall that the \emph{surface gradient} operator on $\Gamma$ is the unique operator such that, when $v \in C^1(\overline{\Oi})$, $\gv = \bn (\bn\cdot \gv) + \nT (\gmv)$ on $\Gamma$ (and similarly for $v \in C^1(\overline{\Oe})$); for an explicit expression for $\nT$ in terms of a parametrisation of $\Gamma$, see, e.g., \cite[Equation A.14]{ChGrLaSp:12}.

The following results all rely on the harmonic-analysis results in \cite{CoMcMe:82} and \cite{Ve:84} (see also the accounts in \cite[Chapter 15]{MeCo:00}, \cite[Chapter 4]{Ta:00}, \cite[Chapter 2, Section 2]{Ke:94}).
Define
\beqs
\nT S \phi(\bx):= \int_\Gamma \nabla_{\Gamma,\bx}\Phi(\bx,\by)\phi(\by)\rd s(\by) \quad\text{ for a.e. } \bx\in \Gamma,
\eeqs
where the integral is understood in the principal-value sense. By \cite[Theorem 1.6]{Ve:84}, $\nT S:\LtG\rightarrow\LtG$, with this mapping continuous, and $(\nT S)\phi= \nT (S\phi)$.
The following potential was introduced in \cite[\S2]{Ca:85}; given  $\vfb\in (L^\infty(\Gamma))^d$ that is real-valued (which we assume throughout), let
\beq\label{eq:Caldp}
\Caldp \phi(\bx):= \int_\Gamma \vfb(\by) \cdot \nabla_\by \Phi(\bx,\by) \phi(\by)\, \rd s(\by) \quad
\quad\tfor\bx \in \Rea^d\setminus \Gamma,
\eeq
and let
\beq\label{eq:Caldb}
\Caldb \phi(\bx):= \int_\Gamma  \vfb(\by)\cdot\nabla_\by\Phi(\bx,\by)\phi(\by) \, \rd s(\by) \quad\quad\text{ for a.e. }\bx \in \Gamma,
\eeq
where the integral in \eqref{eq:Caldb} is understood in the principal-value sense.
The results of \cite{CoMcMe:82} and \cite{Ve:84} imply that $\Caldb:\LtG\rightarrow \LtG$ and, for $\phi\in\LtG$,
$\Caldp \phi \in C^2(\Omega^\pm)$, $\Caldp \phi$ satisfies Laplace's equation, and
$(\Caldp\phi)^*\in \LtG$ with
\beq\label{eq:jumpCald}
\widetilde{\gamma}^{\pm}
\Caldp \phi(\bx) = \pm \half \big( \vfb(\bx)\cdot \bn(\bx)\big) \phi(\bx) + \Caldb \phi (\bx) \quad\quad\text{ for a.e. } \bx\in \Gamma.
\eeq

Observe that (i) when $\vfb= \bn$, $\Caldp = \mathcal{D}$, $\Caldb= D$, and \eqref{eq:jumpCald} is the usual jump relation for the double-layer potential, and (ii) we can rewrite $\Caldb$ as
\beq\label{eq:Caldb2}
\Caldb \phi(\bx) = \int_\Gamma \left( \vfb(\by)\cdot \bn(\by) \pdiff{\Phi(\bx,\by)}{n(\by)} \phi(\by) + \vfb(\by) \cdot \nT \Phi(\bx,\by) \phi(\by)\right) \rd s(\by).
\eeq
In a similar way to how the $L^2$ adjoint of $D$ is $D'$ (see, e.g., \cite[Chapter 15, text around Equation 4.10]{MeCo:00}), the $L^2$ adjoint of $\Caldb$ is
\beq\label{eq:Caldb'}
\Caldb' \phi(\bx) := \big(\vfb(\bx) \cdot \bn(\bx) \big)D' \phi(\bx) + \vfb(\bx) \cdot \nT S \phi(\bx).
\eeq
The significance of the operator $\Caldb'$ is that it appears in the inner product of $\vfb$ with the non-tangential limit of $\nabla\mathcal{S}$, where $\cS$ is the single-layer potential defined for $\phi\in \LtG$ by
\beq\label{eq:SLP}
\cS\phi(\bx):= \int_\Gamma \Phi(\bx,\by)\phi(\by)\, \rd s(\by) \quad\quad\tfor \bx \in\Rea^d\setminus \Gamma.
\eeq
Indeed, by \cite[Theorems 1.6 and 1.11]{Ve:84} (see also \cite[Theorem 5]{MeCo:00}, \cite[Equation 2.30]{ChGrLaSp:12}), for almost every $\bx\in \Gamma$,
\beq\label{eq:gradS}
\widetilde{\gamma}^\pm
\nabla \cS\phi(\bx) = \bn(\bx) \left( \mp \half I + D'\right) \phi(\bx) + \nT(S\phi)(\bx),
\eeq
so that
\beq\label{eq:gradS2}
 \vfb(\bx)\cdot\widetilde{\gamma}^\pm\nabla \cS\phi(\bx) = \left( \mp\half(\bZ(\bx)\cdot\bn(\bx))I + \Caldb'\right)\phi(\bx).
\eeq

\section{Statement of the main results}\label{sec:main_results}

\subsection{New boundary integral equations for the Laplace interior and exterior Dirichlet problems on general Lipschitz domains for $d\geq 3$}\label{sec:dgeq3}

We focus on the case $d\geq 3$, since the question of whether or not there exist BIE formulations of the Laplace IDP and EDP that are coercive, or coercive up to a compact perturbation, on Lipschitz domains is more pressing when $d=3$ than $d=2$ (because of
the existing convergence theory for $\pm\half I + D$ and $\pm\half I+ D'$ on curvilinear polygons \cite{Sh:69, Ch:84, Sh:91}
 but negative results for these operators for certain 3-d star-shaped polyhedra \cite{ChSp:22} recapped in \S\ref{sec:openbook}).
Results for $d=2$ are given in \S\ref{sec:d=2}.

\subsubsection{The interior Dirichlet problem}

Given $\vfb\in (L^\infty(\Gamma))^d$ and $\alpha \in \Rea$,
define the integral operators $\opLIZ$,
$\opLIZi$,
and $\opLIZRHS$ by
\begin{align}
\label{eq:opLIZ}
&\opLIZ:= \half (\vfb\cdot \bn)I - \Caldb' + \alpha S, \qquad
\opLIZi:= \half (\vfb\cdot \bn)I - \Caldb + \alpha S, \\
&\opLIZRHS:=-(\vfb\cdot \bn) H  - \vfb\cdot\nT \left( \half I + D\right) + \alpha\left(\half I + D\right),
\label{eq:opLIZRHS}
\end{align}
with the subscript $I$ standing for ``interior", and the $'$ superscript indicating that $\opLIZ$ is the $L^2$ adjoint of $\opLIZi$.

\begin{theorem}[New integral equations for Laplace IDP with $d\geq 3$]\label{thm:Laplace_int}

\

(i) \textbf{\emph{Direct formulation.}} Let $u$ be the solution of the Laplace IDP of Definition \ref{def:idp} with $d\geq 3$ and additionally $g_D\in H^1(\Gamma)$.
Then $\dnmu$ satisfies
\beq\label{eq:new_int}
\opLIZ \dnmu =\opLIZRHS g_D.
\eeq

(ii) \textbf{\emph{Indirect formulation.}} Given $g_D\in \LtG$, if $\phi\in\LtG$ satisfies
\beq\label{eq:new_int_indirect}
\opLIZi \phi =- g_D,
\eeq
then
\beq\label{eq:indirectansatz1}
u:= (\Caldp - \alpha \cS)\phi
\eeq
is the solution of the Laplace IDP of Definition \ref{def:idp2}.

(iii) \textbf{\emph{Continuity.}}
$\opLIZ: \LtG\rightarrow\LtG$,
$\opLIZi: \LtG\rightarrow\LtG$, and $\opLIZRHS:\HoG\rightarrow\LtG$, and these mappings are continuous.

 (iv) \textbf{\emph{Coercivity up to compact perturbation.}} If $\bZ\in (C(\Gamma))^d$ and there exists $c>0$ such that
\beq\label{eq:c}
\vfb(\bx) \cdot \bn(\bx) \geq c \quad \text{ for almost every } \bx\in \Gamma,
\eeq
 then both $\opLIZ$ and $\opLIZi$ are the sum of a coercive operator and a compact operator on $\LtG$.

 (v) \textbf{\emph{Invertibility for all $\alpha>0$.}} If $\alpha>0$, $\bZ\in (C^{0,\beta}(\Gamma))^d$ for some $0<\beta<1$, and there exists $c>0$ such that
\eqref{eq:c} holds,
then both $\opLIZ:\LtG\to\LtG$ and $\opLIZi:\LtG\to\LtG$ are invertible.

 (vi) \textbf{\emph{Coercivity for sufficiently large $\alpha$.}} If $\bZ$ satisfies \eqref{eq:c}, $\bZ \in (C^{0,1}(\Gamma))^d$ with Lipschitz constant $L_{\vfb}$, and
 \beq\label{eq:alpha_bound_boundary}
 2\alpha \geq 3d L_{\vfb},
 \eeq
  then both
$\opLIZ$ and $\opLIZi$ are coercive on $\LtG$ with coercivity constant $c/2$ (with $c$ defined by \eqref{eq:c});
indeed,
\beq\label{eq:BIEcoercive}
\big(\opLIZ \psi,\psi\big)_{\LtG} \geq \frac{c}{2} \N{\psi}^2_{\LtG} \quad\text{ for all real-valued } \psi \in \LtG,
\eeq
and similarly for $\opLIZi$.
 \end{theorem}

Recall that if $A$ is real and $(A\psi,\psi)\geq \Ccoer \|\psi\|^2_{\LtG}$ for all real-valued $\psi\in \LtG$, then $\Re(A\phi,\phi)\geq \Ccoer \|\phi\|^2_{\LtG}$ for all complex-valued $\phi\in\LtG$; thus \eqref{eq:BIEcoercive} implies that $\opLIZ$ and $\opLIZi$ are coercive on complex-valued $\LtG$.

For any bounded Lipschitz  open set $\Oi$ there exists $\bZ \in (C^{0,1}(\Gamma))^d$ such that \eqref{eq:c} holds; for completeness, and because a formula for $\bZ$ is needed for implementation, we include this result and a concrete, constructive proof in \S\ref{sec:Z} below. \esnote{to edit when cut A} 
The combination of this result and Parts (iii) and (vi) of Theorem \ref{thm:Laplace_int} imply that, for any bounded Lipschitz open set, there exists a BIE formulation of the Laplace IDP that is continuous and coercive in $\LtG$.

The vector field $\bZ$ can be thought of as a ``regularised normal vector"; the choice $\bZ= \bn$ satisfies \eqref{eq:c} but does not have the regularity required for Parts (iv), (v), and (vi) of Theorem \ref{thm:Laplace_int} unless $\Oi$ is, respectively, $C^1$, $C^{1,\beta}$, or $C^{1,1}$. Indeed, from Parts (iv)-(vi) of the theorem we see that the stronger the property one wishes to obtain for $\opLIZ$ and $\opLIZi$, the more regularity of $\mathbf \bZ$ is required. E.g.,  coercivity up to a compact perturbation is proved  for continuous $\mathbf Z$ satisfying \eqref{eq:c} but coercivity is proved only for Lipschitz $\mathbf Z$ satisfying \eqref{eq:c}. 

 \subsubsection{The exterior Dirichlet problem}

Given  $\vfb\in (L^\infty(\Gamma))^d$ and $\alpha \in \Rea$,
define the integral operators $\opLEZ$,
$\opLEZi$,
and $\opLEZRHS$ by
\begin{align}
\label{eq:opLEZ}
&\opLEZ:= \half (\vfb\cdot \bn)I + \Caldb' + \alpha S, \qquad
\opLEZi:= \half (\vfb\cdot \bn)I + \Caldb + \alpha S, \\
&\opLEZRHS:=(\vfb\cdot \bn) H  + \vfb\cdot\nT \left( -\half I + D\right) + \alpha\left(-\half I + D\right),
\label{eq:opLEZRHS}
\end{align}
with the subscript $E$ standing for ``exterior".

\begin{theorem}[New integral equations for Laplace EDP with $d\geq 3$]\label{thm:Laplace_ext}

\

(i) \textbf{\emph{Direct formulation.}} Let $u$ be the solution of the Laplace EDP of Definition \ref{def:edp} with $d=3$ and additionally $g_D\in H^1(\Gamma)$. Then $\dnpu$ satisfies
\beq\label{eq:new_ext}
\opLEZ\dnpu =\opLEZRHS g_D.
\eeq

(ii) \textbf{\emph{Indirect formulation.}} Given $g_D\in \LtG$, if $\phi\in \LtG$ satisifes
\beq\label{eq:new_ext_indirect}
\opLEZi \phi = g_D,
\eeq
then
\beq\label{eq:indirect_ext_u}
u:= (\Caldp + \alpha \cS)\phi
\eeq
is the solution of the Laplace EDP of Definition \ref{def:edp2}.

  (iii) \textbf{\emph{Continuity.}}
 $\opLEZ: \LtG\rightarrow\LtG$,
$\opLEZi: \LtG\rightarrow\LtG$,
and $\opLEZRHS:\HoG\rightarrow\LtG$, and these mappings are continuous.

  (iv) \textbf{\emph{Coercivity up to compact perturbation.}} If $\bZ\in (C(\Gamma))^d$ and there exists $c>0$ such that \eqref{eq:c} holds,
 then both $\opLEZ$ and $\opLEZi$ are the sum of a coercive operator and a compact operator on $\LtG$.

 (v) \textbf{\emph{Invertibility for all $\alpha>0$}.} If $\alpha>0$, $\bZ\in (C^{0,\beta}(\Gamma))^d$ for some $0<\beta<1$, and there exists $c>0$ such that
\eqref{eq:c} holds, then
 both $\opLEZ:\LtG\to\LtG$ and $\opLEZi:\LtG\to\LtG$ are invertible.

 (vi) \textbf{\emph{Coercivity for sufficiently large $\alpha$.}} If $\bZ\in (C^{0,1}(\Gamma))^d$ with Lipschitz constant $L_{\vfb}$ and \eqref{eq:alpha_bound_boundary} holds, then both
$\opLEZ$ and $\opLEZi$ are coercive on $\LtG$ with coercivity constant $c/2$ (with $c$ defined by \eqref{eq:c}), in that \eqref{eq:BIEcoercive} holds with $\opLIZ$ replaced by either $\opLEZ$ or $\opLEZi$.
 \end{theorem}

Similar to the case of the IDP, the result in \S\ref{sec:Z}
\esnote{to edit when cut A} 
 and Parts (iii) and (vi) of Theorem \ref{thm:Laplace_ext} imply that, for any bounded Lipschitz open set $\Oi$ such that $\Oe$ is connected, there exists a BIE formulation of the Laplace EDP that is continuous and coercive in $\LtG$.

\subsubsection{The new formulations of the IDP and EDP for $d\geq 3$ on domains that are star-shaped with respect to a ball}\label{sec:Laplace_star}

When $\Oi$ is star-shaped with respect to a ball, the coercivity results in Theorems \ref{thm:Laplace_int} and \ref{thm:Laplace_ext} take a particularly simple form.

\begin{definition}
(i) $\domaingen$ is \emph{star-shaped with respect to the point $\bx_0$} if, whenever $\bx \in \domaingen$, the segment $[\bx_0,\bx]\subset \domaingen$.

\noi (ii) $\domaingen$ is \emph{star-shaped with respect to the ball $B_{\radius}(\bx_0)$} if it is star-shaped with respect to every point in $B_{\radius}(\bx_0)$.

\end{definition}

\begin{lemma}\textbf{\emph{(\cite[Lemma 5.4.1]{Mo:11})}}\label{lem:star} If $\domaingen$ is Lipschitz with outward unit normal vector $\bnu$, then
$\domaingen$ is star-shaped with respect to $B_{\radius}(\bx_0)$, for some $\radius>0$, if and only if $(\bx-\bx_0) \cdot \bnu(\bx) \geq {\radius}$ for all  $\bx \in \partial \domaingen$ for which $\bnu(\bx)$ is defined.
\end{lemma}

\noi From now on, if a domain $\domaingen$ is star-shaped with respect to $\bx_0$, we assume (without loss of generality) that $\bx_0=\bze$.

\begin{theorem}[Coercivity for star-shaped domains]\label{thm:star}
Let $\Oi \subset \Rea^d$, $d\geq 3$, be a bounded Lipschitz domain that is star-shaped with respect to a ball of radius $\radius$, i.e.
\beq\label{eq:ssg}
\radius:= \essinf_{\bx\in\Gamma}( \bx \cdot \bn(\bx)),
\eeq
Then
\beqs
\opLIx \,\,\tand\,\,\opLIxi, \,\,\text{ with } \alpha\geq -(d-2)/2,
\eeqs
and
\beqs
\quad\opLEx\,\, \tand\,\, \opLExi,\,\,\text{ with } \alpha\geq (d-2)/2,
\eeqs
are all coercive on $\LtG$ with coercivity constant $\radius/2$, in that \eqref{eq:BIEcoercive} holds with $c$ replaced by $\radius$, and $\opLIZ$ replaced by any one of $\opLIx, \opLIxi, \opLEx$, or $\opLExi$.
\end{theorem}

\subsection{Convergence and conditioning of the associated Galerkin methods}\label{sec:Galerkin_imp}

We now show how Theorems \ref{thm:Laplace_int} and \ref{thm:Laplace_ext} imply that (i) the associated Galerkin methods converge (see \S\ref{sec:convergence}), and (ii)
the associated Galerkin matrices are provably well-conditioned
as the discretisation is refined, without the need for operator preconditioning (see \S\ref{sec:conditioning}).
We focus on the case $d\geq 3$ and the new BIE formulations for the IDP and EDP (appearing in Theorems \ref{thm:Laplace_int} and \ref{thm:Laplace_ext}), but analogous results hold for the BIEs for star-shaped domains in \S\ref{sec:Laplace_star}
and also for the new BIEs for $d=2$ in \S\ref{sec:d=2} below.

\subsubsection{Convergence of the Galerkin method for the new formulations}\label{sec:convergence}

\begin{corollary}[Convergence of the Galerkin method]\label{cor:Galerkin}
Let $(\cH_N)_{N=1}^\infty$ denote any sequence of finite-dimensional subsets of $\cH := L^2(\Gamma)$ that is asymptotically dense in $L^2(\Gamma)$ in the sense defined in \S\ref{sec:Gal}.

(a) If $\bZ\in (C^{0,\beta}(\Gamma))^d$ for some $0<\beta<1$, and there exists $c>0$ such that
\eqref{eq:c} holds then, for all $\alpha >0$, the Galerkin method \eqref{eq:G} applied to any one of the BIEs \eqref{eq:new_int}, \eqref{eq:new_int_indirect}, \eqref{eq:new_ext} or \eqref{eq:new_ext_indirect} converges (in the sense defined in \S\ref{sec:intro}).

(b) If, additionally, $\bZ$ is Lipschitz and $\alpha$ satisfies \eqref{eq:alpha_bound_boundary}, then, additionally, the Galerkin solution exists for every finite-dimensional subspace $\cH_N\subset\LtG$ and satisfies the quasioptimal error estimate \eqref{eq:quasioptimal}, with constant $2\|\opG\|_{\LtGt}/c$, where $\opG :=\opLIZi$ for the BIEs \eqref{eq:new_int} and \eqref{eq:new_int_indirect}, $\opG :=\opLEZi$ for the BIEs \eqref{eq:new_ext} and \eqref{eq:new_ext_indirect}.
\end{corollary}

Since the proof is so short, we include it here.

\

\bpf[Proof of Corollary \ref{cor:Galerkin}]
(a) This follows from Parts (iv) and (v) of Theorem \ref{thm:Laplace_int}/Theorem \ref{thm:Laplace_ext}
and Part (b) of Theorem \ref{thm:Galerkin}.
(b) This follows from Part (vi) of Theorem \ref{thm:Laplace_int}/Theorem \ref{thm:Laplace_ext} and Part (c) of Theorem \ref{thm:Galerkin}.
\epf

\

We highlight that Corollary \ref{cor:Galerkin} is the first time convergence of the Galerkin method for a BIE posed in $\LtG$ used to solve a boundary-value problem for Laplace's equation has been proved with the only assumption on $\Gamma$ that it is Lipschitz; the same is true if $\Gamma$ is assumed to be Lipschitz polyhedral in 3-d.

\bre[Bounding the best approximation and Galerkin errors] \label{rem:BestApprox}
For 3-d Lipschitz polyhedra  the smoothness of the solution, in particular its singularities at corners and edges, is well understood (see, e.g., \cite{PeSt:90}) for the direct formulations \eqref{eq:new_int} and \eqref{eq:new_ext}, where the solution of the integral equation is $\phi=\partial_n^\pm u$. Moreover, it is well understood how to design effective $h$- and $hp$-boundary element approximation spaces $\cH_N$
based on graded, anisotropic meshes so as to obtain optimal best approximation error estimates (see, e.g., \cite{PeSt:90,El:92a,El:95,LuNi:95,MaSt:97}), indeed exponential convergence of $\min_{\psi \in \cH_N}\big\|\phi-\psi\big\|_{\LtG}$ as a function of $M_N:=\dim(\cH_N)$ if the Dirichlet data $g_D$ is the restriction to $\Gamma$ of an analytic function (see \cite[Theorem 3.1]{MaSt:97}). Further,  by Part (a) of Corollary \ref{cor:Galerkin} and the quasioptimality \eqref{eq:quasioptimal2}, the same rates of convergence follow for the Galerkin error $\|\phi-\phi_N\|_\LtG$ as long as $\alpha>0$.
\ere

\subsubsection{Solution of the Galerkin linear systems of the new formulations}\label{sec:conditioning}

Let $\cH_N = {\rm span} \{\psi^{N}_1,\ldots, \psi^{N}_{M_N}\}$, with $M_N=\dim(\cH_N)$ and $\{\psi^{N}_1,\ldots, \psi^{N}_{M_N}\}$ a basis for $\cH_N$.
The Galerkin method \eqref{eq:G} applied to \eqref{eq:new_int_indirect}
is then equivalent to the linear system
\beq\label{eq:system0}
\MA \bx = \bbb
\eeq
where
\beq\label{eq:matrix}
(\MA)_{ij} := \big(\opG \psi^{N}_j,\psi^{N}_i\big)_{L^2(\Gamma)} \quad\tand\quad b_i := -(g_D,\psi^{N}_i)_{L^2(\Gamma)}, \quad i,j=1,\ldots, M_N,
\eeq
with $\opG:= \opLIZi$, and with the Galerkin solution $\phi_N$ given by $\phi_N= \sum_{j=1}^{M_N} x^{N}_j \psi^{N}_j$, where $\bx = (x_1^N,\ldots,x_{M_N}^N)^T$. The Galerkin method applied to \eqref{eq:new_int}, \eqref{eq:new_ext_indirect}, or \eqref{eq:new_ext}, respectively, is also equivalent to \eqref{eq:system0} with $\opG:=\opLIZ$, $\opLEZi$, or $\opLEZ$  in \eqref{eq:matrix} and with correspondingly different definitions of the right-hand side components $b_i$.

In each case, whether $\opG=\opLIZi$, $\opLIZ$, $\opLEZi$, or $\opLEZ$, the matrix $\MA$ defined in \eqref{eq:matrix} is non-symmetric, and a popular method for solving such non-symmetric linear systems is
the \emph{generalised minimum residual method (GMRES)} \cite{SaSc:86}, which we now briefly recall.
Consider the abstract  linear system
$\MC \bx = \bd$
in $\mathbb{C}^{M_N}$, where $\MC \in \mathbb{C}^{M_N\times M_N}$ is invertible. Let $\bx_{0}$ be an  initial guess for $\bx$, and define the corresponding initial residual $\br^0 :=\MC \bx^0-\bd$ and the corresponding standard Krylov spaces by
\beqs
\cK^m(\MC, \br_0) := \mathrm{span}\big\{\MC^j \br_0 : j = 0, \ldots, m-1\big\}.
\eeqs
For $m \geq 1$, define the $m$th GMRES iterate $\bx_m$  to be the unique element of $\cK^m$ such that its residual $\br_m:= \MC \bx_m-\bd$ satisfies  the   minimal residual  property
\beqs
 \Vert \br_m \Vert_2 =  \min_{\by \in \cK_m(\MC, \br^0)} \Vert   {\MC} {\by} -{\bd}\Vert_2.
\eeqs

The main result of this subsection (Theorem \ref{thm:conditioning} below) is a result about the convergence of GMRES applied to \eqref{eq:system0} preconditioned by diagonal matrices. This result is proved under the following assumption in which (and subsequently) for every $v_N\in \cH_N$ we denote by $\bv\in \Com^{M_N}$ the unique vector $\bv=(v^N_1,\ldots, v^N_{M_N})^T$ such that $v_N = \sum_{j=1}^{M_N} v^N_j \psi^{N}_j$.

\begin{assumption}\label{ass:D}
$(\cH_N)_{N=1}^\infty$, and the associated bases $(\{\psi^N_1,\ldots,\psi^N_{M_N}\})_{N=1}^\infty$, are such that there exists a sequence of diagonal matrices $(\MD_N)_{N=1}^\infty$ and $C_1, C_2>0$, independent of $N$, such that
\beq\label{eq:norm_equiv}
C_1 \big\|\MD_N^{1/2} \bw\big\|_2 \leq \N{w_N}_{\LtG} \leq C_2 \big\|\MD_N^{1/2} \bw\big\|_2 \quad\tfa w_N \in \cH_N.
\eeq
\end{assumption}

\bre[Relation of $C_1$ and $C_2$ in \eqref{eq:norm_equiv} to the mass matrix]
Let $\MM_N$ be the mass matrix defined by
\beq\label{eq:massmatrix}
(\MM_N)_{ij} := \big(\psi_j^{N},\psi^{N}_i\big)_{L^2(\Gamma)}.
\eeq
Since $(\MM_N\bw,\bw)_2 = \|w_N\|^2_{L^2(\Gamma)}$,
and thus $\|\MM_N^{1/2}\bw\|_2 = \|w_N\|_{L^2(\Gamma)}$,
 \eqref{eq:norm_equiv} implies that
\beqs
C_1 \| \bv\|_2 \leq \big\| \MM^{1/2}_N \MD^{-1/2}_N \bv\big\|_{\LtG} \leq C_2 \|\bv\|_2 \quad\tfa v_N \in \cH_N;
\eeqs
i.e., $\MD_N^{1/2}$ can be considered as a right-preconditioner for $\MM_N^{1/2}$, removing the $N$-dependence of the norms of $\MM_N^{1/2}$ and $\MM_N^{-1/2}$.
\ere

\bre[When does Assumption \ref{ass:D} hold?]
If $(\psi^{N}_j)_{j=1}^{M_N}$ is an orthogonal basis of $\cH_N$, then Assumption \ref{ass:D} is satisfied with $\MD_N= \MM_N$ and $C_1=C_2=1$; therefore, if $(\psi^{N}_j)_{j=1}^{M_N}$ is an orthonormal basis of $\cH_N$, then Assumption \ref{ass:D} is satisfied with $\MD_N$ equal the identity matrix.

Lemma \ref{lem:meshes} below shows that Assumption \ref{ass:D} is satisfied (and specifies the matrices $\MD_N$) when $(\cH_N)_{N=1}^\infty$ are piecewise-polynomial subspaces allowing discontinuities across elements, under very mild constraints on the sequence of meshes; in particular
Lemma \ref{lem:meshes} covers nodal basis functions on highly anisotropic meshes, such as the meshes highlighted in Remark \ref{rem:BestApprox}. 
We highlight that the assumption that discontinuities are allowed is made so that we can assume in the proof that each basis function is supported on only one element, but we expect the result to hold more generally. In particular, if $d=3$ and the sequence of meshes is regular, shape-regular, and quasi-uniform (in the sense of \cite[Definitions 4.1.4, 4.1.12, and 4.1.13]{SaSc:11}, respectively)
on a polyhedral or piecewise curved domain (in the sense of \cite[Assumptions 4.3.17 and 4.3.18]{SaSc:11}, respectively),
then Assumption \ref{ass:D} holds for a general nodal basis with $\MD_N = h^{d-1}\MI_N$ by \cite[Theorem 4.4.7]{SaSc:11}.
\ere

Let $\by_m$ be the $m$th iterate when the linear system
\beq\label{eq:pclinear}
(\MD_N^{-1/2}\MA \MD_N^{-1/2})\by = \MD_N^{-1/2}\bbb
\eeq
is solved using GMRES with zero initial guess. (Since $\MD_N$ is diagonal, the cost of calculating the action of $\MD_N^{-1/2}\MA \MD_N^{-1/2}$ is dominated by the cost of calculating the action of $\MA$.)
Let
\beq\label{eq:phiNm}
\phi_N^m :=\sum_{j=1}^{M_N} (\MD_N^{-1/2}\by_m)_j \psi^{N}_j,
\eeq
and observe that, by \eqref{eq:pclinear}, \eqref{eq:system0}, and \eqref{eq:matrix}, the Galerkin solution $\phi_N$ is given by
\beq\label{eq:phiN}
\phi_N=\sum_{j=1}^{M_N} (\MD_N^{-1/2}\by)_j \psi^{N}_j.
\eeq

\begin{theorem}[Convergence of GMRES applied to the linear system \eqref{eq:pclinear}]\label{thm:conditioning}
Assume that $\bZ$ is Lipschitz, there exists $c>0$ such that
\eqref{eq:c} holds, $\alpha$ satisfies \eqref{eq:alpha_bound_boundary}, and Assumption \ref{ass:D} holds.
With $C_1$ and $C_2$ as in  \eqref{eq:norm_equiv}, and where $\|\opG\|$ denotes $\|\opG\|_{L^2(\Gamma)\to L^2(\Gamma)}$,
 let $\beta\in [0,\pi/2)$ be defined such that
\beq\label{eq:cosbeta}
\cos\beta= \frac{c}{2\|\opG\|}\left(\frac{C_1}{C_2}\right)^{2}\quad\text{ and let } \quad \gamma_\beta:= 2 \sin\left(\frac{\beta}{4-2\beta/\pi}\right)
\eeq
(observe that $\cos \beta$ is indeed $\leq 1$ since, by definition, $C_1\leq C_2$ and $c/2 \leq \|\opG\|$).
Given $\eps>0$, if
\beq\label{eq:iteration_bound}
m\geq  \left( \log\left(\frac{1}{\gamma_\beta}\right)\right)^{-1}
\left[
\log \left( \frac{24 \|\opG\|}{c}\left(\frac{C_2}{C_1}\right)^3\right)+ \log\left(\frac{1}{\eps}\right) \right],
\eeq
then
\beq\label{eq:full_discrete}
\frac{
\N{\phi-\phi_N^m}_{\LtG}}
{\N{\phi}_{\LtG}}
\leq (1+\eps)\frac{2\|\opG\|}{c}\left( \min_{\psi\in\cH_N}\frac{\N{\phi-\psi}_{\LtG}}{\N{\phi}_{\LtG}}\right)+ \eps
\eeq
(compare to \eqref{eq:quasioptimal}).
\end{theorem}

The key point about Theorem \ref{thm:conditioning} is that \emph{both} the bound on the number of iterations \eqref{eq:iteration_bound} \emph{and} the terms on the right-hand side of
\eqref{eq:full_discrete} other than the best-approximation error are independent of the dimension $M_N$. Therefore, the number of iterations required to solve systems involving $\MD_N^{-1/2}\MA \MD_N^{-1/2}$  to a prescribed accuracy
does \emph{not} increase as the discretisation is refined and
$M_N$ increases.
The same property holds when the conjugate-gradient method is applied to
sequences of $M_N\times M_N$ symmetric, positive-definite matrices whose condition number is bounded independently of $M_N$.

\bre[Bounds on the condition number]
Recall that, in general, a bound on the condition number
for a nonnormal matrix cannot be used to rigorously prove results about the convergence of GMRES applied to that matrix; see, e.g., \cite[Page 165]{LiTi:04}, \cite[Page 3]{Em:99}. We have no reason to expect that $\MA$ is normal, so to prove Theorem \ref{thm:conditioning} we crucially use the coercivity of $\opLIZi$.

Nevertheless, since there is a long history of studying the condition numbers of second-kind integral equations posed on $\LtG$, we record that in the course of proving Theorem \ref{thm:conditioning} we prove that, where $C:= \|\opG\|_{L^2(\Gamma)\to L^2(\Gamma)}$,
\beqs
\cond\big( \MD_N^{-1/2}\MA \MD_N^{-1/2}\big) \leq \frac{2 C}{c}\left(\frac{C_2}{C_1}\right)^2,
\eeqs
where $\cond(\MB):= \|\MB\|_2 \|\MB^{-1}\|_2$ (see \eqref{eq:GMRES3} and \eqref{eq:GMRES3a} below); i.e.,
$\cond( \MD_N^{-1/2}\MA \MD_N^{-1/2})$ is bounded independently of the dimension $M_N$.
Furthermore, by the arguments in \cite[\S III]{BeChGrLaLi:11}, \cite[Equation B.8]{MaGaSpSp:21}, and \cite[Equation 3.6.166]{At:97},
\beq\label{eq:conditionnumber}
\cond(\MA) \leq\frac{2C}{c}\,  \cond(\MM_N);
\eeq
recall that for a piecewise polynomial boundary element approximation space (in dimensions $d=2$ or $3$) using nodal basis functions on a quasiuniform mesh  (with these terms defined in \S\ref{sec:conditioning2a}), $\cond(\MM_N)$ is independent of the dimension $M_N$ (see, e.g., the proof of Part (i) of Lemma \ref{lem:meshes} or \cite[Remark 4.5.3]{SaSc:11}).
\ere

\bre[Calculating the entries of the Galerkin matrices for the new BIEs]
Calculating the entries of the Galerkin matrices for the new BIEs requires evaluating integrals involving \emph{only} the operators
$S$, $D$, and $D'$.
Indeed, for the direct BIE \eqref{eq:new_int}, the expression \eqref{eq:Caldb'} shows that constructing the Galerkin matrix requires evaluating integrals involving the operators above, and evaluating integrals of the form
\beq\label{eq:ibps}
\int_\Gamma \vfb(\bx)\cdot \nT \big(S \psi_j(\bx)\big) \psi_i (\bx)\, \rd s(\bx)
\eeq
where $\psi_i, \psi_j \in \cH_N$.
It is shown in \cite[\S4.3]{SpChGrSm:11} that, using integration by parts, the integral \eqref{eq:ibps} can be evaluated in terms of integrals involving derivatives of $\psi_i$ and values (but not derivatives) of $S\psi_j$.
Constructing the Galerkin matrix of the indirect BIE \eqref{eq:new_int_indirect} requires evaluating the integral
\beq\label{eq:ibps2}
\int_\Gamma \big(\Caldb \psi_j(\bx)\big) \psi_i (\bx)\, \rd s(\bx),
\eeq
where $\Caldb$ is defined by \eqref{eq:Caldb}.
Using the expansion $\bZ = \sum_{i=1}^d Z_i \be_i$ in the definition of $\Caldb$, we have
\beqs
\Caldb \phi(\bx)= \sum_{i=1}^d \be_i \cdot\int_\Gamma  \nabla_\by\Phi(\bx,\by)Z_i(\by)\phi(\by) \, \rd s(\by).
\eeqs
Given $\bx\in \Gamma$, $\be_i = (\be_i\cdot \bn(\bx)) \bn(\bx) + \be_T(\bx)$, where $\be_T(\bx)$ is tangent to $\Gamma$ at $\bx$.
Thus,
\beqs
\Caldb \phi(\bx) = -\sum_{i=1}^d\Big( (\be_i\cdot \bn(\bx)) D'\big(Z_i \phi\big)(\bx) + \be_i\cdot\nabla_\Gamma S \big(Z_i \phi\big)(\bx)\Big);
\eeqs
the integral \eqref{eq:ibps2} can therefore be evaluated in terms of integrals only involving $D'$ and (by the discussion above regarding \eqref{eq:ibps}) $S$.
\ere

\subsection{New boundary integral equations for the Laplace interior and exterior Dirichlet problems on general 2-d Lipschitz domains}\label{sec:d=2}

The biggest difference in going from $d\geq 3$ to $d=2$ is that the single-layer potential is no longer $o(1)$ at infinity, and is only $O(1)$ for a restricted class of densities; see \eqref{eq:Sasym}, \eqref{eq:Sasym2} below.
In this section, we first outline what parts of the $d\geq 3$ results in \S\ref{sec:dgeq3} immediately carry over to $d=2$.
We then present modifications of the integral equations in Theorem \ref{thm:Laplace_int} and Theorem \ref{thm:Laplace_ext} that are coercive for general 2-d Lipschitz domains when $\alpha$ is sufficiently large.

Inspecting the proof of Theorem \ref{thm:Laplace_int} in \S\ref{sec:proofs}, we see that
Parts (i), (ii), (iii), and (iv) hold when $d=2$ (i.e.,~everything apart from invertibility (v) and coercivity for sufficiently large $\alpha$ (vi)).

Similarly, inspecting the proof of Theorem \ref{thm:Laplace_ext} in \S\ref{sec:proofs}, we see that
Parts (i), (iii), (iv), and (v) hold when $d=2$ (i.e.,~everything apart from the indirect formulation (ii) and coercivity for sufficiently large $\alpha$ (vi)), although, firstly,
$\alpha u_\infty$ must be added to the right-hand side of the BIE \eqref{eq:new_ext}, where $u_\infty$ is the limit of $u$ at infinity\footnote{This is because Green's integral representation for the solution of the Laplace EDP with $d=2$ takes the form $u(\bx) = -\cS\dnpu (\bx) + \cD \gpu (\bx) + u_\infty$ for $\bx \in \Oe$.} and, secondly, Part (v) holds when $d=2$ provided the constant $a$ in the fundamental solution \eqref{eq:fund} is not equal to the capacity of $\Gamma$, $\mathrm{Cap}_\Gamma$ (defined in, e.g., \cite[Page 263]{Mc:00}), which holds, in particular, if $a>\diam(\Gamma)$.

Let
\beq\label{eq:P}
\projP \phi (\bx):= \frac{1}{|\Gamma|} \int_\Gamma \phi (\by)\, \rd s(\by)= \frac{1}{|\Gamma|} \big(\phi,1\big)_\LtG \quad \tfor\bx\in\Gamma;
\eeq
i.e., $\projP \phi$ is the mean value of $\phi$. Observe that $\projP^2=\projP$ and $\projP'=\projP$. Let $\projQ:=I-\projP$.

We give two theorems: the first for general 2-d Lipschitz domains, the second for 2-d star-shaped Lipschitz domains. Recall that we are assuming throughout that $\vfb$ is real-valued.

\begin{theorem}[New integral equations for Laplace IDP and EDP in 2-d]\label{thm:2d} Suppose that $\vfb\in (L^\infty(\Gamma))^2$ and $\alpha,\beta\in \Rea$.

(i) \textbf{\emph{IDP direct formulation.}} Let $u$ be the solution of the Laplace IDP of Definition \ref{def:idp} with $d=2$ and $g_D\in H^1(\Gamma)$.
Let
\beq\label{eq:TIprime}
T'_{I, \vfb,\alpha,\beta}:= \projQ \opLIZ \projQ + \beta \projP.
\eeq
Then $\dnmu$ satisfies
\beq\label{eq:2d_IDP_direct}
T'_{I,\vfb,\alpha,\beta}(\dnmu) =\projQ \opLIZRHS \,g_D.
\eeq

(ii) \textbf{\emph{IDP indirect formulation.}}
Let
\beqs
T_{I, \vfb,\alpha,\beta}:= \projQ \opLIZi \projQ + \beta \projP.
\eeqs
Given $g_D\in \LtG$, if $\phi$ satisifes
\beq\label{eq:2d_IDP_indirect}
T_{I,\vfb,\alpha,\beta}\phi =- g_D,
\eeq
then, if $d=2$,
\beq\label{eq:u2d_int}
u:= (\Caldp-\alpha\cS) \projQ\phi + \projP \opLIZi \projQ \phi - \beta \projP\phi
\eeq
is the solution of the Laplace IDP of Definition \ref{def:idp2}.

(iii) \textbf{\emph{EDP direct formulation.}} Let $u$ be the solution of the Laplace EDP of Definition \ref{def:edp} with $d=2$ and $g_D\in H^1(\Gamma)$.
Let
\beq\label{eq:TEprime}
T'_{E, \vfb,\alpha,\beta}:= \projQ \opLEZ \projQ + \beta \projP.
\eeq
Then $\dnpu$ satisfies
\beq\label{eq:2d_EDP_direct}
T'_{E,\vfb,\alpha,\beta}(\dnpu) =\projQ \opLEZRHS \,g_D.
\eeq

(iv) \textbf{\emph{EDP indirect formulation.}}
Let
\beqs
T_{E, \vfb,\alpha,\beta}:= \projQ \opLEZi \projQ + \beta \projP.
\eeqs
Given $g_D\in \LtG$, if $\phi$ satisifes
\beq\label{eq:2d_EDP_indirect}
T_{E,\vfb,\alpha,\beta}\phi =g_D,
\eeq
then, if $d=2$,
\beq\label{eq:u2d_ext}
u:= (\Caldp+\alpha\cS) \projQ\phi - \projP \opLIZi \projQ \phi + \beta \projP\phi
\eeq
is the solution of the Laplace EDP of Definition \ref{def:edp2}.

(v) \textbf{\emph{Coercivity.}} If $\vfb\in (C^{0,1}(\Gamma))^2$ satisfies \eqref{eq:c}, $\alpha$ satisfies \eqref{eq:alpha_bound_boundary}, and $\beta=c/2$, then
$T'_{I,\vfb,\alpha,\beta}$, $T_{I,\vfb,\alpha,\beta}$, $T'_{E,\vfb,\alpha,\beta}$, and $T_{E,\vfb,\alpha,\beta}$ are all coercive on $\LtG$ with coercivity constant $c/2$.
\end{theorem}

\begin{theorem}[New integral equations for 2-d star-shaped domains]\label{thm:2dstar}

\

(i) \textbf{\emph{IDP direct formulation.}} Let $u$ be the solution of the Laplace IDP of Definition \ref{def:idp} with $d=2$ and $g_D\in H^1(\Gamma)$.
Then $\dnmu$ satisfies
\beq\label{eq:2dstar1}
\left(A'_{I,\bx,0} - \frac{|\Gamma|}{4\pi}\projP\right) \dnmu =B_{I,\bx,0} \,g_D.
\eeq

(ii) \textbf{\emph{IDP indirect formulation.}}
Given $g_D\in \LtG$, if $\phi$ satisifes
\beq\label{eq:2dstar2}
\left(A_{I,\bx,0} - \frac{|\Gamma|}{4\pi}\projP\right)
\phi =- g_D,
\eeq
then, if $d=2$,
\beqs
u:= \Caldp \phi + \frac{|\Gamma|}{4\pi}\projP\phi
\eeqs
is the solution of the Laplace IDP of Definition \ref{def:idp2}.

(iiii) \textbf{\emph{EDP direct formulation.}} Let $u$ be the solution of the Laplace EDP of Definition \ref{def:idp} with $d=2$ and $g_D\in H^1(\Gamma)$.
Then $\dnpu$ satisfies
\beq\label{eq:2dstar3}
\left(A'_{E,\bx,0} + \frac{|\Gamma|}{4\pi}\projP\right) \dnpu =B_{E,\bx,0} \,g_D.
\eeq

(iv) \textbf{\emph{EDP indirect formulation.}}
Given $g_D\in \LtG$, if $\phi$ satisifes
\beq\label{eq:2dstar4}
\left(A_{E,\bx,0} + \frac{|\Gamma|}{4\pi}\projP\right)
\phi = g_D,
\eeq
then, if $d=2$,
\beqs
u:= \Caldp \phi + \frac{|\Gamma|}{4\pi}\projP\phi
\eeqs
is the solution of the Laplace EDP of Definition \ref{def:edp2}.

(v) \textbf{\emph{Coercivity.}} If $\Oi$  is star-shaped with respect to a ball of radius $\radius$ (i.e.~\eqref{eq:ssg} holds),
then each of the integral operators on the left-hand sides of \eqref{eq:2dstar1}, \eqref{eq:2dstar2}, \eqref{eq:2dstar3}, and \eqref{eq:2dstar4} is coercive on $\LtG$ with coercivity constant $\radius/2$.
\end{theorem}

\section{Discussion of the ideas behind the new BIEs and links to previous work}\label{sec:idea}

\subsection{How the BIEs arise}
The indirect BIE \eqref{eq:new_int_indirect} for the IDP arises from imposing the boundary condition on the ansatz $u= (\Caldp -\alpha \cS)\phi$ via taking the nontangential limit.
Similarly, the indirect BIE \eqref{eq:new_ext_indirect} for the EDP arises from the ansatz $u= (\Caldp +\alpha \cS)\phi$.
For the indirect BIEs for $d=2$ in Theorem \ref{thm:2d}, the idea is the same, except now a) the density in the ansatz is not a general $L^2(\Gamma)$ function (so that $\cS\phi$ has the correct behaviour at infinity), and b) extra terms are added to the ansatz to ensure that the resulting BIE is still coercive on $\LtG$.

For the direct BIE \eqref{eq:new_int} for the IDP, recall that $u=\cS \dnmu - \cD \gmu$ by Green's integral representation. The direct BIE \eqref{eq:new_int} then arises from considering
\beqs
-\vfb \cdot \widetilde{\gamma}^- (\nabla u ) + \alpha \gmu.
\eeqs
Similarly, the direct BIE \eqref{eq:new_ext} for the EDP arises from considering
\beqs
\vfb \cdot \widetilde{\gamma}^+ (\nabla u ) + \alpha \gpu,
\eeqs
with $u = -\cS\dnpu + \cD \gpu$. Alternatively, since (informally) $\vfb \cdot \nabla = (\vfb \cdot \bn) \partial_n + \vfb \cdot \nT$, the direct BIE \eqref{eq:new_int} can be obtained by adding (i) $(\vfb\cdot \bn)$ multiplied by the standard direct second-kind BIE
\beq\label{eq:inteq2}
\left(\half I - D'\right)\dnmu = -H\gmu,
\eeq
(ii) $-\vfb\cdot \nT$ applied to the standard direct first-kind BIE
\beq\label{eq:inteq1}
S\dnmu= \left(\half I + D\right) \gmu,
\eeq
and (iii) $\alpha$ multiplied by \eqref{eq:inteq1}. Similar considerations hold for the direct BIE \eqref{eq:new_ext}, and the 2-d direct BIEs of Theorems \ref{thm:2d} and \ref{thm:2dstar}, where, additionally, one uses that $\projP (\partial^{\pm}_n u)=0$ (see Lemma \ref{lem:Neumann_trace}).

\subsection{The other BVPs solved by the new BIEs}
The BIEs introduced in \S\ref{sec:main_results} to solve the Dirichlet problem can be used to solve other Laplace BVPs. Although the focus of this paper is on solving the Dirichlet problem, we highlight this fact here since
these other BVPs affect the properties of the new BIEs.

For example, the BIO $\opLEZ$ used to solve the EDP in Theorem \ref{thm:Laplace_ext} can also be used to solve the Laplace interior oblique Robin problem, i.e., the problem of finding $u$ in $\Oi$ satisfying $\Delta u=0$ and
\beq\label{eq:obliqueLintro}
(\bZ\cdot\bn) \dnmu + \bZ \cdot \nabla_\Gamma( \gmu) + \alpha \, \gmu =  g \quad\text{ on } \Gamma;
\eeq
see Definition \ref{prob:ioipL} and Theorem \ref{thm:star_in_obliqueLE} below. Similarly, the BIO $\opLIZ$ used to solve the IDP in Theorem \ref{thm:Laplace_int} can also be used to solve the Laplace exterior oblique Robin problem; see Definition \ref{prob:eoipL} and Theorem \ref{thm:star_in_obliqueL}  below.
This relationship means that the injectivity results implicit in Part (v) of Theorems \ref{thm:Laplace_int} and \ref{thm:Laplace_ext} are obtained by proving uniqueness of these oblique Robin problems; see \S\ref{subsec:oblique3}.

\subsection{The use of similar BIEs by Calder\'on \cite{Ca:85} and Medkov\'a \cite{Me:18}}
Calder\'on \cite{Ca:85} used indirect versions of the BIEs in Theorems \ref{thm:Laplace_int} and \ref{thm:Laplace_ext} with $\alpha=0$ to prove wellposedness results about the Dirichlet problem and the oblique derivative problem (i.e., \eqref{eq:obliqueLintro} with $\alpha=0$) with data in $L^p(\Gamma)$.
Indeed, \cite{Ca:85} posed the ansatz $u= \Caldp \phi$ for the IDP, which gives the BIE $A_{I, \vfb, 0} \phi =-g_D$ \cite[Page 39]{Ca:85}, and posed the ansatz $u=\cS \phi$ for the oblique derivative problem, which gives the BIE $A'_{E,\vfb,0} \phi = g$ \cite[Page 45]{Ca:85}.
Furthermore, Medkov\'a \cite[\S5.23]{Me:18}
posed the ansatz $u = \cS\phi$ for the interior oblique Robin problem, resulting in $A'_{I,\vfb,\alpha}\phi = -g$.

In both \cite{Ca:85} and \cite{Me:18}, the BIOs are proved to be Fredholm of index zero on $L^2(\Gamma)$; see \cite[Page 39]{Ca:85} (where the result is proved to hold on a slightly wider range of $L^p(\Gamma)$ spaces) and \cite[Proposition 5.23.2]{Me:18}.

\subsection{The main new properties of the BIEs of this paper:~coercivity for appropriate $\alpha$}\label{sec:Rellich_recap}
Building on the work of Calder\'on and Medkova, we show that the BIOs are not only Fredholm of index zero on $L^2(\Gamma)$, but \emph{invertible} for general Lipschitz domains as soon as $\alpha>0$, and, crucially, \emph{coercive} if $\alpha$ is chosen appropriately (so also \emph{coercive plus compact} for all $\alpha>0$). For star-shaped domains this coercivity can be proved using a simple modification of Calder\'on's proof that the BIOs are Fredholm of index zero (see Lemmas \ref{lem:Calderon} and \ref{lem:Calderon2} below).
For general domains this coercivity (for appropriate $\alpha$) is proved using Rellich-type identities (with this method also giving an alternative proof of coercivity for star-shaped domains). Recall that identities arising from multiplying $\Delta u$ by a derivative of $u$ are associated with the name Rellich, due to Rellich's introduction of the multiplier $\bx \cdot \gu$ for the Helmholtz equation in \cite{Re:40}; these identities have been well-used in the study of the Laplace, Helmholtz, and other elliptic equations, see, e.g., the overviews in \cite[Pages 111 and 112]{Ke:94}, \cite[\S5.3]{ChGrLaSp:12}, \cite[\S1.4]{MoSp:14}. Verchota \cite{Ve:84} famously used Rellich identities to prove invertibility of $\half I - D$ and $\half I - D'$ on $\LtG$ (see Remark \ref{rem:Verchota} below) and  Medkov\'a  \cite[\S5.23]{Me:18} also used Rellich identities to prove that $A'_{I, \vfb, \alpha}$ is invertible for sufficiently large $\alpha$ \cite[Lemma 5.23.1, Prop.~5.23.2, Theorem 5.23.4]{Me:18}.

Our coercivity results are proved using the identity arising from multiplying $\Delta u$ by $\vfb \cdot \nabla u + \alpha u$ (see Lemma \ref{lem:Rellich} below); our use of a multiplier that is a linear combination of $u$ and a derivative of $u$ is inspired by the use of such multipliers by Morawetz \cite{Mo:61, MoLu:68, Mo:75}, and the particular identity we use also appears as \cite[Equation 2.28]{KuMaVa:02}.
As recalled in \S\ref{sec:rationale}, the idea of proving coercivity of Laplace BIOs in the trace spaces goes back to N\'ed\'elec and Planchard \cite{NePl:73}, Le Roux \cite{LeRo:74}, Hsiao and Wendland \cite{HsWe:77}, and Steinbach and Wendland \cite{StWe:01}, with this method based on using Green's identity (i.e. multiplying $\Delta u$ by $u$).
The idea of proving coercivity of second-kind BIOs in $\LtG$ using Rellich-type identities was introduced in
\cite{SpChGrSm:11} for a particular  Helmholtz BIE on star-shaped domains and then further developed in  \cite{SpKaSm:15} for the standard second-kind Helmholtz BIE on smooth convex domains. The main differences between \cite{SpChGrSm:11, SpKaSm:15} and the present paper are that (i) \cite{SpChGrSm:11, SpKaSm:15} only consider direct BIEs for the exterior Helmholtz Dirichlet problem whereas the present paper considers direct and indirect BIEs for the interior and exterior Laplace Dirichlet problems and (ii) \cite{SpChGrSm:11, SpKaSm:15} only prove coercivity under geometric restrictions on $\Gamma$ (which is somewhat expected for the high-frequency Helmholtz equation; see \cite{BeSp:11}, \cite[\S6.3.2]{ChSpGiSm:20}), namely star-shapedness with respect to a ball for \cite{SpChGrSm:11} and strict convexity and a piecewise analytic $C^3$ boundary for  \cite{SpKaSm:15}, whereas the present paper proves coercivity of Laplace BIOs for general Lipschitz domains.

\subsection{Combined-potential ansatz for solutions of Laplace's equation}
A key difference between the indirect BIEs in the present paper and those in \cite{Ca:85} is that ours arise from the ansatz $u= (\Caldp -\alpha \cS)\phi$ for the solution of the Laplace IDP, whereas \cite{Ca:85} poses the ansatz $u = \Caldp\phi$. We saw in the discussion above that the presence of the parameter $\alpha$ -- i.e., the fact that we use a combined-potential ansatz -- is crucial for proving coercivity of our BIOs.

The combined-potential ansatz is also crucial to proving uniqueness for cases where coercivity does not hold. Indeed, using a linear combination of double- and single-layer potentials to find solutions of the Helmholtz equation is standard, and goes back to \cite{BrWe:65, Le:65, Pa:65}, with the motivation to ensure uniqueness at all wavenumbers (see \S\ref{sec:Helmholtz}).
Using such a combination for Laplace's equation is less common, but this was done by D.~Mitrea in \cite[Theorem 5.1]{MitreaD:97} and subsequently by Medkov\'a in \cite{Me:99}.
The rationale for this combined ansatz is similar, namely that the standard indirect second-kind equations (based on a double-layer-potential ansatz) have non-trivial null spaces for multiply connected domains (with these characterised in \cite{KulkarniEtal:05, MitreaD:97}) but the BIOs resulting from a combined double- and single-layer potential ansatz are invertible no matter the topology of $\Oi$; see \cite[Theorem 5.15.2]{Me:18} (for $d\geq 3$) and \cite[Theorem 5.15.3]{Me:18} (for $d=2$). The BIOs in \S\ref{sec:main_results} are also invertible (and even, for appropriate $\alpha$, coercive) no matter the topology of $\Oi$.

\section{Proofs of the main results}\label{sec:proofs}

In this section we prove all of the results in \S\ref{sec:main_results} \emph{apart from} the invertibility results in Part (v) of Theorem \ref{thm:Laplace_int}/\ref{thm:Laplace_ext}.
As discussed in \S\ref{sec:idea}, these invertibility results are equivalent to uniqueness of the Laplace interior and exterior oblique Robin problems, and these uniqueness results are proved in \S\ref{sec:oblique}.
Indeed, Part (v) of Theorem \ref{thm:Laplace_int} follows from Corollary \ref{cor:oblique:unique2}, and Part (v) of Theorem \ref{thm:Laplace_ext} follows from Corollary \ref{cor:oblique:unique}.

\subsection{Proofs of Parts (i), (ii), and (iii) of Theorems \ref{thm:Laplace_int} and \ref{thm:Laplace_ext}}\label{sec:4.1}

For Part (i) of Theorem \ref{thm:Laplace_int}, first recall that the standard direct BIEs for the IDP (corresponding to the top left of Table \ref{tab:bies}) are
\eqref{eq:inteq1} and \eqref{eq:inteq2}.
If $g_D\in \HoG$, then $\dnmu\in\LtG$ (by Theorem \ref{thm:Necas}), and then the mapping properties \eqref{eq:map1} of $S$ and $D$  imply that both sides of \eqref{eq:inteq1} are in $\HoG$. Taking the surface gradient, $\nT$, of \eqref{eq:inteq1} yields the (vector) integral equation in $(\LtG)^d$
\beq\label{eq:inteq3}
\ntS(\dnmu)= \nT\left(\half I + D\right) \gmu.
\eeq
Taking $(\vfb\cdot\bn)$ times the scalar equation \eqref{eq:inteq2}, minus $\vfb$ dot the vector equation \eqref{eq:inteq3}, plus $\alpha$ times \eqref{eq:inteq1} yields \eqref{eq:new_int}. The proof of Part (i) of Theorem \ref{thm:Laplace_ext} (i.e., that \eqref{eq:new_ext} holds) is very similar.

For Part (ii) of Theorem \ref{thm:Laplace_int}, first recall that $\Caldp\phi$ and $\cS\phi$ are both in $C^2(\Oi)$ and satisfy Laplace's equation
(for $\Caldp$ this was recalled in \S\ref{sec:recap}). When $\phi \in \LtG$, $\cS\phi \in H^{3/2}(\Oi)$ by \eqref{eq:LPmap} and then $(\cS\phi)^*\in \LtG$ by Part (iii) of Theorem \ref{thm:JeKe:95}.
As recalled in \S\ref{sec:recap}, $(\Caldp\phi)^* \in \LtG$ by \cite{Ve:84}, and thus $u$ defined by \eqref{eq:indirectansatz1} satisfies $u^*\in \LtG$.
To show that  $\phi$ satisfies the BIE \eqref{eq:new_int_indirect}, we take the non-tangential limit of \eqref{eq:indirectansatz1}, using \eqref{eq:jumpCald} and that, by Lemma \ref{lem:A9}, $\widetilde{\gamma}^- (\cS\phi) = \gamma^-(\cS\phi)$, where $\gamma^-(\cS\phi)$ is given by the first jump relation in
\beq\label{eq:jump1}
\gamma^{\pm}\cS = S, \qquad \partial_n^{\pm}\cS = \mp \half I + D'.
\eeq
(see, e.g., \cite[Page 219]{Mc:00} or \cite[Equation 2.41]{ChGrLaSp:12}).

Part (ii) of Theorem \ref{thm:Laplace_ext} follows in an analogous way, except that we now need to show that $u$ defined by \eqref{eq:indirect_ext_u}
satisfies $u(\bx)= o(|\bx|^{3-d})$ when $d=3$ as $|\bx|\tendi$; these asymptotics follow from
the first bound in \eqref{eq:Sasym} and the bound
\beq\label{eq:Caldasym}
|\Caldp\phi(\bx)| = O(|\bx|^{1-d}) \quad\tas |\bx|\tendi,
\eeq
which is proved in a similar way to the bound on the double-layer potential in \cite[Equation 3.23]{SaSc:11}.

Part (iii) of both Theorems \ref{thm:Laplace_int} and \ref{thm:Laplace_ext} follows from 
combining: (a) the definitions of $\opLIZ$ \eqref{eq:opLIZ} and $\opLEZ$ \eqref{eq:opLEZ} in terms of $\Caldb'$ and $S$; (b) the definitions of $\opLIZi$ \eqref{eq:opLIZ} and $\opLEZi$ \eqref{eq:opLEZ} in terms of $\Caldb$ and $S$; (c) the definitions of $\opLIZRHS$ \eqref{eq:opLIZRHS} and $\opLEZRHS$ \eqref{eq:opLIZRHS} in terms of $D, H,$ and $\nT$; 
(d) the continuity of 
$\Caldb:\LtGt$ (and hence also of $\Caldb':\LtGt$) recalled in \S\ref{sec:recap};
 (e) the continuity of 
 $S: L^2(\Gamma)\to L^2(\Gamma)$, $H: H^1(\Gamma)\to L^2(\Gamma)$, and 
 $D: H^1(\Gamma)\to H^1(\Gamma)$ (and hence also of $\nT D: H^1(\Gamma)\to L^2(\Gamma)$),  recalled in \eqref{eq:map}.
 
\subsection{Proofs of Part (iv) of Theorems  \ref{thm:Laplace_int} and \ref{thm:Laplace_ext} (coercivity up to a compact perturbation).}

\ble\label{lem:Calderon}
If $d\geq 2$, $\Gamma$ is Lipschitz and $\vfb \in (C(\Gamma))^d$ then
$\Caldb + \Caldb'$
is compact in $\LtG$. Thus
there exists a compact operator $ C:\LtG\rightarrow \LtG$ such that
\beqs
\big( K_\vfb \phi,\phi\big)_{\LtG}= \big(C \phi,\phi\big)_{\LtG} \quad\mbox{for all real-valued } \phi\in \LtG.
\eeqs
\ele

Part (iv) of both Theorems \ref{thm:Laplace_int} and \ref{thm:Laplace_ext} follow by combining Lemma \ref{lem:Calderon} with the assumption \eqref{eq:c} and the fact that $S$ is compact on $\LtG$ (via the mapping property in \eqref{eq:map1} with $s=1/2$).

\

\bpf[Proof of Lemma \ref{lem:Calderon}]
Since $\Phi(\bx,\by)$ is a function of $|\bx-\by|$, $\nabla_\bx \Phi(\bx,\by)= - \nabla_\by\Phi(\bx,\by)$; the definitions of $K_\vfb$ \eqref{eq:Caldb} and $K'_\vfb$ \eqref{eq:Caldb'} then imply that, for all $\phi\in \LtG$,
\beq\label{eq:K3}
\big( K_\vfb + K_{\vfb}'\big)\phi(\bx)= \int_\Gamma \big(\vfb(\by)-\vfb(\bx)\big)\cdot\nabla_\by \Phi(\bx,\by) \phi(\by)\, \rd s(\by).
\eeq
If $\vfb \in (C^{0,\beta}(\Gamma))^d$ for $\beta>0$, then the kernel of the integral on the right-hand side of \eqref{eq:K3} is weakly singular, and thus the operator is compact on $\LtG$ by, e.g., the combination of \cite[Part 3 of the theorem on Page 49]{Pr:91} and
 the Riesz-Thorin interpolation theorem (see, e.g., \cite[Theorem 6.27]{Fo:84a}), where the latter is used to verify the hypothesis of the former. Therefore, the result of this lemma follows if we can show that if, for all $\beta>0$, $K_{\vfb} + K'_{\vfb}$ is compact for all $\vfb \in (C^{0,\beta}(\Gamma))^d$, then $K_{\vfb} + K'_{\vfb}$ is compact for all $\vfb \in (C(\Gamma))^d$.

Given $\vfb \in (C(\Gamma))^d$, there exist $\beta>0$ and $\vfb_\ell \in (C^{0,\beta}(\Gamma))^d$ for all $\ell\in \NN$ such that $\|\vfb_\ell-\vfb\|_{L^\infty}\tendo$ as $\ell \tendi$.
By \eqref{eq:Caldb'}, the operator $K_{\bZ}'$ can be written $K_\vfb' = \vfb \cdot \bT$, where $\bT:L^2(\Gamma)\to (L^2(\Gamma))^d$ is bounded by the results of \cite{CoMcMe:82} and \cite{Ve:84} (as discussed in \S\ref{sec:recap}).
Let $K_{\vfb_\ell}' = \vfb_\ell\cdot \bT$; then
\beqs
\N{K_{\vfb_\ell}' - K_{\vfb}'}_{L^2(\Gamma)\to \LtG}= \N{\vfb_\ell \cdot \bT - \vfb\cdot \bT}_{\LtG\to\LtG} \leq \N{\vfb_\ell -\vfb}_{(L^\infty(\Gamma))^d} \N{\bT}_{\LtG\to\LtG}\to 0
\eeqs
as $\ell\to\infty$.
Therefore also $K_{\vfb_\ell} \to K_{\vfb}$, so that $K_{\vfb_\ell} + K_{\vfb_\ell}' \to K_{\vfb}+ K_{\vfb}'$. Since the space of compact operators is closed, $K_{\vfb}+ K_{\vfb}'$ is compact.
\epf

\subsection{Proof of Theorem \ref{thm:star} (coercivity for $\Oi$ that are  star-shaped with respect  to a ball)}

Theorem \ref{thm:star} is an immediate consequence of combining (i) the following special case of Lemma \ref{lem:Calderon}, (ii) the definitions of $\opLIZ$ and $\opLIZi$ in \eqref{eq:opLIZ} and
$\opLEZ$ and $\opLEZi$ in \eqref{eq:opLEZ}, and (iii) the inequality $(S\phi,\phi)_{\LtG}\geq 0$ for all $\phi \in \LtG$. The inequality in (iii) is well-known, following from Green's identity, and is a special case of Lemma \ref{lem:key} below with $\vfd ={\bf 0}$.

\ble[Key lemma for coercivity for star-shaped $\Oi$]\label{lem:Calderon2}
Let $\Gamma$ be Lipschitz. If $d\geq 3$ then
\beq\label{eq:symmetry3d}
K_\bx + K_\bx'+ (d-2)S=0 \,\,\text{and thus}\,\, \big(K_\bx \phi,\phi\big)_{\LtG} + \frac{d-2}{2}\big(S\phi,\phi\big)_\LtG=0\,\,\mbox{for all real-valued } \phi \in\LtG.
\eeq
If $d=2$ then
\beq\label{eq:symmetry2d}
K_\bx + K_\bx'+ \frac{|\Gamma|}{2\pi}\projP =0 \quad\text{and thus}
\,\,\left( \left( K_\bx + \frac{|\Gamma|}{4\pi}\projP\right) \phi,\phi\right)_{\LtG}= 0 \,\, \mbox{for all real-valued } \phi \in \LtG,
\eeq
where $\projP$ is defined by \eqref{eq:P}.
\ele

\bpf[Proof of Lemma \ref{lem:Calderon2}]
By \eqref{eq:fund}, when $d\geq 3$, $(\by-\bx)\cdot \nabla_\by \Phi(\bx,\by) = -(d-2)\Phi(\bx,\by)$, and when $d=2$, $(\by-\bx)\cdot \nabla_\by \Phi(\bx,\by) =-1/2\pi$. The results then follow from \eqref{eq:K3} with $\vfb(\bx)=\bx$.
\epf

\bre[Link with the work of Fabes, Sand, and Seo \cite{FaSaSe:92}]
The analogue of \eqref{eq:symmetry3d} when $\Gamma$ is the graph of a function (i.e., the boundary of a hypograph) appears in the first sentence after the first displayed equation on \cite[Page 133]{FaSaSe:92}.
Indeed, the analogue of the operator $\Caldb'$ for the hypograph with $\vfb = {\bf e}_d$ (i.e., the unit vector pointing in the $x_d$ direction) arises in \cite{FaSaSe:92} when they apply the Rellich identity \eqref{eq:id1} below with $u= \cS \phi$, as part of their proof that $\lambda I - D'$ is invertible on $L^2(\Gamma)$ for $\lambda \in \Rea$ with $|\lambda|\geq 1/2$.
\ere

\subsection{Proof of Part (vi) of Theorems \ref{thm:Laplace_int} and \ref{thm:Laplace_ext} (coercivity for general $\Oi$)}

\begin{lemma}[Key lemma for coercivity for general $\Oi$]\label{lem:key}
Suppose that $\Oi \subset \Rea^d$ is Lipschitz, $\vfd \in W^{1,\infty}(\Rea^d)$ with compact support, and $\alpha\in \Rea$ satisfies the lower bound
\beq\label{eq:alpha_bound}
2\alpha \geq 2\left(\sup_{\bx \in \Rea^d} \big\|D\vfd(\bx)
\big\|_2 \right)+ \big\|\nabla \cdot \vfd\big\|_{L^\infty(\Rea^d)}
\eeq
(where $D\vfd$ is the matrix with $(i,j)$th element $\partial_i \widetilde{Z}_j$ and $\|\cdot\|_2$ denotes the operator norm on $\Rea^d\times \Rea^d$ induced by the Euclidean norm on $\Rea^d$).
If $d\geq 3$ then
\beq\label{eq:key}
\pm
\big( K_{\vfd} \phi,\phi)_{\LtG}
+ \alpha
(S \phi,\phi)_{\LtG} \geq 0
\eeq
for all real-valued $\phi\in\LtG$. If $d=2$, then \eqref{eq:key} holds for all real-valued $\phi\in \LtG$ with $\projP\phi =0$, where $\projP$ is defined by \eqref{eq:P}.
\end{lemma}

We first show how the coercivity results of Theorems \ref{thm:Laplace_int} and \ref{thm:Laplace_ext} are a consequence of Lemma \ref{lem:key} combined with the following lemma.

\ble\label{lem:Lipschitz}
Given $\vfb\in (C^{0,1}(\Gamma))^d$ with non-zero Lipschitz constant, there exists a compactly supported $\vfextt\in (C^{0,1}(\R^d))^d$ with the same Lipschitz constant as $\vfb$ and such that $\vfextt|_\Gamma = \vfb$.
\ele

The proof of Lemma \ref{lem:Lipschitz} is given in  Appendix \ref{app:Lipschitz}. 
\esnote{to edit when cut A} 
Note that, by the Kirszbraun theorem \cite{Ki:34}, \cite{Va:45},
$\vfb\in (C^{0,1}(\Gamma))^d$ can be extended to a function $\vfext\in (C^{0,1}(\R^d))^d$ with the same (non-zero) Lipschitz constant, so to prove Lemma \ref{lem:Lipschitz} we only need to show that there exists an extension with compact support.

\

\bpf[Proof of Part (vi) of Theorems \ref{thm:Laplace_int} and \ref{thm:Laplace_ext} assuming Lemmas \ref{lem:key} and \ref{lem:Lipschitz}]
Given $\vfb$, by Lemma \ref{lem:Lipschitz} there exists a compactly-supported  Lipschitz extension of $\vfb$ to $\Rea^d$ with the same Lipschitz constant; call this $\vfd$.
This $\vfd$ then satisfies the assumptions of Lemma \ref{lem:key}, and
the inequality \eqref{eq:alpha_bound_boundary} then ensures that \eqref{eq:alpha_bound} holds (where we have used the inequality $\|A\|_2^2 \leq \sum_i \sum_j |(A)_{ij}|^2$ to show that $\sup_\bx \|D\vfd(\bx)\|_2\leq d L_{\vfb}$).
 Thus \eqref{eq:key} holds (with $K_{\vfd}$ replaced by $K_{\vfb}$) and the coercivity results follow from the definitions of $\opLIZ$ and $\opLIZi$ \eqref{eq:opLIZ} and
$\opLEZ$ and $\opLEZi$ \eqref{eq:opLEZ} and the inequality \eqref{eq:c} on $\vfb\cdot\bn$.
\epf

\

The proof of Lemma \ref{lem:key} is based on the following identity. The relationship of this identity to other similar identities
in the literature  is discussed in \S\ref{sec:Rellich_recap}, and we note, in particular, that this identity appears as \cite[Equation 2.28]{KuMaVa:02}; for completeness we include the short proof.

\ble[Rellich-type identity]
\label{lem:Rellich}
Let $v$ be a real-valued $C^2$ function on some open set $D\subset \Rea^d$, $d\geq 2$.
Let $\vfd\in (C^1(D))^d$ and $\alpha\in C^1(D)$
and let both be real-valued.
Then, with the summation convention,
\begin{align}
2 \,\cZ v \Delta v  =  \nabla \cdot \Big[ 2\, \cZ v \nabla v  - \ngvs \vfd\Big] - \big(2 \alpha-\nabla \cdot \vfd  \big)\ngvs  - 2  \partial_i \widetilde{Z}_j \partial_i v \partial_j v - 2 v\nabla \alpha\cdot \nabla v,
\label{eq:id1}
\end{align}
where
\beq\label{cZ}
\cZ v := \big(\vfd\cdot \nabla v + \alpha v\big).
\eeq
\ele

\bpf
Splitting $\cZ v$ into its component parts, we see that
the identity \eqref{eq:id1} is the sum of the identities
\begin{align}
2 \,\vfd\cdot\nabla v \Delta v  =  \nabla \cdot \left[ 2\,  (\vfd\cdot \nabla v)\,\nabla v  - \ngvs \vfd\right]  + \big(\nabla \cdot \vfd \big)\ngvs  - 2 \partial_i \widetilde{Z}_j \partial_i v \partial_j v \label{A}
\end{align}
and
\beq\label{C}
2\alpha v \Delta v  = \nabla \cdot \left[ 2 \alpha v\nabla v \right] - 2\alpha \ngvs  - 2v\nabla\alpha \cdot \nabla v.
\eeq
To prove \eqref{C}, expand the divergence on the right-hand side.
The identity \eqref{A} is obtained by combining the identities
\beq\label{basic}
\vfd\cdot \nabla v \Delta v = \nabla \cdot \left[(\vfd\cdot \nabla v) \nabla v\right] - \partial_i \widetilde{Z}_j \partial_i v \partial_j v - \nabla v \cdot  (\vfd\cdot \nabla) \nabla v
\eeq
and
\beq\label{Melenktrick}
2\nabla v \cdot \Big( \vfd\cdot \nabla\Big) \nabla v= \nabla \cdot \left( \ngvs \vfd \right) - (\nabla \cdot \vfd) \ngvs,
\eeq
which can both be proved by expanding the divergences on the right-hand sides.
\epf

\

For the proof of Lemma \ref{lem:key}, we need the identity \eqref{eq:id1} integrated over a Lipschitz domain when $v$ is the single-layer potential. As a step towards this, we prove the following lemma.

\ble[Integrated version of the identity]\label{lem:int} Let $D$ be a Lipschitz domain with outward-pointing unit normal vector $\bnu$.
Define
\beq\label{eq:V}
V(D):=\Big\{v: \,v\in H^1(D), \; \Delta v \in L^2(D),\; \gamma v \in H^1(\partial D),\; \dnuv \in L^2(\partial D)
\Big\}.
\eeq
If $v\in V(D)$, $\vfd\in (W^{1,\infty}(D))^d$ (i.e. $\widetilde{Z}_i$ and $\partial_i \widetilde{Z}_j \in L^\infty(D)$ for $i,j=1,\ldots,n $),  $\alpha \in W^{1,\infty}(D)$, and $v$, $\vfd$, and $\alpha$ are all real-valued,
then
\begin{align}\nonumber
&\int_{\partial D}
\left[(\vfd\cdot \bnu) \left( \left(\dnuv\right)^2  - |\nT v|^2 \right) + 2
\Big(\vfd\cdot \nT (\gamma v) + \alpha (\gamma v)\Big)
\dnuv\right]\rd s\\
&\qquad= \int_D \bigg(2 \,\cZ v  \Delta v +2 \, \partial_i \widetilde{Z}_j \partial_i v \partial_j v + 2 \, v \nabla \alpha\cdot \nabla v
+\left(2 \alpha-\nabla \cdot \vfd  \right)\ngvs\bigg) \rd \bx. \label{eq:int1}
\end{align}
\ele

\noi Recall that, when  $D$ is Lipschitz, we can identify $W^{1,\infty}(D)$ with $C^{0,1}(\overline D)$ (see, e.g., \cite[\S4.2.3, Theorem 5]{EvGa:92}), and understand $\vfd$ and $\alpha$ on $\partial D$ in \eqref{eq:int1} by restriction without needing a trace operator.

\

\bpf[Proof of Lemma \ref{lem:int}]
We first assume that $\vfd$ and $\alpha$ are as in the statement of the theorem, but $v\in \cD(\overline{D}):=\{ U|_\Omega : U \in C^\infty(\Rea^d)\}$.
Recall that the divergence theorem $\int_\domain \nabla \cdot \textbf{F}\,\rd \bx
= \int_{\pD} \textbf{F} \cdot \bnu\,\rd s$
is valid when $\textbf{F} \in (C^1(\overline{\domaingen}))^d$ \cite[Theorem 3.34]{Mc:00}, and thus for $\bF\in (H^1(\domaingen))^d$ by the density of $C^1(\overline{\domaingen})$ in $H^1(\domain)$ \cite[Theorem 3.29]{Mc:00} and the continuity of trace operator from $H^1(\domain)$ to $H^{1/2}(\pD)$ \cite[Theorem 3.37]{Mc:00}.
Recall also that the product of an $H^1(D)$ function and a $W^{1,\infty}(D)$ function is in $H^1(D)$,
and the usual product rule for differentiation holds for such functions. Thus
$\bF= 2 \cZ v \nabla v  - \ngvs \vfd$ is in $(H^1(D))^d$
and then \eqref{eq:id1} implies that $\nabla\cdot\bF$ is given by the integrand on the left-hand side of \eqref{eq:int1}.
Furthermore,
\begin{align*}
\gamma \bF\cdot \bnu =(\vfd\cdot \bnu) \left( \left(\pdiff{v}{\nu}\right)^2  - |\nT v|^2 \right) + 2\,
\big(\vfd\cdot \nT v + \alpha v\big)
\pdiff{v}{\nu}
\end{align*}
on $\partial D$, where we have used the fact that $\gv = \bnu (\partial v/\partial \nu) + \nT v$ on $\partial D$ for $v\in \cD(\overline{D})$; the identity \eqref{eq:int1} then follows from the divergence theorem.

The result for $v\in V(\domaingen)$ then follows from (i) the density of $\DOmegabar$ in $V(\domaingen)$
\cite[Lemmas 2 and 3]{CoDa:98} and (ii) the fact that \eqref{eq:int1} is continuous in $v$ with respect to the topology of $V(\domaingen)$
\epf

\

\bpf[Proof of Lemma \ref{lem:key}]
As discussed in \S\ref{sec:idea}, our strategy is to
mimic the classic method of ``transferring" coercivity properties of the PDE formulation to the BIOs in the trace spaces, but with Green's identity
\beq\label{eq:Green_int}
-\int_D u \Delta u \, \rd \bx = \int_D |\gu|^2\,\rd \bx - \int_{\partial D} \gamma u\, \dnu \, \rd s,
\eeq
replaced by the integrated version of the Rellich-type identity \eqref{eq:id1}. That is, we apply the integrated version of \eqref{eq:id1}, namely \eqref{eq:int1}, with $v$ replaced by $u=\cS\phi$ (with $\phi \in \LtG$),
and $D$ first equal to $\Oi$, and then equal to $\Oe \cap B_R$, where $R> \sup_{\bx\in \Oi}|\bx|$.
At this stage we let $\vfd$ be a general real-valued $W^{1,\infty}(\Rea^3)$ vector field with compact support, and let $\alpha$ be an arbitrary real constant.
That \eqref{eq:int1} holds with $v$ replaced by $u=\cS\phi$, with $\phi$ real-valued, can be  justified by using the results of \cite{JeKe:95} and \cite[Appendix A]{ChGrLaSp:12} recapped in \S\ref{app:HA}.
Indeed,
when $\phi \in \LtG$, $u= \cS\phi \in H^{3/2}(D)$ when $D= \Oi$ or $\Oe\cap B_R$ by the first mapping property in \eqref{eq:LPmap}; then $u\in V(D)$ by Corollary \ref{cor:CoDa} and \eqref{eq:int1} holds by Lemma \ref{lem:int}.
\footnote{A common alternative to justify that \eqref{eq:int1} holds with $v$ replaced by $u=\cS\phi$ is to
(a) approximate $\Oi$ by a sequence of subdomains (often they are assumed to be smoother than $\Omega_-$, but this is not necessary),
(b)  apply the identity \eqref{eq:int1} to each member of the sequence,
(c) pass to the limit using the facts that (i) the non-tangential limits of $u$ and $\gu$ exist, and (ii) the maximal functions $u^*$ defined by \eqref{eq:ntmax} and $(\gu)^*$ (defined analogously) are in $\LtG$;
for examples of this argument, see, e.g.,  \cite[Proof of Theorem 2.1]{Ve:84}, \cite[Proof of Theorem 1]{MeCo:00}, \cite[Proof of Proposition 4.2]{Ta:00}, and \cite[Proof of Lemma 4.5]{SpChGrSm:11}.
 }

We have therefore established that \eqref{eq:int1} holds when $D= \Oi$ or $\Oe\cap B_R$ and $u= \cS\phi$ for $\phi \in\LtG$ that is real-valued.
That is, with the identity \eqref{eq:id1} written as $\nabla \cdot \bQ =P$,
\beq\label{apofmoi}
\int_\Gamma \bQ_-\cdot \bn \,\rd s = \int_{\Oi} P \,\rd \bx
\eeq
and
\beq\label{apofmoe}
-\int_\Gamma \bQ_+\cdot \bn \,\rd s + \int_{\GR} Q_R \,\rd s = \int_{\Oe \cap B_R} P \,\rd \bx,
\eeq
where (remembering that $\Delta u =0$ and $\alpha$ is a constant)
\begin{align}\label{P}
P=  2 \, \partial_i \widetilde{Z}_j \partial_i u \partial_j u + \big(2 \alpha - \nabla \cdot \vfd \big)\ngus,
\end{align}
\beq\label{Qn}
\bQ_{\pm}\cdot \bn = (\vfd \cdot \bn) \left( \big(\partial_n^{\pm}u\big)^2 - |\nT(\gamma^\pm u)|^2\right) + 2\, \Big(\vfd\cdot \nT(\gamma^\pm u) +\alpha \gamma^\pm u\Big)\partial_n^{\pm}u.
\eeq
If $R$ is chosen large enough so that $\supp \,\vfd \subset B_R$, then
\beq\label{QR}
Q_R = \bQ\cdot \widehat{\bx}= 2\alpha \,u\pdiff{u}{r}\quad\tfor \bx\in \GR,
\eeq
 where we have used the fact that $u$ is $C^\infty$ in a neighbourhood of $\GR$ (either by elliptic regularity or directly by the definition of the single-layer potential \eqref{eq:SLP}) to justify writing $\partial u/\partial r$ in place of some appropriate trace.

Adding \eqref{apofmoi} and \eqref{apofmoe} yields
\beqs
\int_\Gamma(\bQ_- - \bQ_+) \cdot \bn \,\rd s + \int_{\GR} Q_R \,\rd s = \int_{\Oi} P \, \rd \bx + \int_{\Oe \cap B_R} P \, \rd \bx.
\eeqs
Now if $d\geq 3$ and $\phi \in\LtG$, then
\beq\label{eq:Sasym}
|\cS\phi(\bx)|= O(|\bx|^{2-d})\quad\tand\quad |\nabla\cS\phi(\bx)| = O(|\bx|^{1-d})
\eeq
as $|\bx|\tendi$, uniformly in all directions $\bx/|\bx|$.
If $d=2$ then
\beq\label{eq:Sasym2}
\cS\phi(\bx)= \frac{1}{2\pi}\log \left(\frac{a}{|\bx|}\right) \left(\phi,1\right)_{\LtG} + O(|\bx|^{-1})\quad\tand\quad \nabla\cS\phi(\bx)  =
-\frac{1}{2\pi |\bx|} \left(\phi,1\right)_{\LtG} + O(|\bx|^{-2})
\eeq
as $|\bx|\tendi$, uniformly in all directions $\bx/|\bx|$;
 these asymptotics are proved for $d=2,3$ in, e.g., \cite[Lemma 6.21]{St:08} (see also \cite[Equations 3.22 and 3.23]{SaSc:11}  for $d=3$); the proof of \eqref{eq:Sasym} for $d\geq 4$ is analogous.
Recalling the definition of $\projP$ \eqref{eq:P} and the assumption that $\projP\phi=0$ when $d=2$, we see that,
by \eqref{QR},
$\int_{\GR}Q_R \,\rd s=O(R^{2-d})$ for $d\geq 3$ and
$\int_{\GR}Q_R \,\rd s=O(R^{-2})$ for $d=2$
as $R\tendi$. Thus, in this limit,
\beq\label{eq:keygamma-1}
\int_\Gamma(\bQ_- - \bQ_+) \cdot \bn \,\rd s  = \int_{\Oi\cup \Oe} P \, \rd \bx.
\eeq
The expressions for $\bQ_\pm \cdot \bn$ \eqref{Qn} and the single-layer potential jump relations
\eqref{eq:jump1}
then imply that
\beq\label{eq:keygamma}
\int_\Gamma(\bQ_- - \bQ_+) \cdot \bn \,\rd s = 2\Big( (\vfd\cdot \bn)D' + \vfd\cdot \ntS+ \alpha S) \phi, \phi\Big)_{\LtG}.
\eeq
A key identity to help one see this is
\begin{eqnarray} \nonumber
\left(\dnmu(\bx)\right)^2-\left(\dnpu(\bx)\right)^2 = 2\,\phi(\bx)\,\big(D^\prime\phi(\bx)\big) \quad \text{ for a.e. }\bx\in\Gamma,\label{jump5}
\end{eqnarray}
which can be established using $a^2 -b^2 = (a-b)(a+b)$ and the jump relations \eqref{eq:jump1} for $\partial^{\pm}_n u$.

Combining \eqref{eq:keygamma-1}, \eqref{eq:keygamma}, and \eqref{eq:P}, we therefore have that
\beq\label{eq:1}
2 \Big( (\vfd\cdot \bn)D' + \vfd\cdot \ntS+ \alpha S) \phi, \phi\Big)_{\LtG} = \int_{\Oi \cup \Oe}\Big( 2 \partial_i \widetilde{Z}_j \partial_i u \partial_j u + \big(2 \alpha - \nabla \cdot \vfd \big)\ngus \Big) \rd \bx.
\eeq
Using the Cauchy-Schwarz inequality and the definition of the matrix 2-norm for the term involving $2 \, \partial_i \widetilde{Z}_j \partial_i u \partial_j u= 2 \,\gu\cdot(D\vfd\, \gu)$, and then standard results about integrals for both this term and the term involving $\nabla\cdot\vfd$, we find that the right-hand side of \eqref{eq:1} is
\beqs
\geq \left( 2\alpha - \left( 2\sup_{\bx \in \Rea^d} \big\|D\vfd\big\|_2 + \big\|\nabla\cdot \vfd\big\|_{L^\infty(\Rea^d)}\right)\right)\int_{\Oi \cup \Oe} |\gu|^2 \, \rd \bx.
\eeqs
Therefore, choosing $\alpha$ to satisfy the lower bound \eqref{eq:alpha_bound} establishes the lemma with the $+$ sign in \eqref{eq:key}. Multiplying \eqref{eq:1} by $-1$ and letting $\alpha \mapsto -\alpha$ we see again that if $\alpha$ satisfies \eqref{eq:alpha_bound} then this modified right-hand side is $\geq 0$, which establishes the lemma with the $-$ sign  in \eqref{eq:key}.
\epf

\bre[Recovering the results of Lemma \ref{lem:Calderon2} for $d\geq 3$]
If $d\geq 3$ and $\vfd=\bx$, \eqref{eq:1} becomes
\beq\label{eq:1a}
2 \Big( (\bx\cdot \bn)D' + \bx\cdot \ntS+ \alpha S) \phi, \phi\Big)_{\LtG} = (2+ 2\alpha-d)\int_{\Oi \cup \Oe}\ngus \, \rd \bx.
\eeq
This is because, despite the additional terms in the analogue of \eqref{QR} coming from $\vfd$ no longer having compact support,
it turns out that
$\int_{\GR}Q_R \,\rd s
=O(R^{2-d})$ as $R\tendi$ as before.
Letting $\alpha=(d-2)/2$ in \eqref{eq:1a} and recalling the definition \eqref{eq:Caldb'} of $\Caldb'$, we obtain the second equality in \eqref{eq:symmetry3d}.
\ere

\bre[Link with Verchota's proof of invertibility of $\half I- D'$ on $\LtG$]\label{rem:Verchota}
Verchota's proof that $\half I-D'$ is invertible on $\LtG$ when $\Gamma$ is Lipschitz in \cite[Theorem 3.1]{Ve:84} relies on the inequalities
\beq\label{eq:Verchota}
\N{\left(\half I - D'\right)\phi}_\LtG \lesssim \N{\left(\half I + D'\right)\phi}_\LtG \lesssim \N{\left(\half I - D'\right)\phi}_\LtG,
\eeq
which hold for all $\phi\in\LtG$ for $d\geq 3$ and for all $\phi \in \LtG$ with $\projP\phi=0$ for $d=2$, and where the omitted constants depend only on the Lipschitz character of $\Oi$. (Note that \cite[Theorem 2.1]{Ve:84} proves the slightly weaker result that
\beqs
\N{\left(\half I \pm D_0'\right)\phi}_\LtG \lesssim \N{\left(\half I \mp D_0'\right)\phi}_\LtG + \left|\int_\Gamma S_0\phi\, \rd s\right|
\eeqs
but the final term on the right-hand side can be eliminated; see \cite[Chapter 15, Corollary 1, Page 273]{MeCo:00} when $\Gamma$ is the graph of a function and \cite[Corollary 2.20]{AmKa:07} for $\Oi$ bounded.)

The inequalities in \eqref{eq:Verchota} can be obtained by applying the following Dirichlet-to-Neumann and Neumann-to-Dirichlet map bounds with $u=\cS \phi$ and using the jump relations \eqref{eq:jump1}.

\noindent (i) If $u\in H^1(\Oi)$ is such that $\Delta u=0$ in $\Oi$, $\gmu\in \HoG$, and $\dnmu\in\LtG$, then
\beq\label{eq:V1}
\N{\nT(\gmu)}_\LtG \lesssim \N{\dnmu}_{\LtG} \lesssim \N{\nT(\gmu)}_\LtG.
\eeq
(ii) If $u\in H^1_{\rm{loc}}(\Oe)$ is such that $\Delta u=0$ in $\Oe$, $\gpu\in \HoG$, $\dnpu\in\LtG$, and $u(\bx)= O(|\bx|^{2-d})$ for $d\geq 3$ and $u(\bx)= O(|\bx|^{-1})$ for $d=2$, then
\beq\label{eq:V2}
\N{\nT(\gpu)}_\LtG \lesssim \N{\dnpu}_{\LtG} \lesssim \N{\nT(\gpu)}_\LtG.
\eeq

\noi The link with our proofs of coercivity of our new BIEs comes from the fact that the bounds \eqref{eq:V1} and \eqref{eq:V2} can be proved using the identity \eqref{eq:id1} with $\alpha=0$ and $\vfd$ the vector field of Lemma \ref{lem:Lipschitz}; see, e.g., \cite[Corollary 2.20]{AmKa:07}.
\ere

\subsection{Proofs of Theorems \ref{thm:2d} and \ref{thm:2dstar} (the 2d results)}

\ble\label{lem:Neumann_trace}
If $u$ is the solution of the IDP then $\projP(\dnmu)=0$. If $u$ is the solution of the EDP and $d=2$, then $\projP(\dnpu)=0$.
\ele

\bpf
The result for the IDP follows from applying Green's second identity to $u$ and the constant function.
The result for the EDP when $d=2$ follows in a similar way, using the arguments in the proof of \cite[Theorem 8.9]{Mc:00} to deal with the integral at infinity. Alternatively, the result for the EDP when $d=2$ is proved in \cite[Proof of Theorem 6.10]{Kr:89}; see \cite[Equation 6.10]{Kr:89}.
\epf

\

\bpf[Proof of Theorem \ref{thm:2d}]
For Parts (i) and (iii), arguing exactly as in the proofs of Part (i) of Theorems \ref{thm:Laplace_int} and \ref{thm:Laplace_ext} gives
\beq\label{eq:England2}
\opLIZ \dnmu =\opLIZRHS g_D \quad\tand\quad \opLEZ\dnpu =\opLEZRHS g_D + \alpha u_\infty,
\eeq
where $u_\infty$ is the limit of the solution of the EDP at infinity
and we use Green's integral representation $u(\bx) = -\cS\dnpu (\bx) + \cD \gpu (\bx) + u_\infty$ for $\bx \in \Oe$ and $d=2$.
The BIEs \eqref{eq:2d_IDP_direct} and \eqref{eq:2d_EDP_direct} then follow by applying $\projQ = I-\projP$ to the equations in \eqref{eq:England2} and then using that $\projP \partial^{\pm}_n u=0$  by Lemma \ref{lem:Neumann_trace}, so that $\partial^{\pm}_n u = \projQ \partial^{\pm}_n u$.

For Part (ii), taking the non-tangential limit of $u$ defined by \eqref{eq:u2d_int} and using the jump relations \eqref{eq:jumpCald} and \eqref{eq:jump1}
(similar to the proof of Part (ii) of Theorem \ref{thm:Laplace_int}) and the fact that $\projQ= I-\projP$, we obtain that $\gamma_-u =g_D$ if the BIE \eqref{eq:2d_IDP_indirect} holds. Exactly as in the analogous proof for $d\geq 3$ in \S\ref{sec:4.1}, $\Caldp\psi$ and $\cS\psi$
with $\psi\in\LtG$
 are in $C^2(\Oi)$, have non-tangential maximal functions in $\LtG$, and satisfy Laplace's equation; therefore $u$ defined by \eqref{eq:u2d_int} inherits these properties.

The proof of Part (iv) is very similar to the proof of Part (ii), except that now need to show that $u$ defined by \eqref{eq:u2d_ext}
satisfies $u(\bx) = O(1)$ as $|\bx|\tendi$; these asymptotics follow from
the first bound in \eqref{eq:Sasym} (since $\projP \projQ \phi=0$) and the bound \eqref{eq:Caldasym}.

To see Part (v), arguing as in the proof of Part (vi) of Theorems \ref{thm:Laplace_int} and \ref{thm:Laplace_ext} below Lemma \ref{lem:Lipschitz}, but using \eqref{eq:key} with $d=2$, we see that $(A\psi,\psi)_{\LtG}\geq (c/2)\|\psi\|^2_{\LtG}$ for all real-valued $\psi\in L^2_0(\Gamma):= \{\phi\in \LtG:P_\Gamma\phi=0\}$ if $\alpha$ satisfies \eqref{eq:alpha_bound_boundary}, where $A$ denotes any of  $\opLIZi$, $\opLIZ$, $\opLEZi$, or $\opLEZ$.
 Part (v) then follows from the fact that if $(A\psi,\psi)_{\LtG}\geq (c/2)\|\psi\|^2_{\LtG}$ for all real-valued $\psi\in L^2_0(\Gamma)$ (so that $A$ is coercive on $L^2_0(\Gamma)$ with coercivity constant $c/2$),
then $\projQ A \projQ + c\projP/2$ is coercive on $\LtG$ with coercivity constant $c/2$. Indeed, since $\projQ'=\projQ$, $\projP^2=\projP$,  $\projP'=\projP$,
and $\projP\projQ=0$,
it follows that, for all real-valued $\psi\in\LtG$, $\projQ\psi\in L^2_0(\Gamma)$ and
\begin{align*}
\left(\left( \projQ A \projQ + \frac{c}{2}\projP \right) \psi,\psi\right)_{\LtG}
&=\big( A \projQ \psi,\projQ\psi\big)_{\LtG}  +\frac{c}{2}\big(\projP^2 \psi,\psi\big)_{\LtG} \\
&\geq \frac{c}{2}\N{\projQ \psi}^2_{\LtG} + \frac{c}{2}\N{\projP\psi}^2_{\LtG}
= \frac{c}{2}\N{\psi}^2_\LtG.
\end{align*}
\epf

\bpf[Proof of Theorem \ref{thm:2dstar}] For Parts (i) and (iii), taking $\vfb=\bx$ and $\alpha=0$ in \eqref{eq:England2} yields
\beqs
A'_{I,\bx,0} \dnpu =B_{I,\bx,0} \,g_D \quad\tand\quad A'_{E,\bx,0} \dnpu =B_{E,\bx,0} \,g_D.
\eeqs
Since $\projP\partial^{\pm}_n u=0$ by Lemma \ref{lem:Neumann_trace}, the BIEs \eqref{eq:2dstar1} and \eqref{eq:2dstar3} follow.

The proofs of Parts (ii) and (iv) follow in the same way as the proofs of Parts (ii) and (iv) of Theorem \ref{thm:2d}, namely by taking non-tangential limits of $u$, using the jump relations \eqref{eq:jumpCald} and \eqref{eq:jump1}, and using the asymptotics \eqref{eq:Caldasym} for the exterior problem.

Part (v) follows immediately from using the second equation in \eqref{eq:symmetry2d}.
\epf

\subsection{Proof of the results in \S\ref{sec:conditioning} (the conditioning results)}\label{sec:conditioning2}

\subsubsection{Proof of Theorem \ref{thm:conditioning}.}\label{sec:conditioning3}

Theorem \ref{thm:conditioning} is a special case of the following general theorem about GMRES applied to Galerkin linear systems of a continuous and coercive operator on a Hilbert space. We first establish some notation.

As in \S\ref{sec:Gal}, we consider the Galerkin method  applied to the equation $A\phi = f$, where $\phi, f \in \cH$, $A:\cH\rightarrow \cH$ is a continuous (i.e.~bounded) linear operator, and $\cH$ is a Hilbert space over $\mathbb{C}$. Let $\cH_N \subset \cH$ be such that $\cH_N = {\rm span} \{\psi^{N}_1,\ldots, \psi^{N}_{M_N}\}$, with $M_N=\dim(\cH_N)$ and $\{\psi^{N}_1,\ldots, \psi^{N}_{M_N}\}$ a basis for $\cH_N$.
The Galerkin matrix of $A$ is then defined by $(\MA)_{ij} :=(A\psi^{N}_j,\psi^{N}_i)_{\cH}$, $i,j=1,\ldots, M_N$ (compare to \eqref{eq:matrix}).

The rest of the set up of \S\ref{sec:conditioning} then holds exactly as stated; i.e., we consider the equation $\MA\bx= {\bf b}$, let $\by_m$ be the $m$th iterate when the linear system \eqref{eq:pclinear}
is solved using GMRES with zero initial guess, let $\br_m$ denote the corresponding residual, and let $\phi_N^{m}$ be defined by \eqref{eq:phiNm}, so that the Galerkin solution $\phi_N$ is given by \eqref{eq:phiN}.

\begin{theorem}\mythmname{Convergence of GMRES applied to the Galerkin linear system of a continuous and coercive operator}\label{thm:conditioning2}
Suppose that $A:\cH\to\cH$ is coercive (i.e., there exists $\Ccoer>0$ such that \eqref{eq:coer} holds) and
Assumption \ref{ass:D} holds with $\|\cdot\|_{\LtG}$ in \eqref{eq:norm_equiv} replaced by $\|\cdot\|_{\cH}$.
With $C_1$ and $C_2$ as in  \eqref{eq:norm_equiv}, let $\beta\in [0,\pi/2)$ be defined such that
\beq\label{eq:cosbeta_gen}
\cos\beta= \frac{\Ccoer}{\|A\|_{\cH\to\cH}}\left(\frac{C_1}{C_2}\right)^{2}\quad\text{ and let } \quad \gamma_\beta:= 2 \sin\left(\frac{\beta}{4-2\beta/\pi}\right)
\eeq
Given $\eps>0$, if
\beq\label{eq:iteration_bound_gen}
m\geq  \left( \log\left(\frac{1}{\gamma_\beta}\right)\right)^{-1}
\left[
\log \left( \frac{12 \|A\|_{\cH\to\cH}}{\Ccoer}\left(\frac{C_2}{C_1}\right)^3\right)+ \log\left(\frac{1}{\eps}\right) \right],
\eeq
then \beq\label{eq:full_discrete_gen}
\frac{
\N{\phi-\phi_N^m}_{\cH}}
{\N{\phi}_{\cH}}
\leq (1+\eps)\frac{\|A\|_{\cH\to\cH}}{\Ccoer}\left( \min_{\psi\in\cH_N}\frac{\N{\phi-\psi}_{\cH}}{\N{\phi}_{\cH}}\right)+ \eps.
\eeq
\end{theorem}

The first step in proving Theorem \ref{thm:conditioning2} is to establish the following relationship between the error $\N{\phi-\phi_N^m}_{\cH}$, the GMRES relative residual $\N{\br_m}_2/\N{\br_0}_2$, and the Galerkin error $\N{\phi-\phi_N}_{\cH}$.

\ble\label{lem:conditioning1}
Suppose that $A:\cH\to\cH$ is coercive (i.e., there exists $\Ccoer>0$ such that \eqref{eq:coer} holds) and
Assumption \ref{ass:D} holds with $\|\cdot\|_{\LtG}$ in \eqref{eq:norm_equiv} replaced by $\|\cdot\|_{\cH}$.
If $C_1$ and $C_2$ are as in  \eqref{eq:norm_equiv} and $\phi_N^m$ is defined by \eqref{eq:phiNm}, then
\beq\label{eq:full_discrete1}
\frac{
\N{\phi-\phi_N^m}_{\cH}}
{\N{\phi}_{\cH}}
\leq
\left( 1 + \frac{\|A\|_{\cH\to\cH}}{\Ccoer}\left(\frac{C_2}{C_1}\right)^3 \frac{\N{\br_m}_2}{\N{\br_0}_2}\right)
\frac{\N{\phi-\phi_N}_{\cH}}{\N{\phi}_{\cH}} +
\frac{\|A\|_{\cH\to\cH}}{\Ccoer}\left(\frac{C_2}{C_1}\right)^3 \frac{\N{\br_m}_2}{\N{\br_0}_2}.
\eeq
\ele

\

The right-hand side of \eqref{eq:full_discrete1} contains the relative residual
$\N{\br_m}_2/\N{\br_0}_2$. The following bound, from \cite{BeGoTy:06}, gives sufficient conditions on $m$ for this relative residual
 to be controllably small; recall that this bound is an improvement of the so-called ``Elman estimate" from  \cite{El:82, EiElSc:83}.

\begin{theorem}\textbf{\emph{(Elman-type estimate for GMRES from \cite{BeGoTy:06})}} \label{thm:Elman}
Let $\MC$ be an $M_N\times M_N$ matrix with $0\notin W(\MC)$, where
\beqs
W(\MC):= \big\{ \langle \MC \bv, \bv\rangle : \bv \in \mathbb{C}^{M_N}, \|\bv\|_2=1\big\}
\eeqs
is the \emph{field of values}, also called the \emph{numerical range}, of $\MC$.
Let $\beta\in [0,\pi/2)$ be such that
\beq\label{eq:cosbeta2}
\cos \beta \leq \frac{\mathrm{dist}\big(0, W(\MC)\big)}{\| \MC\|_2}
\eeq
(observe that $\cos\beta$ is indeed $\leq 1$ by the definition of $W(\MC)$)
and, given $\beta$, let
\beqs
 \gamma_\beta:= 2 \sin \left( \frac{\beta}{4-2\beta/\pi}\right).
\eeqs
Let  $\br_m$  be the $m$th GMRES residual, as defined in \S\ref{sec:conditioning}.
Then
\beq\label{eq:Elman2}
\frac{\|\br_m\|_{2}}{\|\br_0\|_{2}} \leq \left(2 + \frac{2}{\sqrt{3}}\right)\big(2+ \gamma_\beta\big) \,(\gamma_\beta)^m \leq 12 (\gamma_\beta)^m.
\eeq
\end{theorem}

\

\bpf[Proof of Lemma \ref{lem:conditioning1}]
We first use continuity and coercivity of $A$ to obtain bounds on the norm of $\MD_N^{-1/2}\MA \MD_N^{-1/2}$ and its inverse.
By the definition \eqref{eq:matrix},
\beqs
\big(\MA \bv,\bw\big)_2 = \big(A v_N,w_N\big)_{\cH} \quad\tfa v_N, w_N \in \cH_N.
\eeqs
Using this, along with the norm equivalence \eqref{eq:norm_equiv}, we find that, for all $\bv, \bw\in \Com^{M_N}$,
\beqs
\big|\big(\MA \bv,\bw\big)_2\big| \leq \N{A}_{\cH\to\cH} \N{v_N}_{\cH} \N{w_N}_{\cH}
\leq  \N{A}_{\cH\to\cH}(C_2)^2  \big\|\MD_N^{1/2} \bv\big\|_2 \big\|\MD_N^{1/2} \bw\big\|_2
\eeqs
and
\beqs
\big|\big(\MA \bv,\bv\big)_2\big| \geq \Ccoer \N{v_N}_{\cH}^2 \geq \Ccoer(C_1)^2 \big\|\MD_N^{1/2}\bv\big\|_2^2.
\eeqs
Letting $\widetilde{\bv}= \MD_N^{1/2}\bv$ and $\widetilde{\bw}= \MD_N^{1/2}\bw$, we therefore have that, for all $\widetilde{\bv}, \widetilde{\bw}\in \Com^{M_N}$,
\beq\label{eq:GMRES1}
\big| \big( \MD_N^{-1/2}\MA\MD_N^{-1/2}\widetilde{\bv},\widetilde{\bw}\big)_2 \big| \leq \N{A}_{\cH\to\cH}(C_2)^2  \N{\widetilde{\bv}}_2 \N{\widetilde{\bw}}_2
\eeq
and
\beq\label{eq:GMRES2}
\big| \big( \MD_N^{-1/2}\MA\MD_N^{-1/2}\widetilde{\bv},\widetilde{\bv}\big)_2 \big| \geq \Ccoer  (C_1)^2 \N{\widetilde{\bv}}_2^2.
\eeq
The inequalities \eqref{eq:GMRES1} and \eqref{eq:GMRES2} then imply that
\beq\label{eq:GMRES3}
\big\|\MD_N^{-1/2}\MA \MD_N^{-1/2}\big\|_2 \leq \N{A}_{\cH\to\cH} (C_2)^2 \quad\tand\quad \mathrm{dist}\Big(0, W(\MD_N^{-1/2}\MA \MD_N^{-1/2})\Big) \geq \Ccoer(C_1)^2,
\eeq
with the second inequality and the Lax--Milgram theorem then implying that
\beq\label{eq:GMRES3a}
\big\|\big(\MD_N^{-1/2}\MA \MD_N^{-1/2}\big)^{-1}\big\|_2 \leq \frac{1}{\Ccoer (C_1)^2}.
\eeq

We now prove \eqref{eq:full_discrete1}. By the definitions of $\by_m$ (see \eqref{eq:pclinear}) and $\br_m$,
\beqs
\br_m =\MD_N^{-1/2}\MA \MD_N^{-1/2}(\by_m-\by)
\eeqs
and (since $\by_0={\bf 0}$) $\br_0 = -\MD_N^{-1/2}\bbb= -\MD_N^{-1/2}\MA \MD_N^{-1/2}\by$. Therefore, by  \eqref{eq:GMRES3a} and the first bound in \eqref{eq:GMRES3},
\begin{align}\nonumber
\N{\by_m-\by}_2 \leq \big\|\big(\MD_N^{-1/2}\MA \MD_N^{-1/2}\big)^{-1}\big\|_2 \N{\br_m}_2
&\leq \frac{1}{\Ccoer(C_1)^2}
\left(\frac{\N{\br_m}_2}{\N{\br_0}_2}\right)
\N{\br_0}_2 \\ \nonumber
&\leq \frac{1}{\Ccoer(C_1)^2} \left(\frac{\N{\br_m}_2}{\N{\br_0}_2}\right)\big\|\MD_N^{-1/2}\MA \MD_N^{-1/2}\big\|_2 \N{\by}_2\\
&\leq \frac{\N{A}_{\cH\to\cH}}{\Ccoer}\left(\frac{C_2}{C_1}\right)^2 \left(\frac{\N{\br_m}_2}{\N{\br_0}_2}\right)\N{\by}_2.
\label{eq:GMRES4}
\end{align}
Next, the definition of $\phi_N^m$ \eqref{eq:phiNm}, the expression for $\phi_N$ \eqref{eq:phiN}, and the norm equivalence \eqref{eq:norm_equiv} imply that
\beqs
\frac{
\N{\phi_N^m-\phi_N}_{\cH}
}{
\N{\phi_N}_{\cH}
}
\leq
 \frac{C_2}{C_1}
\frac{
\big\|\MD_N^{1/2}(\MD_N^{-1/2}(\by_m-\by))\big\|_2
}{
\big\|\MD_N^{1/2}\MD_N^{-1/2}\by\big\|_2
}
=
 \frac{C_2}{C_1}
\frac{
\big\|\by_m-\by\big\|_2
}{
\big\|\by\big\|_2
},
\eeqs
and then combining this with \eqref{eq:GMRES4} we obtain
\beqs
\frac{
\N{\phi_N^m-\phi_N}_{\cH}
}{
\N{\phi_N}_{\cH}
}
\leq \frac{\N{A}_{\cH\to\cH}}{\Ccoer}\left(\frac{C_2}{C_1}\right)^3 \left(\frac{\N{\br_m}_2}{\N{\br_0}_2}\right).
\eeqs
Combining this last inequality with the triangle inequality, we obtain that
\begin{align*}
\N{\phi-\phi_N^m}_{\cH} &\leq \N{\phi-\phi_N}_{\cH} + \N{\phi_N-\phi_N^m}_{\cH},\\
& \leq \N{\phi-\phi_N}_{\cH}+  \frac{\N{A}_{\cH\to\cH}}{\Ccoer}\left(\frac{C_2}{C_1}\right)^3 \left(\frac{\N{\br_m}_2}{\N{\br_0}_2}\right)\N{\phi_N}_{\cH},
\end{align*}
and then the result \eqref{eq:full_discrete1} follows by another use of the triangle inequality.
\epf

\

\bpf[Proof of Theorem \ref{thm:conditioning2}]
By Part (c) of Theorem \ref{thm:Galerkin}, the Galerkin error $\N{\phi-\phi_N}_{\cH}$ satisfies the quasioptimal error estimate \eqref{eq:quasioptimal}.
The definition of $\beta$ in \eqref{eq:cosbeta_gen} and the bounds on $\MD_N^{-1/2}\MA \MD_N^{-1/2}$ in \eqref{eq:GMRES3} imply that \eqref{eq:cosbeta2} is satisfied with $\MC= \MD_N^{-1/2}\MA \MD_N^{-1/2}$; note that here it is important that $\cH$ is a Hilbert space over $\mathbb{C}$, so that continuity and coercivity of $A$ control $W(\MA)$ (which involves $\MA$ applied to vectors in $\mathbb{C}^{M_N}$).

Using both \eqref{eq:quasioptimal} and
the relative-residual bound \eqref{eq:Elman2}
in \eqref{eq:full_discrete1}, we obtain that
\begin{align*}
\frac{
\N{\phi-\phi_N^m}_{\cH}}
{\N{\phi}_{\cH}}
&\leq
\left( 1 + \frac{12\N{A}_{\cH\to\cH}}{\Ccoer}\left(\frac{C_2}{C_1}\right)^3
(\gamma_\beta)^m\right)
\frac{\|A\|_{\cH\to\cH}}{\Ccoer}\
\min_{\psi\in \cH_N}\frac{\N{\phi-\psi}_{\cH}}{\N{\phi}_{\cH}}
\\
&\qquad\qquad
+
\frac{12\N{A}_{\cH\to\cH}}{\Ccoer}\left(\frac{C_2}{C_1}\right)^3 (\gamma_\beta)^m.
\end{align*}
Given $\eps>0$, if $m$ satisfies
\eqref{eq:iteration_bound_gen}, then
\beqs
\frac{12\N{A}_{\cH\to\cH}}{\Ccoer}\left(\frac{C_2}{C_1}\right)^3 (\gamma_\beta)^m \leq \eps
\eeqs
and thus the bound \eqref{eq:full_discrete_gen} holds.
\epf

\subsubsection{Conditions under which Assumption \ref{ass:D} holds.}\label{sec:conditioning2a}

Our result about the convergence of GMRES applied to the Galerkin matrices of the new formulations, namely Theorem \ref{thm:conditioning}, is proved under Assumption \ref{ass:D}, which is an assumption about the sequence of finite-dimensional subspaces $(\cH_N)_{N=1}^\infty$ and their associated bases. Recall from \S\ref{sec:conditioning} that Assumption \ref{ass:D} holds, indeed with $\MD_N$ the identity matrix, for any sequence $(\cH_N)_{N=1}^\infty$ (and in any dimension $d\geq 2$) provided that the bases we choose are orthonormal. But many standard implementations of boundary element approximation methods use non-orthogonal bases, particularly bases of so-called nodal basis functions (e.g., \cite{AiMcTr:99,GrMc:06}, \cite[Page 216]{St:08}, \cite[Pages 205 and 280]{SaSc:11}.
We show as Lemma \ref{lem:meshes} below that Assumption \ref{ass:D} holds (moreover specifying the diagonal matrices $\MD_N$) under mild constraints on the sequence of meshes  when the approximation space allows discontinuities across elements. In particular, Lemma \ref{lem:meshes} holds when nodal basis functions are used, including for  sequences of highly anisotropic meshes.

To specify the conditions under which Assumption \ref{ass:D} holds, we recall the notion of a surface mesh on $\Gamma$, and aspects of the standard implementation of boundary element methods, including the notation of a reference element (for the moment, until we indicate otherwise, our results hold for any dimension $d\geq 2$). Following, e.g., \cite[Defn.~4.1.2]{SaSc:11}, we call $\mathcal{G}$ a {\em mesh} of $\Gamma$ if $\mathcal{G}$ is a set of finitely many disjoint, relatively open, topologically regular\footnote{By {\em topologically regular} we mean that the relative interior of the closure of $\tau$ is again $\tau$.} subsets of $\Gamma$ that cover $\Gamma$ in the sense that $\Gamma= \cup_{\tau\in \mathcal{G}} \overline{\tau}$, and are such that the relative boundary of each $\tau\in \mathcal{G}$ has zero surface measure. We call the elements of $\mathcal{G}$ the {\em (boundary) elements} of the mesh and, for $\tau\in \mathcal{G}$, set
$$
h_\tau:= \diam(\tau) \quad \mbox{and} \quad  s_\tau := |\tau|,
$$
where $|\tau|$ denotes the $(d-1)$-dimensional surface measure of $\tau$, and set $h:= \max_{\tau\in \mathcal{G}} h_\tau$.

We assume moreover that, for each $\tau\in \mathcal{G}$, there exists a mapping $\chi_\tau:\widehat{\tau}\to \tau$, for some $\widehat{\tau}\in \mathcal{R}$, the finite set of {\em reference elements}, that is bijective and at least bi-Lipschitz, so that $\chi_\tau^{-1}:\tau\to \widehat{\tau}$ is well-defined and also bi-Lipschitz.\footnote{It is standard (e.g., \cite{AiMcTr:99,GrMc:06}) to assume more smoothness for $\chi_\tau$, e.g.~that $\chi_\tau \in C^r(\overline{\widehat{\tau}})$ for some $r\in \NN$, in which case also $\chi_\tau^{-1}\in C^r(\overline{\tau})$. In the important case when $d=3$, $\Gamma$ is a polyhedron, and $\mathcal{R}=\{\widehat{\tau}\}$, with $\widehat{\tau}$ the unit simplex (a triangle), it is usual (e.g., \cite[Chap.~10]{St:08}, \cite[Defn.~4.1.2]{SaSc:11}) for each $\tau$ to be a triangle and for $\chi_\tau$ to be affine.} Here, by a {\em reference element}, $\widehat{\tau}$, we mean, generically, some bounded, open, topologically regular subset of $\R^{d-1}$, but with the idea that, in practical implementations, $\widehat{\tau}$ is a polyhedron, usually the unit cube $\widehat{\tau} = (0,1)^{d-1}$ or the unit simplex $\widehat{\tau} = \{\widehat{\bx}\in (0,1)^{d-1}:\widehat{\bx}_1+\ldots + \widehat{\bx}_{d-1} < 1\}$.
(In the case $d=2$ it is usual to take $\mathcal{R}=\{\hat\tau\}$ with $\widehat{\tau} = (0,1)$.) For each $\tau\in \mathcal{G}$ let $J_\tau\in (L^\infty(\Gamma))^{d\times (d-1)}$ denote the Jacobian of $\chi_\tau$. Importantly, for every $f\in L^1(\tau)$, where $\widehat{\tau} \in \mathcal{R}$ is the domain of $\chi_\tau$,
\begin{equation} \label{eq:cov}
\int_\tau f(\bx) \,\rd s(\bx) = \int_{\widehat{\tau}} f(\chi_\tau(\widehat{\bx}))g_\tau(\widehat{\bx})\, \rd \widehat{\bx}, \quad \mbox{where} \quad g_\tau:= \big(\det(J_\tau^TJ_\tau)\big)^{1/2}\in L^\infty(\Gamma);
\end{equation}
in particular $s_\tau = \int_{\widehat{\tau}} g_\tau(\widehat{\bx})\, \rd \widehat{\bx}$, so that
\begin{equation} \label{eq:gpm}
g^-_\tau |\widehat{\tau}| \leq s_\tau \leq g^+_\tau |\widehat{\tau}|, \quad \mbox{where} \quad g^+_\tau:= \esssup_{\widehat{\bx}\in \widehat{\tau}} g_\tau(\widehat{\bx}), \quad g^-_\tau:= \essinf_{\widehat{\bx}\in \widehat{\tau}} g_\tau(\widehat{\bx}).
\end{equation}

For $p\in \NN_0$ and $\widehat \tau\in \mathcal{R}$ let $\mathbb{P}_p^{\widehat \tau}$ denote some finite-dimensional set of polynomials $\psi:\widehat{\tau}\to\Rea$ that
contains the polynomials of (total) degree $\leq p$.
When $\widehat{\tau}$ is a simplex one usually takes  $\mathbb{P}_p^{\widehat \tau}$ to be the set of polynomials of total degree $\leq p$; when $\widehat{\tau}$ is a cube one usually takes $\mathbb{P}_p^{\widehat \tau}$ to be the set of polynomials of coordinate degree $\leq p$; see, e.g., \cite[Page 1494, penultimate displayed equation]{GrMc:06}.
Following \cite[Defn.~4.1.17]{SaSc:11}, given a mesh $\mathcal{G}$ on $\Gamma$, define the {\em boundary element approximation space, $S_\mathcal{G}^p$,} of discontinuous piecewise polynomials of degree $\leq p$ on $\mathcal{G}$, by
\begin{equation} \label{eq:SpG}
S^p_\mathcal{G}:= \big\{\psi\in L^\infty(\Gamma):\psi|_\tau \circ \chi_\tau \in \mathbb{P}_p^{\widehat \tau}, \mbox{ for all } \tau\in \mathcal{G}, \mbox{ where }\widehat \tau \mbox{ is the domain of }\chi_\tau\big\}.
\end{equation}
Where $P_{\widehat \tau}=\dim(\mathbb{P}_p^{\widehat \tau})$ and $M=\dim(S^p_{\mathcal{G}})$, we equip $S^p_\mathcal{G}$ with a basis $\{\psi_1,\ldots, \psi_M\}$ constructed as follows.  For each $\widehat{\tau}\in \mathcal{R}$ choose a basis $\{\psi^{\widehat{\tau}}_1, \ldots,\psi^{\widehat{\tau}}_{P_{\widehat \tau}}\}$ for $\mathbb{P}_p^{\widehat \tau}$ (for example, a nodal basis as in \cite{AiMcTr:99,GrMc:06}). For each $\tau\in \mathcal{G}$, where $\widehat{\tau}$ is the domain of $\chi_\tau$, define $\psi^{\tau}_j\in L^\infty(\Gamma)$, for $j=1,\ldots,P_{\widehat \tau}$, by
\beq\label{eq:psij}
\psi^\tau_j(\bx) :=
\begin{cases}
\psi^{\widehat{\tau}}_j\big(\chi_\tau^{-1}(\bx)\big), & \bx\in \tau,\\
  0, & \bx\in \Gamma\setminus \tau.
  \end{cases}
\eeq
Then set
\begin{equation} \label{eq:GlBa}
\big\{\psi_1,\ldots, \psi_M\big\} = \big\{\psi^\tau_j:\tau\in \mathcal{G}, j\in \{1,\ldots,P_{\widehat \tau}\}\big\},
\end{equation}
noting that (see, e.g., \cite[p.~1495]{GrMc:06}) $\{\psi_1,\ldots, \psi_M\}$ is a nodal basis if each $\{\psi^{\widehat{\tau}}_1, \ldots,\psi^{\widehat{\tau}}_{P_{\widehat \tau}}\}$, $\widehat{\tau}\in \mathcal{R}$, is a nodal basis.

Consider now the case that we keep $\mathcal{R}$, $p$, and the bases $\{\psi^{\widehat{\tau}}_1, \ldots,\psi^{\widehat{\tau}}_{P_{\widehat \tau}}\}$, $\widehat{\tau}\in \mathcal{R}$, fixed but use a sequence of meshes $\mathcal{G}_N$, $N\in \NN$, with associated approximation spaces $\cH_N:= S^p_\mathcal{G_N}$ that are such that $h_N:= \max_{\tau\in \mathcal{G}_N} h_\tau\to 0$ as $N\to\infty$, i.e.~we consider the \emph{$h$-version} of the boundary-element method. Lemma \ref{lem:meshes} below applies in this regime under the following assumption on the constants  $g^\pm_\tau$ defined by \eqref{eq:gpm}
(this assumption is the first half, Equation 3.5a, of \cite[Assumption 3.1]{GrMc:06}).

\begin{assumption} \label{ass:mild} There exists a constant $c_1\geq 1$ such that, for every $N\in \NN$ and $\tau\in \mathcal{G}_N$,
\begin{equation} \label{eq:mild}
g^+_\tau \leq c_1 g^-_\tau;
\end{equation}
equivalently, there exists a constant $c_2\geq 1$ such that, for every $N\in \NN$ and $\tau\in \mathcal{G}_N$,
\begin{equation} \label{eq:mild2}
c_2^{-1}s_\tau \leq g_\tau(\widehat{\bx}) \leq c_2 s_\tau \quad \mbox{for almost all } \widehat{\bx}\in \hat\tau.
\end{equation}
\end{assumption}
We make two remarks about Assumption \ref{ass:mild}.

(i) The claimed equivalence of \eqref{eq:mild} and \eqref{eq:mild2}
follows from \eqref{eq:gpm} (precisely, if \eqref{eq:mild} holds then \eqref{eq:mild2} holds with $c_2= c_1\max(|\widehat{\tau}|, |\widehat{\tau}|^{-1})$, and if \eqref{eq:mild2} holds then \eqref{eq:mild} holds with $c_1= c_2^2$).

(ii) Because $\chi_\tau$ is bi-Lipschitz, \eqref{eq:mild} holds for every $\tau\in \mathcal{G}_N$ for some $c_1\geq 1$ (not necessarily independent of $\tau$ and $N$). In particular \eqref{eq:mild} holds with $c_1 =1$  if  each $\chi_\tau$ is affine, so that Assumption \ref{ass:mild} holds in that case (see also the discussion below \cite[Assumption 3.1]{GrMc:06}).

\

For the following lemma, recall that the matrix $\MA$ is given by \eqref{eq:matrix} with $\opG$ equal to one of $\opLIZi$, $\opLIZ$, $\opLEZi$, or $\opLEZ$.

\ble[Conditions under which Assumption \ref{ass:D} holds]\label{lem:meshes}
Suppose that, while keeping $\mathcal{R}$, $p$, and the bases $\{\psi^{\widehat{\tau}}_1, \ldots,\psi^{\widehat{\tau}}_{P_{\widehat \tau}}\}$, $\widehat{\tau}\in \mathcal{R}$, fixed, we use a sequence of meshes $\mathcal{G}_N$, $N\in \NN$, with associated approximation spaces $\cH_N:= S^p_\mathcal{G_N}$ and bases \eqref{eq:GlBa} that are such that $h_N:= \max_{\tau\in \mathcal{G}_N} h_\tau\to 0$ as $N\to\infty$ and Assumption \ref{ass:mild} holds. Then the following is true.

(i) Assumption \ref{ass:D} holds with $(\MD_N)_{ii} := s_{i}$, $i=1,\ldots, M_N$, where $s_i:=s_\tau$ if $\psi_i$ is supported in $\tau$, with $C_1:=c_2^{-1/2}c_{\mathcal{R}}^{-1/2}$, $C_2 := c_2^{1/2}c_{\mathcal{R}}^{1/2}$, where $c_{\mathcal{R}}\geq 1$ depends only on the bases  $\{\psi^{\widehat{\tau}}_1, \ldots,\psi^{\widehat{\tau}}_{P_{\widehat{\tau}}}\}$, $\widehat{\tau}\in \mathcal{R}$.

(ii) If, in addition, $d\geq 3$, \eqref{eq:c} holds for some $c>0$, and $\alpha$ satisfies \eqref{eq:alpha_bound_boundary}, then Assumption \ref{ass:D} holds also with $(\MD_N)_{ii} := |\MA_{ii}|$, for $i=1,\ldots, M_N$, with
\begin{equation} \label{eq:C1C2}
C_1 := C_-C^{-1/2} c_2^{-1}\quad\tand\quad C_2 := C_+c^{-1/2}c_2,
\end{equation}
where $C:=\|\opG\|_{L^2(\Gamma)\to L^2(\Gamma)}$ and $C_\pm>0$ depend only on the bases  $\{\psi^{\widehat{\tau}}_1, \ldots,\psi^{\widehat{\tau}}_{P_{\widehat \tau}}\}$, $\widehat{\tau}\in \mathcal{R}$.
\ele

Part (ii) of Lemma \ref{lem:meshes} is proved for $d=3$ using the coercivity results of Part (vi) of both Theorems \ref{thm:Laplace_int} and \ref{thm:Laplace_ext}. An analogous result holds for $d=2$ using the coercivity results of Part (v) of Theorem \ref{thm:2d}, but we omit this for brevity.

\

\bpf[Proof of Lemma \ref{lem:meshes}]
(i) Given $w_N = \sum_{\ell=1}^{M_N}w^N_{\ell} \psi_{\ell}^N \in \cH_N$,
by \eqref{eq:GlBa},
\beq\label{eq:expand1}
w_N= \sum_{\tau\in \mathcal{G}_N}\sum_{j=1}^{P_{\widehat \tau}} w_{j}^\tau \psi_{j}^\tau
\eeq
for some coefficients $\{w_j^\tau\}_{j=1}^{P_{\widehat \tau}}$, and where $\widehat{\tau}$ is the domain of $\chi_\tau$. Thus, by \eqref{eq:cov}, for all $\tau \in \mathcal{G}_N$,
\beq\label{eq:L2wN}
\big\|w_N|_\tau\big\|^2_{L^2(\tau)} = \int_{\widehat{\tau}} |\widehat{\psi}(\widehat{\bx})|^2 g_\tau(\widehat{\bx})\, \rd \widehat{\bx},
\eeq
where
\beq\label{eq:psihat}
\widehat{\psi}(\widehat{\bx}) := w_N\big(\chi_\tau(\widehat{\bx})\big) = \sum_{j=1}^{P_{\widehat \tau}} w_j^\tau \psi^{\widehat{\tau}}_j(\widehat{\bx}) \quad\tfor\widehat{\bx}\in \widehat{\tau}
\eeq
(and we have used \eqref{eq:expand1} and \eqref{eq:psij}).
For every $\widehat{\tau}\in \mathcal{R}$, every $\widehat{\phi}\in \mathbb{P}_p^{\widehat \tau}$ can be written as $\widehat{\phi}= \sum_{j=1}^{P_{\widehat \tau}} a_j\psi^{\widehat{\tau}}_j$ for some unique vector $\ba=(a_1,\ldots,a_{P_{\widehat \tau}})^T$. With $\mathbb{P}_p^{\widehat \tau}$ equipped with the norm $\|\cdot\|_{\widehat{\tau}}$ defined by
\beq\label{eq:normtau}
\|\widehat{\phi}\|^2_{\widehat{\tau}} = \sum_{j=1}^{P_{\widehat \tau}} a_j^2,
\eeq
since $\mathbb{P}_p^{\widehat \tau}$ is finite-dimensional,   there exists $c_{\widehat{\tau}}\geq 1$ such that
\beq\label{eq:ctau}
c_{\widehat{\tau}}^{-1}\big\|\widehat{\phi}\big\|^2_{\widehat{\tau}}\leq \big\|\widehat{\phi}\big\|^2_{L^2(\widehat{\tau})}\leq c_{\widehat{\tau}}\big\|\widehat{\phi}\big\|^2_{\widehat{\tau}} \quad\tfa \widehat{\phi}\in \mathbb{P}_p^{\widehat \tau}.
\eeq
Therefore, using \eqref{eq:mild2} and \eqref{eq:ctau} (with $\widehat{\phi}=\widehat{\psi}$) in \eqref{eq:L2wN}, we have that, for all $\tau \in \mathcal{G}_N$,
\beq\label{eq:temp4}
c^{-1}_{\widehat{\tau}} c^{-1}_2s_\tau\big\|\widehat{\psi}\big\|^2_{\widehat{\tau}}\leq c_2^{-1}s_\tau\big\|\widehat{\psi}\big\|^2_{L^2(\widehat{\tau})} \leq \big\|w_N|_\tau\big\|^2_{L^2(\tau)} \leq c_2s_\tau\big\|\widehat{\psi}\big\|^2_{L^2(\widehat{\tau})} \leq c_2s_\tau c_{\widehat{\tau}}\big\|\widehat{\psi}\big\|^2_{\widehat{\tau}}.
\eeq
Furthermore, by \eqref{eq:psihat} and \eqref{eq:normtau},
\beq\label{eq:temp5}
\|\widehat{\psi}\|^2_{\widehat{\tau}} = \sum_{j=1}^{P_{\widehat \tau}} (w_j^\tau)^2.
\eeq
If $(\MD_N)_{ii} := s_{i}$, $i=1,\ldots, M_N$, then, by \eqref{eq:expand1},
\beq\label{eq:temp6}
\big\|\MD_N^{1/2} \bw\big\|^2_{2} = \sum_{\tau\in \mathcal{G}_N}s_\tau \sum_{j=1}^{P_{\widehat \tau}} (w_j^\tau)^2.
\eeq
Therefore, combining \eqref{eq:temp4}, \eqref{eq:temp5}, and \eqref{eq:temp6}, we see that, with the choice $(\MD_N)_{ii} := s_{i}$, $i=1,\ldots, M_N$, Assumption \ref{ass:D} holds with $C_1=c^{-1/2}_2c_{\mathcal{R}}^{-1/2}$ and $C_2=c_2^{1/2}c^{1/2}_{\mathcal{R}}$, where $c_{\mathcal{R}}:= \max_{\widehat{\tau} \in \mathcal{R}} c_{\widehat{\tau}}$.

(ii)
By the coercivity and continuity of $\opG$ from Theorem \ref{thm:Laplace_int} or Theorem \ref{thm:Laplace_ext},
$$
\frac{c}{2}\|\psi^{N}_i\|^2_{L^2(\Gamma)}\leq \big|\big(\opG \psi^{N}_i,\psi^{N}_i\big)_{L^2(\Gamma)}\big| \leq C\|\psi^{N}_i\|^2_{L^2(\Gamma)};
$$
therefore, since $|\MA_{ii}|=|(\opG \psi^{N}_i,\psi^{N}_i)_{L^2(\Gamma)}|$,
\beq\label{eq:temp2}
\frac{c}{2}\|\psi^{N}_i\|^2_{L^2(\Gamma)}\leq |\MA_{ii}|\leq C\|\psi^{N}_i\|^2_{L^2(\Gamma)}.
\eeq
Further, where $\tau$ is the support of $\psi^N_i$ and $\widehat{\tau}$ is the domain of $\chi_\tau$, by \eqref{eq:cov},
$$
\|\psi^{N}_i\|^2_{L^2(\Gamma)} = \int_{\tau} |\psi^\tau_j(\bx)|^2 \, \rd s(\bx) = \int_{\widehat{\tau}} |\psi^{\widehat{\tau}}_j(\widehat{\bx})|^2 g_\tau(\widehat{\bx})\, \rd \widehat{\bx},
$$
for some $j\in \{1,...,P_{\widehat \tau}\}$, so that, by \eqref{eq:mild2},
\beq\label{eq:temp3}
c_- c^{-1}_2 s_\tau\leq \|\psi^{N}_i\|^2_{L^2(\Gamma)}\leq c_+ c_2 s_\tau,
\eeq
 where
$$
c_+ := \max_{\widehat{\tau} \in \mathcal{R}, j=1,\ldots,P_{\widehat \tau}} \|\psi^{\widehat{\tau}}_j\|_{L^2(\widehat{\tau})}^2 \quad\tand\quad c_- := \min_{\widehat{\tau} \in \mathcal{R}, j=1,\ldots,P_{\widehat \tau}} \|\psi^{\widehat{\tau}}_j\|_{L^2(\widehat{\tau})}^2.
$$
Thus, combining \eqref{eq:temp2} and \eqref{eq:temp3}, we have that
$$
\frac{c}{2}c_- c^{-1}_2 s_\tau\leq |\MA_{ii}|\leq Cc_+ c_2 s_\tau \quad \mbox{ for } i=1,\ldots, M_N.
$$
Therefore, if $(\MD_N)_{ii} := |\MA_{ii}|$, for $i=1,\ldots, M_N$, then, by \eqref{eq:expand1},
$$
 \frac{c}{2}c_-c_2^{-1}\sum_{\tau\in \mathcal{G}_N}s_\tau \sum_{j=1}^{P_{\widehat \tau}} (w_j^\tau)^2 \leq \big\|\MD_N^{1/2} \bw\big\|^2_{2} \leq Cc_+c_2 \sum_{\tau\in \mathcal{G}_N}s_\tau \sum_{j=1}^{P_{\widehat \tau}} (w_j^\tau)^2.
$$
By \eqref{eq:temp5} and \eqref{eq:temp4}, Assumption \ref{ass:D} therefore holds with $C_1$ and $C_2$ given by \eqref{eq:C1C2}.
\epf

\bre[The novelty of Lemma \ref{lem:meshes}]
Similar results to Lemma \ref{lem:meshes} are given in \cite{GrMc:06}, where, for a continuous, coercive, and symmetric sesquilinear form, scaling by the diagonal part of the Galerkin matrix (as in Part (ii) of Lemma \ref{lem:meshes}) is used to remove the ill-conditioning of the Galerkin matrix due to mesh degeneracy; see \cite[Equations 1.5-1.7]{GrMc:06}.

The advantage of the results of \cite{GrMc:06} compared to those of Lemma \ref{lem:meshes} is that \cite{GrMc:06} works in $H^s(\Gamma)$ for $|s|\leq 1$, whereas Lemma \ref{lem:meshes} only works in $L^2(\Gamma)$. However, Lemma \ref{lem:meshes} works with rather general meshes in arbitrary dimensions, subject only to Assumption \ref{ass:mild}, whereas \cite{GrMc:06} imposes the following conditions on the mesh: (i)
the mesh is regular in the sense of \cite[Definition 4.1.4]{SaSc:11}, see \cite[Page1495]{GrMc:06}, and (ii) the mesh satisfies
\cite[Assumptions 3.1 and 3.2]{GrMc:06}, with the latter requiring, e.g., that neighbouring mesh elements have comparable aspect ratios.
\ere

\section{Wellposedness and regularity results for the Laplace interior and exterior oblique  Robin problems}\label{sec:oblique}

\subsection{Statement of the Laplace interior and exterior oblique Robin problems}

\begin{definition}[The Laplace interior oblique Robin problem (IORP)]\label{prob:ioipL}
With $\Oi$
as in \S\ref{sec:notation}, given $g\in L^2(\Gamma)$, $\bZ\in (L^\infty(\Gamma))^{d}$, and $\alpha\in L^\infty(\Gamma)$, find $u\in H^1(\Oi)$ with $\gmu\in H^1(\Gamma)$ and $\dnmu\in \LtG$
such that
$\Delta u =0$ in $\Oi$ and
\beq\label{eq:obliqueL}
(\bZ\cdot\bn) \dnmu + \bZ \cdot \nabla_\Gamma( \gmu) + \alpha \, \gmu =  g \quad\text{ on } \Gamma.
\eeq
\end{definition}

\begin{definition}[The Laplace exterior oblique Robin problem (EORP)]\label{prob:eoipL}
With $\Oe$
as in \S\ref{sec:notation}, given $g\in L^2(\Gamma)$, $\bZ\in (L^\infty(\Gamma))^{d}$, and $\alpha\in \L^\infty(\Gamma)$, find $u\in H^1_{\mathrm{loc}}(\Oe)$ with $\gpu \in H^1(\Gamma)$
and $\dnpu\in \LtG$
 such that
$\Delta u =0$ in $\Oe$,
\beq\label{eq:obliqueLext}
(\bZ\cdot\bn) \dnpu + \bZ \cdot \nabla_\Gamma( \gpu) - \alpha \, \gpu =  g \quad\text{ on } \Gamma,
\eeq
and,  as $|\bx| \rightarrow \infty$, $u(\bx)= O(1)$ when $d=2$ and $u(\bx) = o(|\bx|^{3-d})$ when $d\geq 3$ (uniformly in all directions $\bx/|\bx|$).
\end{definition}

A regularity result of Ne\v{c}as \cite{Ne:67} (stated as Theorem \ref{thm:Necas} below) implies that either of the requirements $\dnmu\in \LtG$ and $\gmu \in \HoG$ in Definition \ref{prob:ioipL} can be removed; similarly in Definition \ref{prob:eoipL}.

The IORP and EORP can also be formulated in terms of non-tangential maximal functions and non-tangential limits
(similar to the case of the Dirichlet problem discussed in \S\ref{sec:notation}).
We now give this alternative formulation for the IORP and prove that it is equivalent to Definition \ref{prob:ioipL}; this equivalence is necessary to use results from the harmonic-analysis literature on the standard Laplace oblique derivative problem (see Theorem \ref{thm:RegOblSt} below).
The alternative formulation for the EORP and proof of equivalence to Definition \ref{prob:eoipL} are completely analogous and are omitted.

\begin{definition}[The Laplace IORP via non-tangential limits]\label{prob:ioipL2}
With $\Oi$
 as in \S\ref{sec:notation}, given $g\in L^2(\Gamma)$, $\bZ\in (L^\infty(\Gamma))^{d}$, and $\alpha\in L^\infty(\Gamma)$,
find $u\in C^2(\Oi)$ with $(\nabla u)^*\in (L^2(\Gamma))^d$ such that
$\Delta u =0$ in $\Oi$ and
\beq\label{eq:obliqueL2}
\bZ \cdot \widetilde{\gamma}^- (\nabla u ) + \alpha \widetilde{\gamma}^-u = g
 \quad\text{ on } \Gamma,
\eeq
where $\widetilde{\gamma}^-$ is the non-tangential limit defined by \eqref{eq:ntlim}.
\end{definition}

\begin{theorem}[Equivalence of the different formulations of the IORP]\label{thm:BVPequiv2}
The formulations of the IORP in Definition \ref{prob:ioipL} and \ref{prob:ioipL2} are equivalent (i.e., if $u$ is a solution to the IORP in the sense of Definition \ref{prob:ioipL}, then it is a solution in the sense of Definition \ref{prob:ioipL2}, and vice versa).
\end{theorem}

\bpf
If $u$ is a solution of the IORP in the sense of Definition \ref{prob:ioipL}, then $u\in C^\infty(\Oi)$ by elliptic regularity. Furthermore, $u\in H^{3/2}(\Oi)$ by Lemma \ref{lem:A10}, and then $(\nabla u)^* \in L^2(\Gamma)$ by Part (iii) of Theorem \ref{thm:JeKe:95}. By Lemma \ref{lem:A9}, $\widetilde{\gamma}^-u = \gmu$, and, by Lemma \ref{lem:A10},
\beq\label{eq:moderna1}
\widetilde{\gamma}^-(\nabla u) = \bn \dnmu + \nabla_\Gamma(\gmu) \quad\text{ almost everywhere on } \Gamma.
\eeq
Therefore
\beq\label{eq:moderna2}
\vfb\cdot \widetilde{\gamma}^- (\nabla u) + \alpha \widetilde{\gamma}^- u = (\vfb\cdot \bn) \dnmu + \vfb \cdot\nabla_\Gamma (\gmu) + \alpha \gmu,
\eeq
so that the boundary condition \eqref{eq:obliqueL2} is equivalent to \eqref{eq:obliqueL}; therefore, $u$ is a solution of the IORP in the sense of Definition \ref{prob:ioipL2}.

Conversely, if $u$ is the solution of the IORP in the sense of Definition \ref{prob:ioipL2}, then $u\in H^{3/2}(\Oi)$ by Part (iii) of Theorem \ref{thm:JeKe:95}. Then
Lemma \ref{lem:A10} implies that $\dnmu \in \LtG$, $\gmu\in H^1(\Gamma)$,
and \eqref{eq:moderna1} holds. Hence \eqref{eq:moderna2} holds and the boundary condition \eqref{eq:obliqueL} is equivalent to \eqref{eq:obliqueL2}; therefore, $u$ is a solution of the IORP in the sense of Definition \ref{prob:ioipL}.
\epf

\subsection{Link between the IORP/EORP and the BIEs in Theorems \ref{thm:Laplace_int}, \ref{thm:Laplace_ext}}

\begin{theorem}[$A^\prime_{I,\bZ,\alpha}$ can be used to solve the EORP for $d\geq 3$] \label{thm:star_in_obliqueLE} If $d\geq 3$ then the single-layer potential $u={\cal S} \phi$ with density $\phi\in L^2(\Gamma)$ satisfies the exterior oblique Robin problem (Definition \ref{prob:eoipL})  if and only if
\begin{equation} \label{eq:BIE_eoipL}
A^\prime_{I,\bZ,\alpha}\phi = -g.
\end{equation}
Conversely, if $d\geq 3$ and $u$ satisfies the EORP, then $u={\cal S} \phi$ for some $\phi\in L^2(\Gamma)$ that satisfies \eqref{eq:BIE_eoipL}.
\end{theorem}

\begin{theorem}[$A^\prime_{E,\bZ,\alpha}$ can be used to solve the IORP for $d\geq 2$] \label{thm:star_in_obliqueL} The single-layer potential $u={\cal S} \phi$, with density $\phi\in L^2(\Gamma)$, satisfies the IORP (Definition \ref{prob:ioipL}) if and only if
\begin{equation} \label{eq:BIE_ioipL}
A^\prime_{E,\bZ,\alpha}\phi = g.
\end{equation}
Conversely, if $u$ satisfies the IORP, then, provided $a\neq \mathrm{Cap}_\Gamma$ when $d=2$, $u={\cal S} \phi$, where $\phi\in L^2(\Gamma)$ satisfies \eqref{eq:BIE_ioipL}.
\end{theorem}

\begin{proof}[Proof of Theorem \ref{thm:star_in_obliqueLE}]  If $d \geq 3$ and $u={\cal S} \phi$ with $\phi\in L^2(\Gamma)$, then by, e.g., \cite[Theorem 2.14]{ChGrLaSp:12} $u\in C^2(\Oe)$ and $\Delta u =0$ in $\Oe$, and, by \eqref{eq:Sasym}, $u(\bx)=O(|\bx|^{2-d})$ as $|\bx|\to\infty$, uniformly in $\bx/|\bx|$.
By, e.g., \cite[Theorem 2.14]{ChGrLaSp:12}, $u\in H^1_{\mathrm{loc}}(\Oe)$ and, by the jump relations \eqref{eq:jump1} and the definition of $\Caldb'$ \eqref{eq:Caldb'}, \eqref{eq:obliqueLext} holds if and only if $\phi$ satisfies \eqref{eq:BIE_eoipL}. Conversely, if $u$ satisfies the EORP, then, by the invertibility of $S$ recalled in Lemma \ref{lem:SLP} below,
  $\phi:= S^{-1} \gpu  \in L^2(\Gamma)$. Defining $v:=\cS \phi$, $v$ satisfies the Laplace exterior Dirichlet problem with boundary data $\gpv=\gamma^+\cS \phi = S\phi = \gpu$, so that $v=u$ by uniqueness for the EDP. As established in the first part of the proof, since $u$ satisfies the EORP, $\phi$ satisfies \eqref{eq:BIE_eoipL}.
\end{proof}

\

\bpf[Proof of Theorem \ref{thm:star_in_obliqueL}]
This is very similar to the proof of Theorem \ref{thm:star_in_obliqueLE}, except that now we can also consider $d=2$, since (by definition) there are no conditions at infinity imposed on the solution of the IORP.
\epf

\begin{theorem}\label{thm:inverse_formula_L}
Let $P_{\rm DtN}^\pm :H^1(\Gamma)\rightarrow \LtG$ denote the Dirichlet-to-Neumann maps for Laplace's equation in $\Omega^\pm$; i.e., the maps $g_D\mapsto \partial_n^\pm u$ for $u$ as in Definitions \ref{def:idp}/\ref{def:edp} respectively.
Let $P_{\rm ItD}^{-, \alpha, \vfb}: \LtG\rightarrow \HoG$ denote the map $g\mapsto \gmu$ where $u$ is as in Definition \ref{prob:ioipL}.
Let $P_{\rm ItD}^{+, \alpha, \vfb}: \LtG\rightarrow \HoG$ denote the map $g\rightarrow \gpu$ where $u$ is as in Definition \ref{prob:eoipL}.
Then, as operators on $\LtG$,
\beq\label{eq:inverse1}
\big( \opLEZ \big)^{-1} = \frac{1}{\vfb\cdot \bn}I - \left( P^+_{\rm DtN} + \frac{1}{ \vfb\cdot \bn}\big( \alpha +\vfb\cdot\nT \big) \right)
P^{-, \alpha, \vfb}_{\rm ItD}
\eeq
and
\beq\label{eq:inverse2}
\big( \opLIZ \big)^{-1} = \frac{1}{\vfb\cdot \bn}I - \left( P^-_{\rm DtN} + \frac{1}{ \vfb\cdot \bn}\big( -\alpha +\vfb\cdot\nT \big) \right)
P^{+, \alpha, \vfb}_{\rm ItD}.
\eeq
\end{theorem}

\bpf
We first prove \eqref{eq:inverse1}. Suppose $\opLEZ \phi =g$ with $\phi,g\in \LtG$ and let $u:= \cS\phi$. Then $\gpu=\gmu= P^{-,\alpha,\vfb}_{\rm ItD} g$ by the first jump relation in \eqref{eq:jump1} and Theorem \ref{thm:star_in_obliqueL}. By the second jump relation in \eqref{eq:jump1}, the definition of $P_{\rm DtN}^+$, and the boundary condition \eqref{eq:obliqueL},
\begin{align*}
\phi &= \dnmu - \dnpu, \\
&= \frac{1}{\vfb\cdot\bn} \big( g - \vfb \cdot \nT(\gmu) - \alpha \gmu\big) - P_{\rm DtN}^+ \gpu
= \frac{1}{\vfb\cdot\bn}g - \left( P_{\rm DtN}^+ + \frac{\alpha + \vfb\cdot\nT}{\vfb \cdot\bn}\right) P_{\rm ItD}^{-,\alpha,\vfb}g,
\end{align*}
which implies \eqref{eq:inverse1}.
The proof of \eqref{eq:inverse2} is then very similar, using Theorem \ref{thm:star_in_obliqueLE} instead of
 Theorem \ref{thm:star_in_obliqueL}.
\epf

\subsection{Statement of the wellposedness results and implications for the BIEs in Theorems \ref{thm:Laplace_int} and \ref{thm:Laplace_ext}} \label{subsec:oblique3}

\begin{theorem}[Uniqueness for the IORP] \label{thm:oblique:unique} Suppose that, for some $\beta\in (0,1]$, $\bZ\in (C^{0,\beta}(\Gamma))^d$ and $\alpha\in C^{0,\beta}(\Gamma)$ and that, for some constants $c, c_0>0$,
\begin{equation} \label{eq:lowerbounds}
\bZ(\bx)\cdot \bn(\bx) \geq c  \quad \mbox{for almost every } \bx\in \Gamma \quad \mbox{and} \quad \alpha(\bx) \geq c_0 \quad\tfor \bx\in \Gamma.
\end{equation}
Then the IORP has at most one solution.
\end{theorem}

\begin{corollary}[Existence for the IORP and invertibility of $\opLEZ$]\label{cor:oblique:unique}
If the assumptions of Theorem \ref{thm:oblique:unique} hold and $a\neq \mathrm{Cap}_\Gamma$ when $d=2$, then
 $\opLEZ$ is invertible and the IORP has exactly one solution.
\end{corollary}

\begin{theorem}[Uniqueness for the EORP] \label{thm:oblique:uniqueE} Suppose that, for some $\beta\in (0,1]$, $\bZ\in (C^{0,\beta}(\Gamma))^d$ and $\alpha\in C^{0,\beta}(\Gamma)$, and that \eqref{eq:lowerbounds} holds, for some constants $c, c_0>0$.
Then the EORP has at most one solution.
\end{theorem}

\begin{corollary}[Existence for the EORP and invertibility of $\opLIZ$]\label{cor:oblique:unique2}
If the assumptions of Theorem \ref{thm:oblique:uniqueE} hold and $d\geq 3$, then the EORP has exactly one solution and $\opLIZ$ is invertible.
\end{corollary}

\subsection{Proofs of Theorems \ref{thm:oblique:unique} and \ref{thm:oblique:uniqueE}}\label{sec:obliqueproofs}

Recall that, for $1\leq p\leq\infty$, $H^{1,p}(\Gamma):=\{\phi\in L^p(\Gamma):\nabla_\Gamma \phi \in L^p(\Gamma)\}$ is a Banach space with the norm $\|\phi\|_{H^{1,p}(\Gamma)} := \|\phi\|_{L^p(\Gamma)} + \|\nabla_\Gamma \phi\|_{L^p(\Gamma)} $. Note that $H^1(\Gamma)=H^{1,2}(\Gamma)$, with equivalence of norms.

The following result is standard in the theory of potential theory on Lipschitz domains; see, e.g., \cite[Page 203]{Ve:07}.

\begin{lemma}\label{lem:obliqueness}
Suppose that $\bZ\in (C(\Gamma))^d$ and the first of the bounds \eqref{eq:lowerbounds} holds for some $c>0$. Then, for each $\bx\in \Gamma$ there exists $R>0$ and $F\in C^{0,1}(\R^{d-1})$ and a rotated coordinate system $0\tilde x_1...\tilde x_d$, with origin at $\bx$ and with the $\tilde x_d$ axis pointing in the direction $\bZ(\bx)$, such that, where $\tilde y^\prime := (\tilde y_1,...,\tilde y_{d-1})$,
$$
B_R(\bx)\cap \Oe = B_R(\bx)\cap \{\by = (\tilde y^\prime,\tilde y_d):\tilde y_d > F(\tilde y^\prime)\}, \quad B_R(\bx)\cap \Oi = B_R(\bx)\cap \{\by = (\tilde y^\prime,\tilde y_d):\tilde y_d < F(\tilde y^\prime)\}.
$$
\end{lemma}

The following key regularity estimate follows immediately from \cite{Pi:87,Ve:07}.

\begin{theorem}[Regularity for the interior oblique derivative problem.] \label{thm:RegOblSt}

\

Suppose that $\bZ\in (C^{0,\beta}(\Gamma))^d$ for some $\beta\in (0,1]$, the first inequality in \eqref{eq:lowerbounds} holds for some constant $c>0$,
and $u$ satisfies the \emph{Laplace oblique derivative problem} (i.e., the IORP in the special case $\alpha=0$) with data $g$.

(i) If $g\in C^{0,\beta}(\Gamma)$, then
$u\in C^{1,\gamma}(\overline{\Oi})$ for some $\gamma\in (0,\beta]$ depending only on $\Oi$.

(ii) If $g\in L^p(\Gamma)$ with $2\leq p <\infty$, then $(\nabla u)^*\in (L^p(\Gamma))^d$.
\end{theorem}

\begin{proof}
(i) It is known from \cite[Section 4]{Ca:85}, \cite{Pi:87, KePi:88, Ve:07} that
if the first inequality in \eqref{eq:lowerbounds} holds, then the Laplace oblique derivative problem has a solution if and only if $g$ satisfies finitely-many linear conditions (i.e., conditions of the form $(g,\phi_j)_\Gamma=0$, $j=1,...,N$, for some $N\in \NN$ and $\phi_1,...,\phi_N\in L^2(\Gamma)$).
 If $\bZ\in (C^{0,\beta}(\Gamma))^d$ and $u$ is a solution for particular data $g\in C^{0,\beta}(\Gamma)$, the finitely-many linear conditions on $g$ are satisfied, and $u$ can be written as $u=u_P+u_H$ where $u_P$ is the particular solution studied in \cite{Pi:87}, which is shown in \cite[\S3]{Pi:87} to satisfy $u_P
 \in C^{1,\gamma}(\Oi)$ for some $\gamma\in (0,\beta]$ (dependent on $\Oi$),
 and $u_H$ is a solution of the homogeneous oblique derivative problem, which is shown in \cite[Corollary 2.7]{Ve:07} to be constant in $\Oi$.

(ii) This follows from arguing as in (i), but replacing the results of \cite{Pi:87} for H\"older continuous $g$ by those of \cite{Ca:85} for $g\in L^p(\Gamma)$ with $2-\epsilon<p<2+\epsilon$ (for some $\epsilon>0$ dependent on $\Oi$) and \cite{KePi:88} for $g\in L^p(\Gamma)$ with $p>2$ (note that while the results of \cite{Ca:85, KePi:88} only require that $\bZ$ is continuous, \cite[Corollary 2.7]{Ve:07} requires $\bZ$ to be H\"older continuous).
\end{proof}

\begin{theorem}[Regularity for the IORP] \label{thm:RegObl} Suppose that $\bZ\in (C^{0,\beta}(\Gamma))^d$ and $\alpha\in C^{0,\beta}(\Gamma)$  for some $\beta\in (0,1]$, $\bZ$ satisfies the first inequality in \eqref{eq:lowerbounds} for some $c>0$,
and $u$ satisfies the Laplace interior oblique Robin problem with data $g\in C^{0,\beta}(\Gamma)$. Then $u\in C^{1,\gamma}(\overline{\Oi})$ for some $\gamma\in (0,\beta]$.
\end{theorem}
\begin{proof}
Suppose that the conditions of the theorem are satisfied, in particular that $u$ satisfies the IORP with data $g\in C^{0,\beta}(\Gamma)$, for some $\beta\in (0,1]$. Suppose also that $2\leq p <\infty$ and that  $\gmu \in L^p(\Gamma)$. Then since, clearly, $g\in L^p(\Gamma)$, $u$ is a solution of the Laplace oblique derivative problem with data in $L^{p}(\Gamma)$. Therefore, by Part (ii) of Theorem \ref{thm:RegOblSt}, $(\nabla u)^* \in (L^{p}(\Gamma))^d$ and thus $\gmu \in H^{1,p}(\Gamma)$ by Corollary \ref{cor:Lp}. This implies, by the Sobolev embedding theorem \cite[Chapter V, Equations 6 and 4]{Adams}, that, if $p\geq d-1$, then
$\gmu \in L^{q}(\Gamma)$ for all $2\leq q<\infty$, while, if $p<d-1$, then $\gmu \in L^{q}(\Gamma)$ for $2\leq q \leq p_0$ where $1/p_0 = 1/p -1/(d-1)$. Since, certainly, $\gmu \in \LtG$ (as $\gmu\in H^1(\Gamma)$ by definition of the IORP), applying the above argument at most a finite number of times leads to the conclusion that $\gmu \in H^{1,q}(\Gamma)$ for all $2\leq q<\infty$.
But this implies, by the Sobolev embedding theorem \cite[Chapter V, Equation 9]{Adams}, that $\gmu \in C^{0,\beta'}(\Gamma)$ for all $0<\beta'<1$. Thus $u$ is a solution of the Laplace oblique derivative problem with data in $C^{0,\beta}(\Gamma)$, and thus the result that $u\in C^{1,\gamma}(\overline{\Oi})$ for some $\gamma\in (0,\beta]$ follows from Part (i) of Theorem \ref{thm:RegOblSt}.
\end{proof}

\begin{corollary}[Regularity for the EORP] \label{thm:RegOblE} Suppose that $\bZ\in (C^{0,\beta}(\Gamma))^d$ and $\alpha\in C^{0,\beta}(\Gamma)$ for some $\beta\in (0,1]$, $\bZ$ satisfies the first inequality in \eqref{eq:lowerbounds} for some $c>0$, and $u$ satisfies the Laplace exterior oblique Robin problem with data $g\in C^{0,\beta}(\Gamma)$. Then $u\in C^{1,\gamma}(\overline{\Oe\cap B_R})$ for all $R>0$ and for some $\gamma\in (0,\beta]$.
\end{corollary}

\begin{proof}
Since $\Oi$ is bounded, $\Gamma \subset B_r$ for some $r>0$. Suppose that $u$ satisfies the EORP and choose $R_2>R_1>R_0>r$ and $\chi\in C_{\rm comp}^\infty(\R^d)$ with $\chi(\bx)=1$ for $|\bx|\leq R_0$ and $\chi(\bx)=0$ for $|\bx|\geq R_1$. Let $v(\bx):= \chi(\bx) u(\bx)$ for $\bx\in G:=\Oe\cap B_{R_2}$, so that, in particular, $v=u$ in $\Oe\cap B_{R_0}$.
The idea now is to create a solution of an IORP on $G$, and then use the interior regularity result of Theorem \ref{thm:RegObl}.
Since $u$ is harmonic in $\Oe$,
$$
\Delta v = F:= 2\nabla \chi\cdot \nabla u + u \Delta \chi \quad \mbox{in } G.
$$
Since $\chi \in C_{\rm comp}^\infty(\R^d)$ and $u\in C^{\infty}(\Oe)$, $F\in C_{\rm comp}^\infty(G)$. Therefore
$$
\widehat v(\bx) := -\int_{G} \Phi(\bx,\by)F(\by)\, \rd \by, \quad \tfor \bx \in G,
$$
satisfies $\widehat v\in C^2(\overline{G})$ and $\Delta \widehat v = F$ in $G$. Let $w(\bx):= v(\bx)-\widehat v(\bx)$ for $\bx\in G$, and define $\widetilde \bZ\in (C^{0,\beta}(\partial G))^d$ and $\widetilde \alpha\in C^{0,\beta}(\partial G)$ by  $\widetilde \bZ:= - \bZ$ on $\Gamma$, $\widetilde \bZ(\bx):= \bx$, for $\bx\in \partial B_{R_2}$, and $\widetilde \alpha := \alpha$ on $\Gamma$, $\widetilde \alpha := 0$ on $\partial B_{R_2}$. Then $w\in H^1(G)$ with trace $\gamma w\in H^1(\partial G)$, $\Delta w=0$ in $G$, and
\beqs
(\widetilde \bZ\cdot\bn) \partial_\bn w + \widetilde \bZ \cdot \nabla_\Gamma( \gamma w) + \widetilde \alpha \, \gamma w =  \widetilde g \text{ on } \partial G,
\eeqs
where $\bn$ is the unit normal pointing out of $G$ and $\widetilde g\in C^{0,\beta}(\partial G)$ is defined by $\widetilde g := -(\widetilde \bZ\cdot \nabla \widehat v +\widetilde \alpha \widehat v) - g$ on $\Gamma$, and by $\widetilde g := -(\widetilde \bZ\cdot \nabla \widehat v +\widetilde \alpha \widehat v)$ on $\partial B_{R_2}$. Theorem \ref{thm:RegObl} implies that, for some $\gamma\in (0,\beta]$, $w\in C^{1,\gamma}(\overline{G})$, so that $v \in C^{1,\gamma}(\overline{G})$ and $u \in C^{1,\gamma}(\overline{\Oe\cap B_{R_0}})$. Since $u$ is harmonic in $\Oe$, $u \in C^{1,\gamma}(\overline{\Oe\cap B_{R}})$ for every $R>0$.
\end{proof}

\

We can now prove Theorems \ref{thm:oblique:unique} and \ref{thm:oblique:uniqueE} and Corollaries \ref{cor:oblique:unique} and \ref{cor:oblique:unique2}.

\

\begin{proof}[Proof of Theorem \ref{thm:oblique:unique}] Suppose that $u$ satisfies the IORP with $g=0$ and that, without loss of generality, $u$ is real-valued. To show that $u=0$ it is enough to show that $u\leq 0$ in $\Oi$, since this implies, by the same argument applied to $-u$, that also $u\geq 0$, and hence $u=0$. By Theorem \ref{thm:RegObl}, $u\in C^1(\overline{\Oi})$ (indeed $\nabla u$ is H\"older continuous). By the maximum principle, since $u\in C^2(\Oi)\cap C(\overline{\Oi})$ is harmonic in $\Oi$, the maximum value of $u$ in $\overline{\Oi}$ is attained at some point $\bx_0\in\Gamma$. Since $u\in C^1(\overline{\Oi})$ it follows from \eqref{eq:obliqueL} with $g=0$ that
$$
\alpha(\bx_0)u(\bx_0) = -\bZ(\bx_0)\cdot \nabla u(\bx_0)=-\lim_{h\to 0^+}\frac{u(\bx_0)-u(\bx_0-h\bZ(\bx_0))}{h}.
$$
Since $\bZ$ is continuous and $\bZ\cdot \bn\geq c\geq 0$ almost everywhere on $\Gamma$,
$\bx_0-h\bZ(\bx_0) \in \overline \Oi$ for all sufficiently small $h>0$ by Lemma \ref{lem:obliqueness},
so that $\bZ(\bx_0)\cdot \nabla u(\bx_0)\geq 0$ since $\bx_0$ is the global maximum. Since $\alpha(\bx_0) >0$, it follows that $u(\bx_0)\leq 0$, so that $u\leq 0$ in $\Oi$.
\end{proof}

\

\begin{proof}[Proof of Theorem \ref{thm:oblique:uniqueE}] Suppose that $u$ satisfies the EORP with $g=0$ and, without loss of generality, is real-valued. As in the proof of Theorem \ref{thm:oblique:unique}, it is enough to show that $u\leq 0$ in $\Oe$. We recall that, when $d=2$, the condition that
$u$ is bounded on $\Oe$ implies that, for some $u_\infty\in \R$,
$$
u(\bx) = u_\infty + O(|\bx|^{-1}) \quad \mbox{as } |\bx|\to \infty,
$$
uniformly in $\bx/|\bx|$, and that
\begin{equation} \label{eq:average}
u_\infty = \frac{1}{2\pi R}\int_{\partial B_R} u\, \rd s
\end{equation}
if $\Gamma\subset B_R$ \cite[Equation 6.11]{Kr:89}.

By Corollary \ref{thm:RegOblE}, $u\in C^1(\overline {\Oe})$. By the maximum principle, since $u\in C^2(\Oe)\cap C(\overline{\Oe})$ is harmonic in $\Oe$, the maximum value of $u$ in $\overline{\Oe}$ is attained on $\Gamma$ or, when $d=2$, $u(\bx)\leq u_\infty$ for $\bx\in \Oe$.
If the maximum is attained on $\Gamma$, the result that $u\leq 0$ follows by arguing as in the proof of Theorem \ref{thm:oblique:unique}.
Therefore, it is sufficient to prove that the maximum is attained on $\Gamma$ when $d=2$. If $u(\bx)\leq u_\infty$ for $\bx\in \Oe$, then \eqref{eq:average} implies that $u(\bx)= u_\infty$ for $|\bx|\geq R$ if $\Gamma\subset B_R$, so that the maximum is attained in $\Oe$. The maximum principle (see, e.g., \cite[Theorem 6.8]{Kr:89}) then implies that $u$ is constant in $\Oe$, so that the maximum is also attained on $\Gamma$.
\end{proof}

\

The following proofs of Corollaries \ref{cor:oblique:unique} and \ref{cor:oblique:unique2} use the fact that, when $\alpha\in L^\infty(\Gamma)$, $\bZ$ is continuous, and \eqref{eq:c} (i.e., the first lower bound in \eqref{eq:lowerbounds}) holds, then $\opLIZ$ and $\opLEZ$ are Fredholm of index zero by Parts (iii) and (iv) of Theorem \ref{thm:Laplace_int} and \ref{thm:Laplace_ext} respectively. Although these two theorems are for $d=3$, Parts (iii) and (iv) also hold when $d=2$ (as noted at the beginning of \S\ref{sec:d=2}).

\

\bpf[Proof of Corollary \ref{cor:oblique:unique}]
If we can prove invertibility of $\opLEZ$, then existence of a solution to the IORP follows from Theorem \ref{thm:star_in_obliqueLE}. Since $\opLEZ$ is Fredholm of index zero on $\LtG$, by the Fredholm alternative (see, e.g., \cite[Theorem 2.27]{Mc:00}), to prove invertibility it is sufficient to prove injectivity.
Assume that  $\opLEZ\phi=0$ for $\phi\in L^2(\Gamma)$. By Theorem \ref{thm:star_in_obliqueL},
$u:= \cS\phi$ satisfies the IORP, and by Theorem \ref{thm:oblique:unique} $u=0$ in $\Oi$.
Therefore $\gmu=0$ and the first jump relation in \eqref{eq:jump1} implies that $S\phi =0$. Lemma \ref{lem:SLP} then implies that $\phi=0$ and the proof is complete.
\epf

\

\bpf[Proof of Corollary \ref{cor:oblique:unique2}]
This is very similar to that of Corollary \ref{cor:oblique:unique} except that now we only work in $d\geq 3$, since Theorem \ref{thm:star_in_obliqueLE} requires $d\geq 3$.
\epf

\bre[The results of \cite{Li:87a}]
Although not directly used to prove the results in this section, the results of \cite{Li:87a} concern the Laplace IORP in Lipschitz domains with
H\"older continuous $\vfd$ and $g$, and we comment here on their relevance to the results above.

The results of \cite{Li:87a} give an alternative route for obtaining uniqueness of the IORP (i.e., proving Theorem \ref{thm:oblique:unique}). Indeed, in the proof of Theorem \ref{thm:oblique:unique}, once we have established that $u\in C^1(\overline{\Oi})$ (by using Theorem \ref{thm:RegObl}), then uniqueness follows from \cite[Theorem 3.2]{Li:87a}.
The reason we argue as we do in the proof of Theorem \ref{thm:oblique:unique} is that this argument easily carries over to the proof of uniqueness for the EORP (Theorem \ref{thm:oblique:uniqueE}), whereas \cite[Theorem 3.2]{Li:87a} concerns only the IORP.

Furthermore, \cite[Theorem 3.2]{Li:87a} implies that, under the assumptions of Theorem \ref{thm:RegObl}, there exists $\beta_0<1$, depending only on the Lipschitz constant of $\Oi$, such that if $\beta< \beta_0$ then $u \in C^{1,\beta}(\overline{\Oi})$.
\ere

\bre[Additional uniqueness results for the EORP with $\vfb=\bx$]
The coercivity result of Theorem \ref{thm:star} allows us to extend the range of $\alpha$ for which the EORP is unique when $\vfb=\bx$ and $d\geq 3$.
Indeed, Theorem \ref{thm:star} implies that $\opLIZ$ is injective when $\vfb=\bx, \alpha(\bx)\geq -(d-2)/2$ for almost every $\bx\in \Gamma$, and $d\geq 3$. Then, using Theorem \ref{thm:star_in_obliqueLE} and arguing as at the end of the proof of Corollary \ref{cor:oblique:unique}, we see that the solution of the EORP is unique under these conditions.
This result proves uniqueness for certain non-positive values of $\alpha$, which are not covered by Theorem \ref{thm:oblique:uniqueE}.
\ere

\subsection{Link between the IORP/EORP and the BIEs in Theorem \ref{thm:2d}}

\ble\label{lem:2doblique1}
Given $g\in \LtG$, if $\phi$ satisfies
\beq\label{eq:2doblique1}
T_{E,\vfb,\alpha,\beta}' \phi = g,
\eeq
(with $T_{E,\vfb,\alpha,\beta}'$ defined by \eqref{eq:TEprime})
and $d=2$, then
\beq\label{eq:2doblique2}
u = \cS \projQ \phi + \frac{\beta}{\alpha} \projP \phi - \frac{1}{\alpha}\projP \opLEZ \projQ \phi
\eeq
satisfies the IORP.
\ele

\ble\label{lem:2doblique2}
Given $g\in \LtG$, if $\phi$ satisfies
\beq\label{eq:2doblique1a}
T_{I,\vfb,\alpha,\beta}' \phi = -g,
\eeq
(with $T_{I,\vfb,\alpha,\beta}'$ defined by \eqref{eq:TIprime})
and $d=2$, then
\beq\label{eq:2doblique2a}
u = \cS \projQ \phi + \frac{\beta}{\alpha} \projP \phi - \frac{1}{\alpha}\projP \opLIZ \projQ \phi
\eeq
satisfies the EORP.
\ele

\bpf[Proofs of Lemmas \ref{lem:2doblique1}, \ref{lem:2doblique2}]
The fact that $u$ given by \eqref{eq:2doblique2}/\eqref{eq:2doblique2a} is $C^2$ and satisfies Laplace's equation follows from, e.g., \cite[Theorem 2.14]{ChGrLaSp:12}. The condition that $u= O(1)$ at infinity for the EORP follows from the asymptotics \eqref{eq:Sasym2}, the definition of $\projP$ \eqref{eq:P}, and that $\projQ:= I- \projP$.
The BIEs \eqref{eq:2doblique2}/\eqref{eq:2doblique2a} follow from the jump relations \eqref{eq:jump1} and the definitions of $\opLIZ$ \eqref{eq:opLIZ}, $\opLEZ$ \eqref{eq:opLEZ}, and $\Caldb'$ \eqref{eq:Caldb'}.
\end{proof}

\bre[Link with the work of Medkova \cite{Me:18}]\label{rem:Medkova_oblique}
In \cite[Theorem 5.23.5]{Me:18}, the solution of the IORP is sought as \eqref{eq:2doblique2} without the final term on the right-hand side,
resulting in the BIE $(\opLEZ \projQ + \beta\projP)\phi =g$; this BIO is then proved to be invertible on $\LtG$ if $\beta=\alpha$ and $\alpha$ is sufficiently large \cite[Theorem 5.23.5]{Me:18}.
The advantage of including the final term on the right-hand side of \eqref{eq:2doblique2} is that, by Theorem \ref{thm:2d}, the resulting BIO $T_{E,\vfb,\alpha,\beta}' $ is not just invertible when $\alpha$ is sufficiently large, but also coercive by Part (v) of Theorem \ref{thm:2d}.
\ere

\section{New formulations of the Helmholtz exterior Dirichlet problem}\label{sec:Helmholtz}

\subsection{Statement of the new formulations}\label{sec:Helmholtz_state}

As a corollary of Theorem \ref{thm:Laplace_ext}, we obtain results about BIE formulations of the Helmholtz exterior Dirichlet problem. We state the results in this subsection but defer the more substantial proofs to \S\ref{sec:Helmholtz_prove} (and see Remark \ref{rem:earlier}).

\begin{definition}[Helmholtz EDP]\label{def:Hedp}
With $\Oi$ and $\Oe$ as in \S\ref{sec:notation}, assume further that $\Oe$ is connected.
Given $g_D\in H^{1/2}(\Gamma)$, we say $u\in H^1_{\rm{loc}}(\Oe)$ satisfies the Helmholtz exterior Dirichlet problem (EDP) if $\Delta u +k^2u  =0$ in $\Oe$, $\gpu = g_D$ on $\Gamma$, and
 $u$ satisfies the Sommerfeld radiation condition
\beq\label{eq:src}
\pdiff{u}{r}(\bx) - \ri k u(\bx) = o\left( \frac{1}{r^{(d-1)/2}}\right)
\eeq
as $r:= |\bx|\tendi$, uniformly in $\bx/r$.
\end{definition}

Existence and uniqueness of the solution of the Helmholtz EDP for Lipschitz $\Oi$ is shown in, e.g., \cite[Theorem 2.10]{ChGrLaSp:12}.
Theorem \ref{thm:Hedp} below considers $g_D\in \HoG$; this holds, for example, for the \emph{sound-soft scattering problem} when $u$ corresponds to the scattered field and $g_D$ is the Dirichlet trace of the ($C^\infty$) incident field -- see, e.g., \cite[Proof of Theorem 2.12]{ChGrLaSp:12}.

\begin{definition}[Helmholtz EDP formulated via non-tangential limits]\label{def:Hedp2}
With $\Oi$ and $\Oe$ as  in \S\ref{sec:notation}, assume further that $\Oe$ is connected.
Given $g_D\in L^2(\Gamma)$, we say that $u\in C^2(\Oe)$ with $u^* \in L^2(\Gamma)$ satisfies the Helmholtz EDP if $\Delta u +k^2u=0$ in $\Oe$, $\widetilde{\gamma}^+ u = g_D$ on $\Gamma$, and $u$ satisfies the Sommerfeld radiation condition \eqref{eq:src}.
\end{definition}

The solution of the EDP in the sense of Definition \ref{def:Hedp} exists and is unique by \cite[Theorem 5.6, Part (ii)]{ToWe:93}.
The following equivalence result is proved in Appendix \ref{app:BVPequiv}.

\begin{theorem}
\label{thm:BVPequivH}
If $g_D\in H^{1/2}(\Gamma)$, then the solution of the Helmholtz EDP in the sense of Definition \ref{def:edp} is the solution of the Helmholtz EDP in the sense of Definition \ref{def:edp2}, and vice versa.
\end{theorem}

The fundamental solution for the Helmholtz equation is defined by
\beq\label{eq:fundH}
\Phi_k(\bx,\by):=
\frac{\ri}{4} \left( \frac{k}{2\pi\nxy}\right)^{(d-2)/2} H^{(1)}_{(d-2)/2}\big(k\nxy\big);
\eeq
see, e.g., \cite[Equation 5.118]{Sta67}; the definition of the Hankel function $H^{(1)}_{(d-2)/2}$ implies that $\Phi_k(\bx,\by)= (\ri/4)H_0^{(1)}(k\nxy)$ for $d=2$ and $\exp(\ri k \nxy)/ (4\pi\nxy)$ for $d=3$.

The Helmholtz
BIOs $S_k, D_k, D'_k,$ $H_k$, $K_{\vfb,k}$, and $K_{\vfb,k}'$, are defined by \eqref{eq:bio1}, \eqref{eq:bio2}, \eqref{eq:Caldb}, and \eqref{eq:Caldb'}, respectively, with $\Phi$ replaced by $\Phi_k$.
Similarly, the potentials $\cS_k$ and $\cK_{\vfb,k}$ are defined by \eqref{eq:SLP} and \eqref{eq:Caldp}, respectively, with $\Phi$ replaced by $\Phi_k$.

Given $\vfb\in (L^\infty
(\Gamma))^d$ and $\eta \in L^\infty(\Gamma)$,
in analogy with \eqref{eq:opLEZ}, define the integral operators
 $A^\prime_{k,\eta,\vfb}$
 $A_{k,\eta,\vfb}$,
 and $B_{k,\eta,\vfb}$ by
\begin{align} \label{eq:AketaZpr_def}
&A^\prime_{k,\eta,\vfb} :=  \half (\vfb\cdot \bn) I + K_{\vfb,k}' - \ri \eta S_k, \qquad
A_{k,\eta,\vfb} :=  \half (\vfb\cdot \bn) I + K_{\vfb,k} - \ri \eta S_k, \\
&B_{k,\eta,\vfb} := (\vfb\cdot \bn) H_k + \vfb\cdot\nabla_\Gamma \left(- \frac{1}{2}I + D_k\right) - \ri \eta  \left(-{\frac{1}{2}}I+ D_k\right).
\label{eq:BketaZ_def}
\end{align}
Regarding notation, since we are only dealing with the exterior Helmholtz Dirichlet problem, we omit the subscript $E$ present in the analogous operators for Laplace's equation, but add a subscript $k$ to highlight the $k$ dependence. We use the notation $-\ri\eta$ rather than $\alpha$ for consistency with standard notation for Helmholtz BIEs -- see \eqref{eq:opA} below -- and allow $\eta$ to be a function, rather than just a constant, to make a link to existing results -- see Theorem \ref{thm:Hoblique} below.

\begin{theorem}[New integral equations for the Helmholtz EDP]\label{thm:Hedp}

\

(i) \textbf{\emph{Direct formulation.}} Let $u$ be the solution of the  Helmholtz EDP of Definition \ref{def:Hedp} with additionally $g_D\in H^1(\Gamma)$. Then $\dnpu$ satisfies
\beq\label{eq:Hnew_ext}
A^\prime_{k,\eta,\vfb}\dnpu =B_{k,\eta,\vfb} g_D.
\eeq

(ii) \textbf{\emph{Indirect formulation.}} Given $g_D\in \LtG$, if $\phi$ satisifes
\beq\label{eq:Hnew_ext_indirect}
A_{k,\eta,\vfb}\phi = g_D,
\eeq
then
\beqs
u:= (\cK_{\vfb,k} -\ri \eta \cS_k)\phi
\eeqs
is the solution of the Helmholtz EDP of Definition \ref{def:Hedp2}.

  (iii) \textbf{\emph{Continuity.}}
$A^\prime_{k,\eta,\vfb} : \LtG\rightarrow\LtG$,
$A_{k,\eta,\vfb}: \LtG\rightarrow\LtG$,
and $B_{k,\eta,\vfb}:\HoG\rightarrow\LtG$, and these mappings are continuous.

 (iv) \textbf{\emph{Coercivity up to compact perturbation for all $\eta$.}} If $\bZ\in (C(\Gamma))^d$ and there exists $c>0$ such that \eqref{eq:c} holds,
 then, for all $\eta \in L^\infty(\Gamma)$, both $A^\prime_{k,\eta,\vfb}$ and $A_{k,\eta,\vfb}$ are the sum of a coercive operator and a compact operator on $\LtG$.
\end{theorem}

If $\vfb=\bn$ then $A^\prime_{k,\eta,\vfb}=A'_{k,\eta}$, $A_{k,\eta,\vfb}=A_{k,\eta}$, and $B_{k,\eta,\vfb}=B_{k,\eta}$ where
\beq\label{eq:opA}
A'_{k,\eta} := \half I + D'_k - \ri \eta S_k,
\qquad A_{k,\eta} := \half I + D_k - \ri \eta S_k,
\quad\tand\quad B_{k,\eta}:=H_k + \ri \eta \left(\half I - D_k \right)
\eeq
are the operators appearing in the standard ``combined field" or ``combined potential" BIEs for the Helmholtz EDP; see, e.g., \cite[Equations 2.68 and 2.69]{ChGrLaSp:12}.
Recall from \S\ref{sec:openbook} that there exist Lipschitz polyhedra such that each of $A'_{k,\eta}$ and $A_{k,\eta}$ is \emph{not} the sum of a coercive operator and a compact operator on $\LtG$.

By Part (b) of Theorem \ref{thm:Galerkin}, to prove that the Galerkin method converges when applied to either of the BIEs \eqref{eq:Hnew_ext} or \eqref{eq:Hnew_ext_indirect},
it is sufficient to show that  $A^\prime_{k,\eta,\bZ}$ and $A_{k,\eta,\bZ}$ are injective.
Analytic Fredholm theory and the fact that $A^\prime_{k,\eta,\bZ}$ is a compact perturbation of $\opLEZ$ imply that $A^\prime_{k,\eta,\vfb}$ is invertible for all except at most a discrete set of wavenumbers.

\ble\label{lem:analyticFredholm}
If $\vfb$ is Lipschitz, then
there exists a discrete set $D \subset \Com$ (i.e.~each element of $D$ is isolated) such that
$A^\prime_{k,\eta,\vfb}$ is invertible on $\LtG$
  for all $k\in \Com\setminus D$ ($d\geq 3$ and odd) or $k\in (\Com\setminus \Rea^-)\setminus D$ ($d$ even), where $\Rea^-:= (-\infty,0]$.
\ele

Exactly as in the Laplace case, injectivity of $A^\prime_{k,\eta,\bZ}$ and $A_{k,\eta,\bZ}$ is equivalent to uniqueness of the solution of an interior oblique Robin problem.

\begin{definition}[The Helmholtz interior oblique impedance problem]\label{prob:ioip}
Let $\Oi$ be a bounded Lipschitz open set. Given $g\in L^2(\Gamma)$, $\vfb\in (L^\infty(\Gamma))^{d}$, and $\eta\in L^\infty(\Gamma)$, find $u\in H^1(\Oi)$ with $\gmu \in H^1(\Gamma)$ and $\dnmu\in \LtG$ such that
$\Delta u +k^ 2u =0$ in $\Oi$ and
\beq\label{eq:oblique}
(\vfb\cdot\bn) \dnmu + \vfb \cdot \nabla_\Gamma( \gmu) -\ri \eta \, \gmu =  g \text{ on } \Gamma.
\eeq
\end{definition}

\begin{lemma}[$A^\prime_{k,\eta,\vfb}$ can be used to solve the interior oblique impedance problem] \label{thm:star_in_oblique} The single-layer potential $u={\cal S}_k \phi$, with density $\phi\in L^2(\Gamma)$, satisfies the interior oblique impedance problem \eqref{prob:ioip} if
\begin{equation} \label{eq:BIE_ioip1}
A^\prime_{k,\eta,\vfb}\phi = g.
\end{equation}
Conversely, if $u$ satisfies the interior oblique impedance problem, then $u={\cal S}_k \phi$, with $\phi\in L^2(\Gamma)$ satisfying \eqref{eq:BIE_ioip1}.
\end{lemma}

\begin{corollary}\label{cor:star_ioip_equiv}
 For $\eta\in L^\infty(\Gamma)$ and $\vfb\in (L^\infty(\Gamma))^{d}$, the operator $A^\prime_{k,\eta,\vfb}:L^2(\Gamma)\to L^2(\Gamma)$ is injective if and only if the interior oblique impedance problem with $g=0$ has only the trivial solution, and is surjective if and only if the interior oblique impedance problem has a solution for every $g\in L^2(\Gamma)$.
\end{corollary}

Whereas BVPs involving oblique derivatives have been well-studied for Laplace's equation (see the references in Theorem \ref{thm:RegOblSt} and also \cite{Li:13a}, \cite[\S\S5.23, 5.24, 6.19]{Me:18}), there do not appear to be any results in the literature on the unique solvability of the interior oblique impedance problem for the Helmholtz equation for all wavenumbers and general Lipschitz $\Oi$.
The following theorem collects three situations in which the solution of the interior oblique impedance problem is known to be unique, and hence, by the results above, the Galerkin method
applied to either of the BIEs \eqref{eq:Hnew_ext} or \eqref{eq:Hnew_ext_indirect}
is provably convergent
for every asymptotically-dense sequence of subspaces.

\begin{theorem}\label{thm:Hoblique}
The solution of the Helmholtz interior oblique impedance problem is unique in the following situations.

(i) $\Oi$ is Lipschitz and star-shaped with respect to a ball centred (without loss of generality) at the origin,
$k>0$,
$\vfb(\bx)= \bx$, and
\begin{equation} \label{eq:special_eta}
\eta(\bx) = k|\bx| + \ri \alpha \quad\text{ with }\quad \alpha \geq (d-1)/2.
\end{equation}

(ii) $\Oi$ is Lipschitz and is a  finite union of domains as in (i), $k>0$, and, on each connected part of $\Gamma$, where $\bx_j$ is the point from which that connected component of $\Oi$ is star-shaped with respect to a ball, $\vfb(\bx)=\bx-\bx_j$ and
\beqs
\eta(\bx) = k|\bx-\bx_j| + \ri \alpha  \quad\text{ with }\quad \alpha \geq (d-1)/2.
\eeqs

(iii) $\Oi$ is $C^\infty$, $k$ is sufficiently large, $\eta= c_1 k$, where $c_1>0$ is independent of $k$, and $\vfb\in (C^\infty(\Gamma))^d$ is such that, for all $\bx \in \Gamma$,
\beq\label{eq:GLS1}
\vfb(\bx)\cdot\bn(\bx) >0 \quad\tand \quad |\vfb(\bx)\cdot \bt(\bx) - c_1 |>0
\eeq
for all unit tangent vectors $\bt(\bx)$ at the point $\bx$.
\end{theorem}

The following theorem is the Helmholtz analogue of Theorem \ref{thm:inverse_formula_L}.

\begin{theorem}\label{thm:inverse_formula_H}
Assume that the solution of the Helmholtz interior oblique impedance problem of Definition \ref{prob:ioip} exists and is unique, and let $P^{-, \eta, \vfb}_{\rm ItD}: L^2(\Gamma)\rightarrow \HoG$ denote the map that takes $g$ to $\gmu$, where $u$ is as in Definition \ref{prob:ioip}.
Let $P^+_{\rm DtN}: H^1(\Gamma)\rightarrow \LtG$ denote the exterior Dirichlet-to-Neumman map for the Helmholtz exterior Dirichlet problem. Then, as an operator on $\LtG$,
\beq\label{eq:inverse_formula_H}
\big( A'_{k,\eta,\vfb}\big)^{-1} = \frac{1}{\vfb\cdot \bn}I - \left( P^+_{\rm DtN} - \frac{1}{ \vfb\cdot \bn}\big( \ri \eta - \vfb\cdot\nT \big) \right)
P^{-, \eta, \vfb}_{\rm ItD}.
\eeq
\end{theorem}

\bre[The star-combined operator]\label{rem:star}
When $\Oi$ is star-shaped with respect to a ball, one can choose $\vfb$ and $\eta$ so that $A^\prime_{k,\eta,\vfb}$ is coercive (not just coercive up to a compact perturbation). Indeed if
$\vfb(\bx)=\bx$ and $\eta(\bx)$ is given by \eqref{eq:special_eta},
then \cite[Theorem 1.1]{SpChGrSm:11} showed that $A^\prime_{k,\eta,\vfb}$ is coercive for all $k>0$ with coercivity constant given by $\radius$ in \eqref{eq:ssg}
\footnote{Strictly speaking, \cite[Theorem 1.1]{SpChGrSm:11} proves coercivity of $A^\prime_{k,\eta,\vfb}$ with $\eta$ as in \eqref{eq:special_eta} with $\alpha=(d-1)/2$. However, using the second inequality in \cite[Equation 2.8]{ChMo:08} in the proof of \cite[Lemma 2.4]{SpChGrSm:11}, one obtains the result for $\alpha \geq (d-1)/2$.
Furthermore, \cite{SpChGrSm:11} considers only $d=2,3$, but the proofs go through in exactly the same way for $d\geq 4$.
}. This operator was named the \emph{star-combined operator} and given the notation $\starA$. Note that the coercivity of the star-combined operator implies the uniqueness result about the interior oblique impedance problem
in Part (i) of Theorem \eqref{thm:Hoblique}.
\ere

\bre \label{rem:earlier}
Part (iv) of Theorem \ref{thm:Hedp} was stated (without proof) as \cite[Theorem 2.40]{ChGrLaSp:12}, with reference given to the present paper.
Lemma \ref{thm:star_in_oblique}, Corollary \ref{cor:star_ioip_equiv}, and Theorem \ref{thm:inverse_formula_H} were stated (with proof) as \cite[Theorem 2.38]{ChGrLaSp:12}, \cite[Corollary 2.39]{ChGrLaSp:12}, and \cite[Theorem 2.42]{ChGrLaSp:12} but are included here for completeness.
\ere

\subsection{Discussion of the Helmholtz results and related literature}\label{sec:Helm_diss}

\paragraph{The ideas behind the results in \S\ref{sec:Helmholtz_state}.}
The proofs of Parts (i), (ii), and (iii) of Theorem \ref{thm:Hedp} are very similar to their Laplace analogues in Theorem \ref{thm:Laplace_ext}. The proof of Part (iv) (coercivity up to a compact perturbation on $\LtG$)
follows from Part (iv) of Theorem \ref{thm:Laplace_ext} (the analogous result for $\opLEZ$) since the difference $A'_{k,\eta,\vfb}- \opLEZ$ is compact (for any $\alpha\in \Rea$).
Indeed,
\beq\label{eq:Helm_diff}
A^\prime_{k,\eta,\vfb}- \opLEZ= (\vfb\cdot\bn)(D'_k-D') + \vfb \cdot \nT(S_k-S) - \ri \eta S_k - \alpha S.
\eeq
The bounds on $\Phi_k(\bx,\by)- \Phi(\bx,\by)$ in, e.g., \cite[Equation 2.25]{ChGrLaSp:12} imply that $D'_k-D'$ and $\nT(S_k-S)$ map $\LtG$ continuously to $\HoG$, and are thus compact. Furthermore, the mapping property \eqref{eq:map2} holds with $S$ replaced by $S_k$ (again because of properties of $\Phi_k(\bx,\by)- \Phi(\bx,\by)$) and thus $S_k$ is compact on $\LtG$.

The link between $A'_{k,\eta,\vfb}$ and the Helmholtz interior oblique impedance problem expressed in Theorem \ref{thm:star_in_oblique} and Corollary \ref{cor:star_ioip_equiv} is proved in an analogous way to the link between $\opLEZ$ and the Laplace IORP in Theorem \ref{thm:star_in_obliqueL} and Corollary \ref{cor:oblique:unique}.

\paragraph{Discussion of our results in the context of related literature.}

The summary is that there does not yet exist a BIE posed in $\LtG$ for solving the Helmholtz EDP that, for all Lipschitz $\Oi$ and all $k>0$, is bounded, invertible, and the sum of a coercive operator and a compact operator. The standard BIOs $A'_{k,\eta}$ and $A_{k,\eta}$ \eqref{eq:opA} fail to be the sum of a coercive operator and a compact operator on the Lipschitz domains and 3-d star-shaped Lipschitz polyhedra given in \cite{ChSp:22}. The new BIOs $A_{k,\eta,\vfb}'$ and $A_{k,\eta,\vfb}$ have not been proved to be injective for all $k>0$. The formulations of  \cite{BuHi:05, BuSa:07, EnSt:07, EnSt:08} are invertible for all $k>0$ and the sum of a coercive operator and a compact operator, but only from $\HmhG\rightarrow \HhG$. We now give more detail on all these points.

The motivation for considering second-kind Laplace BIEs over first-kind Laplace BIEs is primarily based on the good conditioning of second-kind BIEs.
For Helmholtz BIEs, however, there is an additional complication compared to Laplace BIEs:~the BIEs in the exterior Helmholtz analogues of \eqref{eq:inteq1} and \eqref{eq:inteq2}, namely
\beq\label{eq:HBIE1}
S_k \dnpu = - \left(\half I - D_k\right) g_D \quad\tand\quad \left( \half I + D'_k \right) \dnpu = H_k g_D,
\eeq
are not uniquely solvable for all values of $k$. Indeed,
$S_k$ is not injective when $k^2$ is an Dirichlet eigenvalue of $-\Delta$ in $\Oi$; this is because, from the first equation in \eqref{eq:HBIE1}, the Neumann trace of a Dirichlet eigenfunction of $-\Delta $ in $\Oi$ satisfies $S_k \dnmu=0$.
Similarly, $\half I + D'_k $ is not injective when $k^2$ is a Neumann eigenvalue of $-\Delta$ in $\Oi$; this is because
the operator $\half I + D'_k $ appears in the indirect formulation of the interior Neumann problem (see Table \ref{tab:bies}) and thus cannot be injective when $k^2$ is a Neumann eigenvalue.

The standard remedy (going back to \cite{BrWe:65, Le:65, Pa:65} in the context of indirect BIEs and \cite{BuMi:71} in the context of direct BIEs) is to take a linear combination of the BIEs in \eqref{eq:HBIE1} and observe that, for $\eta\in \Com\setminus\{0\}$, $\dnpu$ satisfies
\beq\label{eq:HBIE2}
A'_{k,\eta}\dnpu = B_{k,\eta} g_D
\eeq
with $A'_{k,\eta}$ and $B_{k,\eta}$ defined in \eqref{eq:opA}.
The mapping properties \eqref{eq:map2} imply that, for $|s|\leq 1/2$,
\beqs
\opA: H^{s-1/2}(\Gamma)\rightarrow H^{s-1/2}(\Gamma) \quad\tand\quad B_{k,\eta}: H^{s+1/2}(\Gamma)\rightarrow H^{s-1/2}(\Gamma),
\eeqs
 and one can then show that, for $k>0$ and $\Re\eta \neq 0$, $\opA$ and $B_{k,\eta}$ are invertible; see, e.g, \cite[Theorem 2.27]{ChGrLaSp:12}.

Since $D'_k-D'$ and $S_k$ are both compact on $\LtG$ when $\Gamma$ is Lipschitz,
$\opA$ is the sum of $\half I + D'$ and a compact operator. Therefore the question of whether or not $\opA$ is the sum of a coercive operator and a compact operator is equivalent to the analogous question for $\half I +D'$. The results recalled in \S\ref{sec:openbook} imply that the answer to this question is yes if $\Gamma$ is $C^1$, $\Oi$ is a 2-d curvilinear polygon with $C^{1,\alpha}$ sides, or $\Gamma$ is a Lipschitz domain with sufficiently-small Lipschitz character; furthermore the answer to this question is no if $\Gamma$ is one of the 2- and 3-d Lipschitz domains or 3-d star-shaped Lipschitz polyhedra described in \cite{ChSp:22} (see \cite[\S1.2.1]{ChSp:22} for more detail).

The lack of a convergence theory for the Galerkin method applied to either $A'_{k,\eta}$ or $A_{k,\eta}$ on $\LtG$
motivated \cite{BuHi:05,BuSa:07, EnSt:07, EnSt:08} to introduce modifications of $A'_{k,\eta}$ and $A_{k,\eta}$.
For direct BIEs, the idea in these papers is to choose an operator $M: \HmhG\rightarrow \HhG$ so that
\beqs
\widetilde{A}_{k,\eta, M}:= M\left(\half I + D'_k\right) - \ri \eta S_k
\eeqs
maps $\HmhG \to \HhG$, is injective, and is the sum of a coercive operator and a compact operator (with this last property following from the facts that $S:\HmhG\rightarrow \HhG$ is coercive and $\widetilde{A}_{k,\eta, M}+\ri \eta S_k$ is compact);
different choices of $M$ were then proposed and analysed in \cite{BuHi:05, BuSa:07, EnSt:07, EnSt:08}, and the corresponding interior boundary value problems proved to have a unique solution (so that the operator $\widetilde{A}_{k,\eta, M}$ is injective).
Convergence of the Galerkin method then follows from Part (b) of Theorem \ref{thm:Galerkin}. However, the BIE involving $\widetilde{A}_{k,\eta, M}$ is now a first-kind equation (in contrast to the second-kind equation \eqref{eq:HBIE2}), and so one expects the conditioning of the Galerkin linear systems to worsen as the meshes are refined (as discussed in \S\ref{sec:rationale}).

\subsection{Proofs of the results in \S\ref{sec:Helmholtz_state}} \label{sec:Helmholtz_prove}

\begin{proof}[Proof of Theorem \ref{thm:Hedp}]
The proof that the Neumann trace $\dnpu$ of the solution of the Helmholtz EDP satisfies the integral equation \eqref{eq:Hnew_ext} is very similar to the
analogous arguments for the Laplace integral equations; see the proofs of Theorems \ref{thm:Laplace_int} and \ref{thm:Laplace_ext} above.
As recalled in \S\ref{sec:Helm_diss}, the difference $A'_{k,\eta,\vfb}- \opLEZ$ is compact;
therefore, Part (iv) of Theorem \ref{thm:Laplace_ext} implies that, for any $\eta\in \Com$, $A^\prime_{k,\eta,\vfb}$ is the sum of a coercive operator and a compact operator.
\end{proof}

\

\bpf[Proof of Lemma \ref{lem:analyticFredholm}]
By the definitions of $\opLEZ$ \eqref{eq:opLEZ} and $A^\prime_{k,\eta,\vfb}$ \eqref{eq:AketaZpr_def},
\begin{align*}
A^\prime_{k,\eta,\vfb}&= \opLEZ + \big(K_{\vfb,k}'- K_\vfb\big) - \ri \eta S_k - \alpha S,\\
&=\opLEZ \Big( I + (\opLEZ)^{-1}\big( K_{\vfb,k}'- K_\vfb - \ri \eta S_k - \alpha S\big)\Big).
\end{align*}
By Part (v) of Theorem \ref{thm:Laplace_ext}, $\opLEZ$ is invertible if $\alpha$ satisfies \eqref{eq:alpha_bound_boundary}. Furthermore, for all $k\in \Com$ ($d\geq 3$ and odd) or $k\in \Com\setminus \Rea^-$ ($d$ even),
$K_{\vfb,k}'- K_\vfb - \ri \eta S_k - \alpha S$ is compact on $\LtG$; indeed, as discussed in the proof of Theorem \ref{thm:Hedp} above,
$S_k, S_0, D_k'-D',$ and $\nabla_\Gamma S_k  -\nabla_\Gamma S$ are all compact operators on $\LtG$. The result then follows from  analytic Fredholm theory; see, e.g., \cite[Theorem 8.26]{CoKr:98}.
\epf

\

\begin{proof}[Proof of Theorem \ref{thm:star_in_oblique}]
The proof is similar to the proofs of Theorems \ref{thm:star_in_obliqueL} and \ref{thm:star_in_obliqueLE}.
If $u={\cal S}_k \phi$ with $\phi\in L^2(\Gamma)$ then, by the jump relations \eqref{eq:jump1}, the fact that $S_k:\LtG\rightarrow \HoG$ by \eqref{eq:map1} and the fact that $\cS_k$ satisfies the same mapping properties as $\cS$ in \eqref{eq:LPmap}, $u$ satisfies the oblique impedance BVP if and only if the boundary condition \eqref{eq:oblique} holds.
However, the jump relations \eqref{eq:jump1} imply that this equation is precisely \eqref{eq:BIE_ioip1}. On the other hand, if $u$ satisfies the oblique impedance problem then
$g_0 := \dnmu - \ri\eta\gmu\in L^2(\Gamma)$. Recall the interior impedance BVP
\beqs
\Delta u +k^2 u =0 \quad\text{ in } \Oi, \qquad\dnmu- \ri \eta \gmu = g \quad\text{ on }\Gamma,
\eeqs
 with $\eta\in\Rea\setminus\{0\}$ and $g=g_0$ has exactly one solution (see, e.g., \cite[Theorem 2.3]{ChGrLaSp:12}). Clearly this solution is $u$. Further, recalling first that $A^\prime_{k,k,\bn}=A^\prime_{k,k}$, which is invertible by \cite[Theorem 2.27]{ChGrLaSp:12}, and second that the interior impedance problem is the special case of the interior oblique impedance problem when $\vfb=\bn$, we see that $u = {\cal S}_k \phi$ with $\phi = (A^\prime_{k,k})^{-1}g_0\in L^2(\Gamma)$ by the first part of this theorem.
\end{proof}

\

\begin{proof}[Proof of Corollary \ref{cor:star_ioip_equiv}]
Surjectivity follows from Lemma~\ref{thm:star_in_oblique}.
Injectivity follows
provided that
$u={\cal S}_k\phi=0$ in $\Oi$ only if $\phi=0$.
Assume that $\phi\in L^2(\Gamma)$ and $u={\cal S}_k\phi$ is zero in $\Oi$. Then, by the jump relations \eqref{eq:jump1}, $u$ satisfies the Helmholtz exterior Dirichlet problem in $\Oe$ with zero Dirichlet data, and hence is zero by uniqueness of the solution of this BVP (see, e.g., \cite[Corollary 2.9]{ChGrLaSp:12}). The jump relations \eqref{eq:jump1} then imply that  $\phi = \partial_n^-u-\partial_n^+u=0$, as required.
\end{proof}

\

\bpf[Proof of Theorem \ref{thm:Hoblique}]

(i) In this situation, $A^\prime_{k,\eta,\vfb}$ is coercive (not just coercive up to a compact perturbation) by \cite[Theorem 1.1]{SpChGrSm:11};
see Remark \ref{rem:star}.

(ii) This is proved in \cite[Theorem 5.6]{Gi:17} by showing that the corresponding $A^\prime_{k,\eta,\vfb}$ is injective;
for related results about this $A^\prime_{k,\eta,\vfb}$ see \cite[Appendix A]{GiChLaMo:21}.

(iii)
Let $\hbar:=k^{-1}$ be the semiclassical parameter.
In the notation of \cite[\S4]{GaLaSp:21}, the boundary condition \eqref{eq:oblique} then becomes
\beqs
\mathcal{N}\left(\frac{1}{\ri} \hbar\partial_n u\right) - \mathcal{D} u = \frac{1}{\ri} \hbar g,
\eeqs
where $\mathcal{N}:= \vfb \cdot \bn$ and $\mathcal{D}= - \vfb \cdot\nabla_\Gamma/\ri + c_1$; see \cite[Equation 4.1]{GaLaSp:21}.
We now check that the hypotheses of \cite[Theorem 4.6]{GaLaSp:21} hold, and then the result then follows from \cite[Theorem 4.6]{GaLaSp:21} (in fact, this theorem proves that, under these conditions, the interior oblique derivative problem is well-posed, and the bound on the solution in terms of the data has the same $k$ dependence as the bound for the standard interior impedance problem).
In the notation of \cite{GaLaSp:21}, $\sigma(\mathcal{N})=  \vfb \cdot \bn$, $\sigma(\mathcal{D})= - \vfb\cdot \xi'+ c_1$, $m_1=0$, and $m_0=1$.
The first/second equation in \cite[Equation 4.2]{GaLaSp:21} holds because of the first/second equation in \eqref{eq:GLS1}, respectively.
Then, \cite[Equations 4.3, 4.10, 4.11, and 4.12]{GaLaSp:21} all hold because $m_0=m_1+1$.
\epf

\

\bpf[Proof of Theorem \ref{thm:inverse_formula_H}]
This is very similar to the proof of Theorem \ref{thm:inverse_formula_L} (the Laplace analogue), using
Lemma \ref{thm:star_in_oblique} in place of  Theorem \ref{thm:star_in_obliqueLE}.
\epf

\begin{appendix}

\section{Recap of mapping properties of layer potentials and boundary integral operators}\label{app:B}

Recall that the single-layer potential $\cS\phi$ is defined by \eqref{eq:SLP}. For $\phi\in L^2(\Gamma)$, define the double-layer potential $\cD\phi$ by
\beq\label{eq:DLP}
\cD\phi(\bx):= \int_\Gamma \pdiff{\Phi(\bx,\by)}{n(\by)}\phi(\by)\, \rd s(\by) \quad\tfor \bx \in \Rea^d\setminus\Gamma.
\eeq
For $\chi\in C^\infty_{\rm comp}(\Rea^d$) and $|s|\leq 1/2$,
\beq\label{eq:LPmap}
\chi \cS : H^{s-1/2}(\Gamma)\rightarrow H^{s+1}(\Rea^d) \quad\tand\quad
\chi \cD : H^{s+1/2}(\Gamma)\rightarrow H^{s+1}(\Omega^{\pm}).
\eeq
With $S,D,D',$ and $H$ defined by \eqref{eq:bio1} and \eqref{eq:bio2},
for all $|s|\leq 1/2$,
\begin{subequations}\label{eq:map}
\begin{align}\label{eq:map1}
S: H^{s-1/2}(\Gamma)\rightarrow H^{s+1/2}(\Gamma), \quad D: H^{s+1/2}(\Gamma)\rightarrow H^{s+1/2}(\Gamma),\\
D': H^{s-1/2}(\Gamma)\rightarrow H^{s-1/2}(\Gamma), \quad H: H^{s+1/2}(\Gamma)\rightarrow H^{s-1/2}(\Gamma).\label{eq:map2}
\end{align}
\end{subequations}
The results in \eqref{eq:LPmap} and \eqref{eq:map} for $|s|=1/2$ (which then imply the results for $|s|<1/2$ by interpolation) are consequences of the results in \cite{CoMcMe:82}, \cite{Ve:84}, and \cite{JeKe:95}; see, e.g., \cite[Theorems 2.15 and 2.16 and Corollary A.8]{ChGrLaSp:12}. (Note that the results in \eqref{eq:LPmap} for $|s|<1/2$ can also be obtained from mapping properties of the Newtonian potential and Green's integral representation, with the results in \eqref{eq:map} then following from results about the trace map; see \cite{Co:88}, \cite[Theorem 6.11]{Mc:00}.)

\begin{lemma}[Invertibility of $S:\LtG\rightarrow \HoG$ when $\Gamma$ is Lipschitz]
 \label{lem:SLP}
If $\Gamma$ is Lipschitz and either $d=3$, or $d=2$ and $a\neq\mathrm{Cap}_\Gamma$, then $S:L^2(\Gamma)\to H^{1}(\Gamma)$ is bounded and invertible.
\end{lemma}

\bpf[References for the proof]
The boundedness is \eqref{eq:map1} above with $s=1/2$.
The invertibility is proved in  \cite[Theorem 5.1]{Ve:84} for $d=3$ and \cite[Theorem 4.11]{Ve:84} for $d=2$.
Note that \cite{Ve:84} assumes for simplicity that $\Oi$ and $\Gamma$ are connected, but it is clear that this implies that the result holds when $\Gamma$ is the boundary of any bounded Lipschitz open set
(with this result for $d=3$ contained in \cite[Theorem 4.1]{MitreaD:97}). Indeed, in this case, $\Gamma$ and $\R^d\setminus \Gamma$ each have finitely-many connected components, and the results of \cite{Ve:84} for the case when $\Gamma$ is connected imply that $S$ is Fredholm of index zero as an operator $L^2(\Gamma)\mapsto H^{1}(\Gamma)$. Further, $S:L^2(\Gamma)\to H^1(\Gamma)$ is injective since $S$ is invertible as an operator from $H^{-1/2}(\Gamma)$ to $H^{1/2}(\Gamma)$ \cite[Corollary 8.13, Theorem 8.16]{Mc:00}.
\epf

\section{Recap of harmonic-analysis results}
\label{app:HA}

In this appendix we recap results on the behaviour of solutions to Laplace's or Poisson's equation near the boundary of the domain. For simplicity, these results are stated for a bounded Lipschitz domain $D$ with boundary $\partial D$. Analogues of the results then hold with $D=\Oi$ and $D=\Oe$, where in the latter case spaces such as $H^1(D)$ become $H^1_{\rm loc}(\Oe)$ (since these results do not assume any particular behaviour at infinity).

\begin{theorem}\mythmname{\cite[\S\S5.1.2, 5.2.1]{Ne:67}, \cite[Theorem 4.24]{Mc:00}}
\label{thm:Necas}
If $u\in H^1(D)$ and $\Delta u \in L^2(D)$, then $\partial_n u \in L^2(\Gamma)$ iff $\gamma u \in H^1(\Gamma)$.
\end{theorem}

Given $\bx\in \Gamma$, let $\Theta(\bx)$ be the non-tangential approach set to $\bx$ from $D$ defined, for some sufficiently large $C>1$, as in \eqref{eq:ntapproach2}. Given $u\in C^2(D)$ with $\Delta u=0$, let the non-tangential maximal function of $u$,  $u^*$, be defined by \eqref{eq:ntmax}, and let the non-tangential limit of $u$, $\widetilde{\gamma}u$, be defined by \eqref{eq:ntlim}.

\begin{theorem}\mythmname{\cite[Corollaries 5.5 and 5.7]{JeKe:95}}\label{thm:JeKe:95}
Let $u \in C^2(D)$ with $\Delta u=0$.

(i) $u^*\in L^2(\partial D)$ iff $u\in H^{1/2}(D)$.

(ii) $u^*\in L^2(\partial D)$ implies that $\widetilde{\gamma} u \in L^2(\partial D)$.

(iii) $(\nabla u)^*\in (L^2(\partial D))^d$ iff $u\in H^{3/2}(D)$.

(iv) $(\nabla u)^*\in (L^2(\partial D))^d$ implies that $\widetilde{\gamma} (\nabla u) \in (L^2(\partial D))^d$.
\end{theorem}

\ble\mythmname{\cite[Lemma A.9]{ChGrLaSp:12}}\label{lem:A9}
If $u\in H^s(D)$ with $s>1/2$ and $\Delta u=0$, then $\widetilde{\gamma}u = \gamma u $.
\ele

\ble
\label{lem:A10}
Let $u \in C^2(D)\cap H^1(D)$ with $\Delta u=0$.
Then $u\in H^{3/2}(D)$ iff $\partial_n u\in L^2(\partial D)$ and $\gamma u \in H^1(\partial D)$. Furthermore, if $u\in H^{3/2}(D)$ then, almost everywhere on $\partial D$,
\beq\label{eq:gradutrace}
\widetilde{\gamma}(\nabla u) = \bn \,\partial_n u + \nabla_{\partial D}(\gamma u).
\eeq
\ele

\bpf
The forward implication and the trace result \eqref{eq:gradutrace} are proved in \cite[Lemma A.10]{ChGrLaSp:12}.
For the reverse implication, assume that $\partial_n u \in L^2(\partial D)$ and $\gamma u \in H^1(\partial D)$. Let $v:= \cS\phi$ where the single-layer potential $\cS$ is defined by \eqref{eq:SLP} (with $\Gamma$ replaced by $\partial D$) with $a\neq {\rm Cap}_{\partial D}$ when $d=2$. Let $\phi:= S^{-1}\gamma u$, so that $\phi\in L^2(\partial D)$ by Lemma \ref{lem:SLP}. By the first jump relation in \eqref{eq:jump1} $\gamma v= \gamma u$, so that $v=u$ by uniqueness of the interior Dirichlet problem of Definition \ref{def:idp}.
Since $\phi\in L^2(\partial D)$, $(\nabla u)^*\in (L^2(\partial D))^d$ by \cite[Theorem 1.6]{Ve:84} (see also \cite[Chapter 15, Theorem 5]{MeCo:00}), and thus
$u\in H^{3/2}(D)$ by Part (iii) of Theorem \ref{thm:JeKe:95}.
\epf

\begin{corollary}\label{cor:CoDa}
The space $V(D)$ defined by \eqref{eq:V} is equal to $\{ v : v \in H^{3/2}(D), \, \Delta v \in L^2(D)\}$.
\end{corollary}

\bpf
With $\Phi$ defined by \eqref{eq:fund} and $f\in L^2(D)$, let
\beqs
\cN f(\bx) := \int_D \Phi(\bx,\by) \, f(\by)\, \rd \by, \quad\tfor \bx \in D;
\eeqs
i.e., $\cN$ is a Newtonian potential. Recall that $\cN: L^2(D)\rightarrow H^2(D)$ by, e.g., \cite[Theorem 6.1]{Mc:00}, and $\Delta( \cN f)=-f$.

Given $v \in V(D)$, observe that $\cN(\Delta v) \in H^2(D)$, $\partial_n \cN(\Delta v) = \bn \cdot \gamma(\nabla \cN(\Delta v))\in L^2(\partial D)$, and $\gamma \cN(\Delta v) \in H^1(\partial D)$; to see this last point observe that Theorem \ref{thm:Necas} implies that
\beqs
\gamma: \big\{ u \, :\, u\in H^1(D),\Delta u \in L^2(D), \partial_n u \in L^2(\partial D) \big\} \rightarrow H^1(\partial D),
\eeqs
and thus, in particular, $\gamma : H^2(D) \rightarrow H^1(\partial D)$. Therefore $\cN(\Delta v)\in V(D)$.
Now let $\widetilde{v}: = v + \cN(\Delta v)$. Then $\widetilde{v}\in V(D)$ with $\Delta \widetilde{v}=0$ and hence $\widetilde{v}\in C^2(D)$ by elliptic regularity. Therefore $\widetilde{v}\in H^{3/2}(D)$ by Lemma \ref{lem:A10}. The result that $v\in H^{3/2}(D)$ then follows
since $v= \widetilde{v} - \cN(\Delta v)$ and $\cN(\Delta v)\in H^2(D)$.

The reverse inclusion is proved similarly:~given $v\in H^{3/2}(D)$ with $\Delta v\in L^2(D)$, define $\widetilde{v}$ as before.
Since $\cN(\Delta v)\in H^2(D)$, $\widetilde{v}\in H^{3/2}(D)$. Since  $\Delta \widetilde{v}=0$, $\widetilde{v}\in C^2(D)$ by elliptic regularity, and then  $\partial_n \widetilde{v}\in L^2(\partial D)$ and $\gamma \widetilde{v}\in H^1(\Gamma)$ by Lemma \ref{lem:A10}; thus $\widetilde{v} \in V(D)$. The result that $v\in V(D)$ then follows from the definition of $\widetilde{v}$ and the fact that $\cN(\Delta v)\in H^2(D)\subset V(D)$
\epf

\

We also need the following results in $L^p(\partial D)$ for $p\neq 2$ (as opposed to the $L^2$-based results above).

\begin{theorem}\mythmname{\cite[Theorem 5.6.1]{Me:18}}\label{lem:Medkova}
Suppose $u \in C^2(D)$ with $\Delta u=0$ and $(\nabla u)^*\in (L^p(\partial D))^d$ for some $1<p<\infty$. Then $u^* \in L^p(\partial D)$, $\widetilde{\gamma} u \in L^p(\partial D)$, and
$\widetilde{\gamma} (\nabla u) \in (L^p(\partial D))^d$.
\end{theorem}

\begin{corollary}\label{cor:Lp}
Suppose $u \in C^2(D)$ with $\Delta u=0$ and $(\nabla u)^*\in (L^p(\partial D))^d$ for some $2<p<\infty$. Then $\gamma u \in L^p(\partial D)$, $\partial_n u\in L^p(\partial D)$, and $\nabla_{\partial D}(\gamma u) \in (L^p(\partial D))^d$.
\end{corollary}

\bpf
Since $(\nabla u)^*\in (L^2(\partial D))^d$, $u\in H^{3/2}(D)$ by Part (iii) of Theorem \ref{thm:JeKe:95}. Further, $\widetilde{\gamma} u \in L^p(\partial D)$ and
$\widetilde{\gamma} (\nabla u) \in (L^p(\partial D))^d$ by Theorem \ref{lem:Medkova}. Also, $\gamma u =\widetilde{\gamma} u$ by Lemma \ref{lem:A9}, and thus $\gamma u \in L^p(\partial D)$. By Lemma \ref{lem:A10}, \eqref{eq:gradutrace} holds, and thus $\partial_n u = \bn \cdot \widetilde{\gamma} (\nabla u)\in L^p(\partial D)$ and
$\nabla_{\partial D}(\gamma u) \in (L^p(\partial D))^d$.
\epf

\section{Proofs of Theorems \ref{thm:BVPequiv} and \ref{thm:BVPequivH}}\label{app:BVPequiv}

\bpf[Proof of Theorem \ref{thm:BVPequiv}]
We prove the result for the IDP when $d=3$; the proof for the EDP is very similar.
The proof for the IDP when $d=2$ is also similar, with use of \cite[Theorem 5.15.2]{Me:18} replaced by use of \cite[Theorem 5.15.3]{Me:18}.

If $u$ is the solution of the IDP in the sense of Definition \ref{def:idp} then, since $u\in H^1(\Oi)$, $\widetilde{\gamma}^- u = \gmu$ by Lemma \ref{lem:A9} and thus $\widetilde{\gamma}^- u =g_D$. Furthermore, since $u\in H^{1/2}(\Oi)$, $u^* \in\LtG$ by Part (i) of Theorem \ref{thm:JeKe:95}. Finally, by elliptic regularity $u\in C^2(D)$. Therefore $u$ is a solution of the IDP in the sense of Definition \ref{def:idp2}

To prove the converse,
let  $v := (-\cD+\cS)\phi$ for $\phi\in \LtG$, with $\cD$ the double-layer potential defined by \eqref{eq:DLP} and $\cS$ the single-layer potential defined by \eqref{eq:SLP}.
Now $\widetilde{\gamma}^+ v = (\half I - D + S)\phi$, where we have used that (i) $\widetilde{\gamma}^- \cD\phi = (-\half I + D)\phi$
by \cite[Theorem 1.10]{Ve:84} (similarly to \eqref{eq:jumpCald}) and (ii) $\widetilde{\gamma}^+ \cS\phi = \gamma^+ \cS\phi=S\phi$ by Lemma \ref{lem:A9} and the first jump relation in \eqref{eq:jump1}.
Since $(\half I - D+S):\LtG\rightarrow \LtG$ is invertible by  \cite[Theorem 5.15.2]{Me:18} \footnote{Note that the operator $K$ in \cite{Me:18} equals minus our $D$ (see \cite[\S5.3]{Me:18}).},
if $\phi := (\half I - D+S)^{-1}g_D$, then $v$ is a solution to the IDP of Definition \ref{def:idp2}.
The solution of this BVP is unique by \cite[Page 41]{Ca:85} and \cite[Lemma 3.7]{Ve:84} \footnote{ \cite[Lemma 3.7]{Ve:84} justifies how the uniqueness argument of \cite[Theorem 2.3]{FaJoRi:78} for $C^1$ domains also holds for Lipschitz domains.},
 and thus $v=u$.
 Arguing as in the proof of \cite[Theorem 2.27]{ChGrLaSp:12}, one can show that $(-\half I + D+S):\HhG\rightarrow \HhG$ is invertible -- this follows by proving that $(-\half I + D'+S):\HmhG\rightarrow \HmhG$ is invertible, which in turn follows since $-\half I + D':\HmhG\rightarrow \HmhG$ is Fredholm of index zero by \cite[Theorem 2.25]{ChGrLaSp:12}, $S:\HmhG\rightarrow \HmhG$ is compact by \eqref{eq:map1},
 and $-\half I + D'+S$ is injective by uniqueness of the Laplace exterior Robin problem. Therefore, if $g_D\in \HhG$, then $\phi \in \HhG$, and $u\in H^1(\Oi)$ by the mapping properties of $\cD$ and $\cS$ in \eqref{eq:LPmap}. Finally, $\gmu = \widetilde{\gamma}^-u$ by Lemma \ref{lem:A9}, and thus $\gamma u = g_D$ and $u$ is the solution of the IDP in the sense of Definition \ref{def:idp}.
\epf

\

\bpf[Proof of Theorem \ref{thm:BVPequivH}]
The proof is similar to the proof of Theorem \ref{thm:BVPequiv}. The main difference is that now we define $v:= (\cD_k - \ri k \cS_k)\phi$, where the Helmholtz single- and double-layer potentials $\cS_k$ and $\cD_k$ are defined by \eqref{eq:SLP} and \eqref{eq:DLP} with $\Phi$ replaced by $\Phi_k$. Now $\widetilde{\gamma}^+ v = (\half I + D_k - \ri k S_k)\phi$, where we have used that (i) $\widetilde{\gamma}^+ \cD\phi = (\half I + D_k)\phi$ by \cite[\S4]{ToWe:93}, \cite[Page 111]{ChGrLaSp:12}, and (ii) $\widetilde{\gamma}^+ \cS\phi = \gamma^+ \cS\phi =S_k \phi$ by Lemma \ref{lem:A9} and the Helmholtz analogue of the first jump relation in \eqref{eq:jump1}.
The proof then follows the same steps as in the proof of Theorem \ref{thm:BVPequiv}, using uniqueness of the Helmholtz EDP in the sense of Definition \ref{def:Hedp2} (see \cite[Theorem 5.6, Part (ii)]{ToWe:93}) and the fact that $A_{k,k}:= \half I + D_k - \ri k S_k :H^{s+1/2}(\Gamma)\rightarrow H^{s+1/2}(\Gamma)$ is bounded and invertible for $|s|\leq 1/2$ by \cite[Theorem 2.27]{ChGrLaSp:12}.
\epf

\section{The vector field $\bZ$}\label{app:Lipschitz}

This appendix contains (i) the construction of a vector field $\bZ$ satisfying \eqref{eq:c} (\S\ref{sec:Z}) and (ii) the proof of the extension theorem Lemma \ref{lem:Lipschitz} (\S\ref{sec:Z2}).

\subsection{Constructing the vector field $\vfd$}\label{sec:Z}

\ble\label{lem:vf}
If $\Oi$ is a bounded Lipschitz open set then there exists a real-valued $\vfd\in (C_{\rm{comp}}^\infty(\Rea^d))^d$ and $c>0$ such that $\vfb:= \vfd|_\Gamma$ satisfies
\beq\label{eq:c2}
\vfb(\bx) \cdot \bn(\bx) \geq c \quad \text{ for almost every } \bx\in \Gamma.
\eeq
\ele
The construction of such a $\vfd$ can be found in, e.g., \cite[Lemma 1.5.1.9]{Gr:85}, \cite[Proof of Lemma 1.3]{Ne:67}.

We now spell out the construction from \cite{Gr:85} (in slightly modified form), calculate bounds on the $\vfd$ that we construct, and then show how both the constant $\alpha$ (required to satisfy \eqref{eq:alpha_bound_boundary}/\eqref{eq:alpha_bound}) and $\|\opLIZ\|_{\LtGt}$ depend on constants in the construction of $\vfd$ (see \eqref{eq_alphadef} and \eqref{eq:Abound} below).

Suppose that $\bx_m\in \R^d$ and $a_m>0$, for $m=1, \ldots , M$, are chosen so that
\begin{equation} \label{eq:cover}
\Gamma \subset \bigcup_{m=1}^M B_{a_m}(\bx_m),
\end{equation}
and so that, for each $m=1,...,M$, $\Gamma_m:= \Gamma\cap B_{a_m}(\bx_m)$ is, in some rotated coordinate system $0\tilde x_1...\tilde x_d$, the graph of a Lipschitz continuous function with Lipschitz constant $L_m$ such that the unit vector $\vfd_m$ in the $\tilde x_d$ direction points out of $\Omega$. It is clear from the definition of a Lipschitz open set (see, e.g., \cite[Definition 3.28]{Mc:00}) that it is always possible to find such $\bx_m$ and $a_m$. Note that, for almost all $\bx\in \Gamma_m$,
\begin{equation} \label{eq:Zn}
\bn(\bx)\cdot \vfd_m \geq c_m := (1+L_m^2)^{-1/2}.
\end{equation}

Suppose that $\theta_m\in C^{0,1}(\R^d)$  with $\mbox{supp}(\theta_m)\subset B_{a_m}(\bx_m)$, for $m=1,...,M$, and that
\begin{equation} \label{eq:theta}
\sum_{m=1}^M \theta_m(\bx) \geq 1, \quad \bx\in \Gamma.
\end{equation}
Functions $\theta_1,...,\theta_M$ satisfying \eqref{eq:theta} exist by, e.g., \cite[Theorem 2.17]{Gr:09}, and indeed can be chosen so that equality rather than inequality holds in \eqref{eq:theta} and so that each $\theta_m\in C_{\rm comp}^\infty(\R^d)$, in which case we say that $(\theta_1,...,\theta_M)$ is a partition of unity for $\Gamma$ subordinate to the cover $(B_{a_1}(\bx_1),...,B_{a_M}(\bx_m))$. We construct explicitly functions satisfying \eqref{eq:theta} below; a choice of $\vfd$ that satisfies \eqref{eq:c2} is then
\begin{equation} \label{eq:Zdef}
\vfd(\bx) := \sum_{m=1}^M \theta_m(\bx) \vfd_m,
\end{equation}
since, for almost all $\bx\in \Gamma$,
\beq\label{eq:Zc}
\bn(\bx) \cdot \vfd(\bx) = \sum_{m=1}^M \theta_m(\bx) \bn(\bx) \cdot \vfd_m \geq \sum_{m=1}^M c_m\theta_m(\bx) \geq c := \min_{m=1,...,M} c_m.
\eeq
Observe that, since $\vfd_m$, $m=1, \ldots, M$, are unit vectors, for $\bx\in\Rea^d$,
$$
|\vfd(\bx)| \leq \Theta(\bx) := \sum_{m=1}^M \theta_m(\bx),
\quad\text{ and also } \quad
\nabla \cdot \vfd(\bx) = \sum_{m=1}^M \nabla \theta_m(\bx) \cdot \vfd_m,
$$
so that
\beq\label{eq:noise}
|\nabla \cdot \vfd(\bx)| \leq \Theta^\prime(\bx) := \sum_{m=1}^M |\nabla \theta_m(\bx)|.
\eeq
Furthermore, for unit vectors $\ba$ and ${\bf b}$, writing $\vfd_m$ in terms of its components as $\vfd_m=(Z_{m,1},...,Z_{m,d})$,
$$
\partial_j Z_i a_ib_j = \sum_{m=1}^M \partial_j \theta_m Z_{m,i} a_ib_j = \sum_{m=1}^M (\nabla \theta_m \cdot {\bf b}) \; (\vfd_m \cdot \ba),
\quad\text{so that}\quad
|\partial_j Z_i a_ib_j |\leq \Theta^\prime.
$$
Therefore, with $D \vfd$ the matrix with $(i,j)$th element $\partial_i Z_j$ and $\|\cdot\|_2$ denoting the matrix $2$-norm,
$
\|D \vfd\|_2 \leq \Theta^\prime.
$
Combining this with \eqref{eq:noise}, we see that
the inequality for $\alpha$ \eqref{eq:alpha_bound}
holds if
\begin{equation} \label{eq_alpha_bound2}
2\alpha \geq 3 \|\Theta^\prime\|_{L^\infty(\R^d)}.
\end{equation}

We now choose a specific form for the partition-of-unity functions $\theta_m$, and hence obtain a bound on
$\|\Theta^\prime\|_{L^\infty(\R^d)}.$
The right hand side of \eqref{eq:cover} is still a cover for $\Gamma$ if we reduce the size of the balls slightly, i.e., \eqref{eq:cover} implies that, for some  $\mu \in (0,1)$,
\begin{equation} \label{eq:cover1}
\Gamma \subset \bigcup_{m=1}^M B_{\mu a_m}(\bx_m).
\end{equation}
Assuming that \eqref{eq:cover1} holds for some $\mu\in (0,1)$, one way of constructing the functions $\theta_m$ to satisfy \eqref{eq:theta} is to set
\begin{equation} \label{eq:thetdef}
\theta_m(\bx) := \chi\left(\frac{|\bx-\bx_m|}{a_m}\right)
\end{equation}
with $\chi\in C^{0,1}[0,\infty)$ chosen so that $0\leq \chi(t)\leq 1$, for $t\geq 0$, $\chi(t)=1$, for $0\leq t\leq \mu$, while $\chi(t)=0$ for $t\geq 1$.
(Observe that, with these choices of $\theta_m$, $\vfd$ is no longer smooth, but only Lipschitz.)
Then, for $\bx\in \R^d$,
\begin{equation} \label{eq:zbound}
|\vfd(\bx)| \leq M^*,
\end{equation}
where $M^*\in \mathbb{N}$ is the smallest integer such that every $\bx\in \Gamma$ is in at most $M^*$ balls $B_{a_m}(\bx_m)$. Furthermore, for $\bx\in \R^d$,
$$
\nabla \theta_m(\bx) = \chi^\prime\left(\frac{|\bx-\bx_m|}{a_m}\right) \, \frac{\bx-\bx_m}{a_m|\bx-\bx_m|},
$$
so that, if
$$
a := \min_{m=1,...,M} a_m,
$$
then, by \eqref{eq:noise},
\beq\label{eq:Thetaprimebound}
\|\Theta^\prime\|_{L^\infty(\R^d)} \leq \frac{M^* \|\chi^\prime\|_\infty}{a}.
\eeq
 Thus \eqref{eq_alpha_bound2} holds if
\begin{equation} \label{eq_alpha_bound3}
2\alpha \geq  \frac{3M^*\|\chi^\prime\|_\infty}{a}, \quad \mbox{i.e., \,if} \quad 2\alpha \geq \frac{3M^*}{(1-\mu)a}
\end{equation}
if we make the simple choice that
\begin{equation} \label{eq:chidef}
\chi(t) = (1-t)/(1-\mu) \quad \mbox{for} \quad  \mu\leq t\leq 1.
\end{equation}
With $\opLIZ$ defined by \eqref{eq:opLIZ}, assume that
\begin{equation} \label{eq_alphadef}
\alpha :=  \frac{3M^*\|\chi^\prime\|_\infty}{2a}.
\end{equation}
Then, by Theorem \ref{thm:Laplace_int} combined with \eqref{eq_alpha_bound2} and \eqref{eq:Thetaprimebound}, $\opLIZ$ is coercive with coercivity constant $c$ given by \eqref{eq:Zc}. Furthermore, by \eqref{eq:zbound},
$\opLIZ$ is bounded with
\begin{equation} \label{eq:Abound}
\N{\opLIZ}_{\LtGt} \leq M^* \left(\frac{1}{2} + \N{D'}_{\LtGt} + \N{\nabla_\Gamma S}_{\LtGt}\right) +  \frac{3\|\chi^\prime\|_\infty}{2a} \N{S}_{\LtGt}.
\end{equation}

\subsection{Proof of Lemma \ref{lem:Lipschitz}}\label{sec:Z2}

\bpf
By the Kirszbraun theorem \cite{Ki:34, Va:45},
$\vfb\in (C^{0,1}(\Gamma))^d$ can be extended to a function $\vfext\in (C^{0,1}(\R^d))^d$ with the same (non-zero) Lipschitz constant.

Let $R>0$ be such that $\overline{\Oi}\subset B_R$. For $a>1$ and $R^*>aR$, let
\beqs
\chi(r):=
\begin{cases}
1, & R<r <aR,\\
(R^*-r)/(R^*-aR), & aR<r<R^*,\\
0, & r> R^*,
\end{cases}
\eeqs
and let
\beqs
\vfextt (\bx):=
\begin{cases}
\vfext(\bx), & |\bx|\leq R,\\
\vfext(R\widehat{\bx})\,\chi(|\bx|), & |\bx|> R,
\end{cases}
\eeqs
where $\widehat \bx:= \bx/|\bx|$, so that $\supp\, \vfextt = B_{R^*}$.

Let $L$ be the Lipschitz constant of $\vfext$.  We now prove that if $a$ and $R^*$ are both large enough then
\beq\label{eq:Lip}
\big|\vfextt(\bx) - \vfextt(\by)\big|\leq L |\bx-\by| \quad\tfa \bx,\by \in \Rea^d,
\eeq
i.e.~the Lipschitz constant of $\vfextt$ does not exceed that of $\vfext$.

First observe that there exists $a_0>1$ such that if $a\geq a_0$ and at least one of $|\bx|$ and $|\by|$ are $\geq aR$ then
\beq\label{eq:largea}
|R\widehat{\bx}-R\widehat{\by}| \leq \half |\bx-\by|.
\eeq
We now prove that if $a\geq a_0>1$ and
\beq\label{eq:R*}
R^*\geq aR + 2 \N{\vfext}_{L^\infty(B_R)}/L.
\eeq
then \eqref{eq:Lip} holds.

\paragraph{Case 1: $|\bx|\leq R, |\by|\leq R$.} Since $\vfextt=\vfext$ in $B_R$, \eqref{eq:Lip} follows from the definition of $L$.

\paragraph{Case 2: $|\bx|\leq R, R\leq |\by|\leq aR$.} The key point here is that $|\bx-\by|\geq |\bx- R\widehat{\by}|$ so that
\beqs
\big|\vfextt(\bx) - \vfextt(\by)\big|= \big|\vfext(\bx) - \vfext(R\widehat{\by})\big|\leq L |\bx- R\widehat{\by}|\leq L |\bx-\by|.
\eeqs

\paragraph{Case 3: $|\bx|\leq R, aR\leq |\by|\leq R^*$.}
Now
\begin{align*}
\big|\vfextt(\bx) - \vfextt(\by)\big|&= \big|\vfext(\bx) - \vfext(R\widehat{\by})\chi(|\by|)\big|\\
&\leq \big|\vfext(\bx) - \vfext(R\widehat{\by})\big| +\big|\vfext(R\widehat{\by})\big|(1- \chi(\by))\\
&\leq L\big|\bx-R\widehat{\by}\big| +\N{\vfext}_{L^\infty(B_R)}\frac{ |\by|-aR}{R^*-aR}.
\end{align*}
If $a\geq a_0$, then
\beqs
\big|\bx-R\widehat{\by}\big|\leq\big|R\widehat{\bx}-R\widehat{\by}\big| \leq \half |\bx-\by\big|,
\eeqs
by \eqref{eq:largea}. Since $|\by|-aR \leq |\bx-\by|$,  \eqref{eq:Lip} then follows if, with this choice of $a$,  $R^*$ is then large enough so that
\beq\label{eq:Rlarge}
\frac{\N{\vfext}_{L^\infty(B_R)}}{R^*-aR}\leq \frac{L}{2},
\eeq
which is equivalent to the inequality \eqref{eq:R*}.

\paragraph{Case 4: $|\bx|\leq R, |\by|\geq R^*$.}
Now
\beqs
\big|\vfextt(\bx)- \vfextt(\by)\big|= \big|\vfext(\bx)\big| \leq \N{\vfext}_{L^\infty(B_R)}.
\eeqs
Since $L(R^*-R)\leq L|\bx-\by|$, \eqref{eq:Lip} follows if
\beq\label{eq:R*1}
\N{\vfext}_{L^\infty(B_R)}\leq L(R^*-R)\quad\text{ i.e.,} \quad
R^* \geq R + \N{\vfext}_{L^\infty(B_R)}/L,
\eeq
which is ensured by \eqref{eq:R*} since $a\geq a_0>1$.

\paragraph{Case 5: $R\leq |\bx|\leq aR, R\leq |\by|\leq aR$.}

Similar to Case 2, the key point here is that $|\bx-\by|\geq |R\widehat{\bx}- R\widehat{\by}|$ so that
\beqs
\big|\vfextt(\bx) - \vfextt(\by)\big|= \big|\vfext(R\widehat{\bx}) - \vfext(R\widehat{\by})\big|\leq L |R\widehat{\bx}- R\widehat{\by}|\leq L |\bx-\by|.
\eeqs

\paragraph{Case 6: $R\leq |\bx|\leq aR, aR\leq |\by|\leq R^*$.}
Now
\begin{align*}
\big|\vfextt(\bx) - \vfextt(\by)\big|&= \big|\vfext(R\widehat{\bx}) - \vfext(R\widehat{\by})\chi(|\by|)\big|\\
&\leq \big|\vfext(R\widehat{\bx}) - \vfext(R\widehat{\by})\big| +\big|\vfext(R\widehat{\by})\big|(1- \chi(\by))\\
&\leq L\big|R\widehat{\bx}-R\widehat{\by}\big| +\N{\vfext}_{L^\infty(B_R)}\frac{|\by|-aR}{R^*-aR}.
\end{align*}
Combining \eqref{eq:largea} and the inequality $|\by|-aR\leq |\bx-\by|$, we see that \eqref{eq:Lip} holds if $R^*$ is chosen so that \eqref{eq:Rlarge} holds, which is equivalent to the inequality \eqref{eq:R*}.

\paragraph{Case 7: $R\leq |\bx|\leq aR,|\by|\geq R^*$.}
Similarly  to Case 4,
\beqs
\big|\vfextt(\bx)- \vfextt(\by)\big|= \big|\vfext(R\widehat{\bx})\big| \leq \N{\vfext}_{L^\infty(B_R)}.
\eeqs
Since $L(R^*-aR)\leq L|\bx-\by|$, \eqref{eq:Lip} follows if
\beqs
\N{\vfext}_{L^\infty(B_R)}\leq L(R^*-aR),
\eeqs
which is ensured by \eqref{eq:R*}.

\paragraph{Case 8: $aR\leq |\bx|\leq R^*, aR\leq |\by|\leq R^*$.}
Now
\begin{align*}
\big|\vfextt(\bx) - \vfextt(\by)\big|&= \big|\vfext(R\widehat{\bx})\chi(|\bx|) - \vfext(R\widehat{\by})\chi(|\by|)\big|\\
&\leq \big|\vfext(R\widehat{\bx}) - \vfext(R\widehat{\by})\big||\chi(|\bx|) | +\big|\vfext(R\widehat{\by})\big|\big|\chi(|\bx|)- \chi(|\by|)\big|\\
&\leq L\big|R\widehat{\bx}-R\widehat{\by}\big| +\N{\vfext}_{L^\infty(B_R)}\frac{|\bx-\by|}{R^*-aR}.
\end{align*}
Using \eqref{eq:largea}, we see that \eqref{eq:Lip} holds if $R^*$ is chosen so that \eqref{eq:Rlarge} holds, which is equivalent to the inequality \eqref{eq:R*}.

\paragraph{Case 9: $aR\leq |\bx|\leq R^*, |\by|\geq R^*$.} Now
\beqs
\big|\vfextt(\bx)- \vfextt(\by)\big|= \big|\vfextt(\bx)\big|=\big|\vfext(R\widehat{\bx})\big|\frac{R^*-|\bx|}{R^*-aR}\leq \N{\vfext}_{L^\infty(B_R)}\frac{R^*-|\bx|}{R^*-aR}.
\eeqs
Since $(R^*-|\bx|)\leq |\bx-\by|$, it is sufficient to prove that the right-hand side of this last expression is $\leq L(R^*-|\bx|)$,
and this is ensured by  the inequality \eqref{eq:R*}.

\paragraph{Case 10: $|\bx|\geq R^*, |\by|\geq R^*$.} In this case, \eqref{eq:Lip} holds since $\vfextt(\bx)=\vfextt(\by)=\bzero$.
\epf

\end{appendix}

\section*{Acknowledgements}

EAS was supported by the UK Engineering and Physical Sciences Research Council grants EP/F06795X/1, EP/1025995/1, and EP/R005591/1, and SCW by grant EP/F067798/1.

\footnotesize{
\bibliographystyle{plain}
\bibliography{../biblio_combined_sncwadditions}
}
\end{document}